\documentclass[twoside]{scrartcl}
\usepackage{scrlayer-scrpage}

\usepackage[utf8]{inputenc}
\usepackage[T1]{fontenc}
\usepackage[english]{babel}
\usepackage[unicode=true,plainpages=false]{hyperref}

\usepackage[left=2.5cm, right=3cm, top=2cm, bottom=3.5cm]{geometry}

\usepackage[shortlabels]{enumitem}

\usepackage{amsmath}
\usepackage{amssymb}
\usepackage{amsthm}
\usepackage{mathrsfs}
\usepackage{yhmath}
\usepackage{wasysym}
\usepackage{mathtools}

\usepackage[capitalise]{cleveref}

\usepackage{caption}
\usepackage{subcaption}

\usepackage{graphicx}
\usepackage{pgf,tikz}
\usepackage{tikz-cd}
\usepackage{rotating}

\usepackage{lmodern} \usepackage{tabularx} \usepackage[draft=false,babel,kerning=true,spacing=true]{microtype}

\allowdisplaybreaks

\let\cprod\times

\makeatletter
\newcommand*{\relrelbarsep}{.386ex}
\newcommand*{\relrelbar}{  \mathrel{    \mathpalette\@relrelbar\relrelbarsep
  }}
\newcommand*{\@relrelbar}[2]{  \raise#2\hbox to 0pt{$\m@th#1\relbar$\hss}  \lower#2\hbox{$\m@th#1\relbar$}}
\providecommand*{\rightrightarrowsfill@}{  \arrowfill@\relrelbar\relrelbar\rightrightarrows
}
\providecommand*{\leftleftarrowsfill@}{  \arrowfill@\leftleftarrows\relrelbar\relrelbar
}
\providecommand*{\xrightrightarrows}[2][]{  \ext@arrow 0359\rightrightarrowsfill@{#1}{#2}}
\providecommand*{\xleftleftarrows}[2][]{  \ext@arrow 3095\leftleftarrowsfill@{#1}{#2}}
\makeatother

\newcommand{\isoarrow}[1][]{\arrow[#1,"\rotatebox{90}{\(\sim\)}"]}

\renewcommand{\times}{\cdot}
\DeclarePairedDelimiter{\abs}{\lvert}{\rvert}

 \newcommand{\noloc}{\nobreak\mskip6mu plus1mu\mathpunct{}\nonscript\mkern-\thinmuskip{:}\mskip2mu\relax}

\newcommand{\isom}{\cong}

\newcommand{\from}{\mathbin{\leftarrow}}
\newcommand{\xto}[1]{\mathbin{\xrightarrow{#1}}} \newcommand{\xfrom}[1]{\mathbin{\xleftarrow{#1}}}  \newcommand{\isoto}{\xto\sim}
\newcommand{\isofrom}{\xfrom\sim}

\DeclareMathOperator{\Hom}{Hom}
\DeclareMathOperator{\Ext}{Ext}
\DeclareMathOperator{\Fun}{Fun}

\newcommand{\bigunion}{\bigcup}

\newcommand{\bigdunion}{\bigsqcup}

\renewcommand{\subset}{\subseteq}

\newcommand{\comp}{\mathbin{\circ}}
\DeclareMathOperator{\id}{id}
\DeclareMathOperator{\ev}{ev}

\newcommand{\restrict}[2]{{#1}|_{#2}}

\newcommand{\injto}{\mathrel{\hookrightarrow}}
\newcommand{\surjto}{\mathrel{\twoheadrightarrow}}

\newcommand{\bigdsum}{\bigoplus}
\newcommand{\cplbigdsum}[1]{\widehat\bigdsum_{#1} \,}
\newcommand{\cd}{\mathrm{cd}}

\newcommand{\Z}{{\mathbb{Z}}}

\newcommand{\Fld}{\mathbb{F}}
\newcommand{\Q}{\mathbb{Q}}

\newcommand{\Cpx}{\mathbb{C}}

\newcommand{\tensor}{\otimes}

\DeclareMathOperator{\cts}{\mathcal C}

\DeclareMathOperator{\Spec}{Spec}

\DeclareMathOperator{\IHom}{\underline{\Hom}}

\DeclareMathOperator{\IExt}{\underline{\Ext}}

   \newcommand{\et}{{\mathrm{et}}}

\newcommand{\proet}{{\mathrm{proet}}}   \newcommand{\vsite}{{\mathrm{v}}}    
\newcommand{\cplt}{\,\widehat{}\,}
  
\DeclareMathOperator{\Spa}{Spa}

\newcommand{\opp}{{\mathrm{op}}}

\newcommand{\solid}{{\scalebox{0.5}{$\square$}}}
\DeclareMathOperator{\Mod}{Mod}

\DeclareMathOperator{\D}{\mathcal D}

\DeclareMathOperator{\Tot}{Tot}

\newcommand{\perf}{{\mathrm{perf}}}

\newcommand{\infcatinf}{\mathcal Cat_\infty}

\DeclareMathOperator{\Ring}{Ring}

\newcommand{\oc}{{\mathrm{oc}}}

\DeclareMathOperator{\cofib}{cofib}

\DeclareMathOperator{\Ind}{Ind} \DeclareMathOperator{\Pro}{Pro}

\DeclareMathOperator{\Corr}{Corr}
\DeclareMathOperator{\vStacks}{vStack}

\DeclareMathOperator{\vStacksCoeff}{vStack_\Lambda}

 \newcommand{\dimtrg}{\mathrm{dim.trg}}
 \newcommand{\nuc}{{\mathrm{nuc}}} \newcommand{\lis}{{\mathrm{lis}}}  \newcommand{\tr}{{\mathrm{tr}}}
\newcommand{\dlb}{{\mathrm{dlb}}}
\DeclareMathOperator{\Bun}{Bun}

\theoremstyle{plain}
\newtheorem{theorem}{Theorem}[section]
\newtheorem{theorem*}{Theorem}
\newtheorem{proposition}[theorem]{Proposition}
\newtheorem{proposition*}[theorem*]{Proposition}
\newtheorem{corollary}[theorem]{Corollary}
\newtheorem{lemma}[theorem]{Lemma}

\theoremstyle{definition}
\newtheorem{definition}[theorem]{Definition}
\newtheorem{definition*}[theorem*]{Definition}
\newtheorem{example}[theorem]{Example}
\newtheorem{examples}[theorem]{Examples}
\newtheorem{remark}[theorem]{Remark}
\newtheorem{remarks}[theorem]{Remarks}
\newtheorem{warning}[theorem]{Warning}

\newtheorem{hypothesis*}[theorem*]{Hypothesis}

\numberwithin{equation}{theorem}
\numberwithin{figure}{section}
\numberwithin{table}{section}

\newlist{thmenum}{enumerate}{1}
\setlist[thmenum]{label=(\roman*), ref=\thetheorem.(\roman*)}
\crefalias{thmenumi}{theorem}

\newlist{propenum}{enumerate}{1}
\setlist[propenum]{label=(\roman*), ref=\theproposition.(\roman*)}
\crefalias{propenumi}{proposition}

\newlist{corenum}{enumerate}{1}
\setlist[corenum]{label=(\roman*), ref=\thecorollary.(\roman*)}
\crefalias{corenumi}{corollary}

\newlist{lemenum}{enumerate}{1}
\setlist[lemenum]{label=(\roman*), ref=\thelemma.(\roman*)}
\crefalias{lemenumi}{lemma}

\newlist{exampleenum}{enumerate}{1}
\setlist[exampleenum]{label=(\alph*), ref=\theexamples.(\alph*)}
\crefalias{exampleenumi}{example}

\newlist{remarksenum}{enumerate}{1}
\setlist[remarksenum]{label=(\roman*), ref=\theremarks.(\roman*)}
\crefalias{remarksenumi}{remark}

\newlist{defenum}{enumerate}{1}
\setlist[defenum]{label=(\alph*), ref=\thedefinition.(\alph*)}
\crefalias{defenumi}{definition}

\title{The 6-Functor Formalism for $\Z_\ell$- and $\Q_\ell$-Sheaves on Diamonds}
\author{Lucas Mann}
\date{\today}

\newcommand{\naive}{{\mathrm{naive}}}

\begin{document}

\maketitle

\begin{abstract}
For every nuclear $\Z_\ell$-algebra $\Lambda$ and every small v-stack $X$ we construct an $\infty$-category $\D_\nuc(X,\Lambda)$ of nuclear $\Lambda$-modules on $X$. We then construct a full 6-functor formalism for these sheaves, generalizing the étale 6-functor formalism for $\Lambda = \Fld_\ell$. Prominent choices for $\Lambda$ are $\Z_\ell$, $\Q_\ell$ and $\overline{\Q_\ell}$ and especially in the latter two cases, no satisfying 6-functor formalism has been found before. Applied to classifying stacks we obtain a theory of nuclear representations, i.e. continuous representations on filtered colimits of Banach spaces.
\end{abstract}

\tableofcontents

\section{Introduction}

Fix two primes $\ell \ne p$. In this paper we construct and study a full 6-functor formalism for certain sheaves of modules over nuclear $\Z_\ell$-algebras $\Lambda$ like $\Z_\ell$, $\Fld_\ell$, $\Q_\ell$, $\overline{\Q_\ell}$ or $\Cpx_\ell$. The main difference to previous attempts at such a 6-functor formalism is that we never impose any finiteness conditions or pass to any isogeny categories -- everything is completely natural and embedds fully faithfully into the category of pro-étale sheaves (even on the derived level). For simplicity this whole paper will work with diamonds and small v-stacks over $\Z_p$ \cite{etale-cohomology-of-diamonds}, i.e. in the realm of rigid-analytic geometry. We expect that a similar 6-functor formalism can be constructed on schemes, albeit somewhat harder to implement.

Recall that for torsion coefficients, e.g. $\Lambda = \Fld_\ell$, a full 6-functor formalism for étale sheaves on diamonds has been worked out in \cite{etale-cohomology-of-diamonds} and has been used extensively in \cite{fargues-scholze-geometrization} to gain new insights into the $\ell$-adic geometric Langlands program. The present paper generalizes this étale 6-functor formalism to non-discrete nuclear $\Z_\ell$-algebras like the ones mentioned above. Note that even in the case that $\Lambda = \Z_\ell$ we allow non-complete $\Z_\ell$-sheaves in our 6-functor formalism (unlike the somewhat ad-hoc definition in \cite[\S26]{etale-cohomology-of-diamonds}), which becomes especially important when working with non-qcqs maps, see \cref{ex:intro-rigid-case} below.

Without further ado, let us state the main definitions and results. In the following, we say that a spatial diamond (e.g. qcqs rigid-analytic variety) is \emph{$\ell$-bounded} if it has finite cohomological dimension for étale $\Fld_\ell$-sheaves. We can then define the following category of nuclear sheaves:

\begin{definition}[{cf. \cref{def:nuclear-module-on-spat-diamond}}]
Let $X$ be a small v-stack. A \emph{nuclear $\Z_\ell$-sheaf} on $X$ is a (derived) v-sheaf $\mathcal M \in \D(X_\vsite,\Z_\ell)$ with the following property: For every map $f\colon Y \to X$ from an $\ell$-bounded spatial diamond $Y$, the sheaf $f^* \mathcal M \in \D(Y_\vsite,\Z_\ell)$ is a filtered colimit of $\ell$-adically complete sheaves, all of which are étale modulo $\ell$. We denote by
\begin{align*}
	\D_\nuc(X,\Z_\ell) \subset \D(X_\vsite,\Z_\ell)
\end{align*}
the full subcategory of nuclear $\Z_\ell$-sheaves. For any nuclear $\Z_\ell$-algebra $\Lambda$ we denote by $\D_\nuc(X,\Lambda) \subset \D(X_\vsite,\Lambda)$ the full subcategory of those $\Lambda$-sheaves whose underlying $\Z_\ell$-sheaf is nuclear.
\end{definition}

If $X$ is a geometric point then $\D_\nuc(X,\Z_\ell) = \D_\nuc(\Z_\ell) \subset \D_\solid(\Z_\ell)$ recovers the $\infty$-category of nuclear $\Z_\ell$-modules as defined in \cite[\S8]{condensed-complex-geometry} and \cite[Definition 13.10]{scholze-analytic-spaces}. In particular, a \emph{nuclear $\Z_\ell$-algebra} is a condensed $\Z_\ell$-algebra whose underlying $\Z_\ell$-module is a filtered colimit of Banach $\Z_\ell$-algebras. If $\Lambda$ is discrete then $\D_\nuc(X,\Lambda) = \D_\et(X,\Lambda)$ (see \cref{rslt:nuclear-over-discrete-Lambda-equiv-etale}), so in this case we recover the classical étale theory.

Perhaps somewhat surprisingly, the $\infty$-category of nuclear sheaves satisfies v-descent. This observation lies at the heart of this paper.

\begin{theorem}[{cf. \cref{rslt:v-descent-for-nuclear-sheaves,rslt:v-descent-for-nuclear-Lambda-modules}}]
The $\infty$-category of nuclear $\Lambda$-modules is stable under pullback and satisfies v-descent. More precisely, the assignment
\begin{align*}
	X \mapsto \D_\nuc(X,\Lambda)
\end{align*}
defines a hypercomplete sheaf of $\infty$-categories on the v-site of all small v-stacks.
\end{theorem}

Let us now come to the 6-functor formalism for nuclear $\Lambda$-modules. The six functors can roughly be constructed as follows. Fix a map $f\colon Y \to X$ of small v-stacks; then:
\begin{enumerate}[1.]
	\item The $\infty$-category $\D_\nuc(X,\Lambda)$ comes naturally equipped with a symmetric monoidal structure
	\begin{align*}
		- \tensor -\colon \D_\nuc(X,\Lambda) \cprod \D_\nuc(X,\Lambda) \to \D_\nuc(X,\Lambda),
	\end{align*}
	which agrees with the solid tensor product from \cite[Proposition VII.2.2]{fargues-scholze-geometrization} (note that clearly $\D_\nuc(X,\Lambda) \subset \D_\solid(X,\Lambda)$).

	\item Since $\D_\nuc(X,\Lambda)$ is presentable, the symmetric monoidal structure is closed and hence allows an internal hom
	\begin{align*}
		\IHom(-,-)\colon \D_\nuc(X,\Lambda)^\opp \cprod \D_\nuc(X,\Lambda) \to \D_\nuc(X,\Lambda)
	\end{align*}
	given as the right adjoint of $\tensor$.

	\item The pullback functor
	\begin{align*}
		f^*\colon \D_\nuc(X,\Lambda) \to \D_\nuc(Y,\Lambda)
	\end{align*}
	is simply the pullback of v-sheaves (it preserves nuclear sheaves by definition).

	\item The pushforward functor
	\begin{align*}
		f_*\colon \D_\nuc(Y,\Lambda) \to \D_\nuc(X,\Lambda)
	\end{align*}
	is the right adjoint of the pullback functor. In general it does not agree with the v-pushforward, but if $f$ is qcqs then this is true under mild assumptions (see \cref{sec:pushforward}).

	\item If $f$ is ``$\ell$-fine'' (see \cref{def:ell-fine-maps}) then we can construct a lower shriek functor
	\begin{align*}
		f_!\colon \D_\nuc(Y,\Lambda) \to \D_\nuc(X,\Lambda)
	\end{align*}
	as follows. If $f$ is proper then $f_! = f_*$. If $f$ is étale then $f_!$ is the left adjoint of $f^*$ (this agrees with the functor $f_!$ on v-sheaves, see \cref{rslt:properties-of-etale-lower-shriek}). If $f$ is compactifiable and qcqs then we can define $f_!$ by composing the previous two examples along a relative compactification of $f$. If $f$ is only locally compactifiable then we can filter $Y$ by qcqs open subsets on which $f$ becomes compactifiable and then define $f_!$ as the colimit over this filtration. Finally, one can extend the construction of $f_!$ even to certain ``stacky'' maps. For more details on the construction of $f_!$ see \cref{rmk:explicit-construction-of-shriek-functors}.

	\item If $f$ is $\ell$-fine then we define
	\begin{align*}
		f^!\colon \D_\nuc(X,\Lambda) \to \D_\nuc(Y,\Lambda)
	\end{align*}
	as the right adjoint of $f_!$.
\end{enumerate}

\begin{theorem}[{cf. \cref{rslt:6-functor-formalism}}]
The above six functors $\tensor$, $\IHom$, $f^*$, $f_*$, $f_!$ and $f^!$ provide a 6-functor formalism. In particular, they satisfy the following properties:
\begin{thmenum}
	\item (Functoriality) For composable maps $f$, $g$ of small v-stacks we have natural isomorphisms $(f \comp g)^* = g^* \comp f^*$ and $(f \comp g)_* = f_* \comp g_*$. If $f$ and $g$ are $\ell$-fine then also $(f \comp g)_! = f_! \comp g_!$ and $(f \comp g)^! = g^! \comp f^!$.

	\item (Special Cases) If $j\colon U \to X$ is an étale map of small v-stacks then $j^! = j^*$. If $f\colon Y \to X$ is a proper $\ell$-fine map of small v-stacks then $f_! = f_*$.

	\item (Projection Formula) Let $f\colon Y \to X$ be an $\ell$-fine map of small v-stacks. Then for all $\mathcal M \in \D_\nuc(X,\Lambda)$ and $\mathcal N \in \D_\nuc(Y, \Lambda)$ there is a natural isomorphism
	\begin{align*}
		f_!(\mathcal N \tensor f^* \mathcal M) = (f_! \mathcal N) \tensor \mathcal M.
	\end{align*}

	\item (Proper Base-Change) Let
	\begin{center}\begin{tikzcd}
		Y' \arrow[r,"g'"] \arrow[d,"f'"] & Y \arrow[d,"f"]\\
		X' \arrow[r,"g"] & X
	\end{tikzcd}\end{center}
	be a cartesian diagram of small v-stacks such that $f$ is $\ell$-fine. Then there is a natural equivalence
	\begin{align*}
		g^* f_! = f'_! g'^*
	\end{align*}
	of functors $\D_\nuc(Y, \Lambda) \to \D_\nuc(X', \Lambda)$.
\end{thmenum}
\end{theorem}

Before we continue, let us briefly mention how the just constructed 6-functor formalism behaves in the classical rigid-analytic world.

\begin{example} \label{ex:intro-rigid-case}
Suppose $f\colon Y \to X$ is a map of rigid-analytic varieties over some fixed non-archimedean base field $K$. Then $f$ is $\ell$-fine. The pushforward functor $f_*\colon \D_\nuc(Y,\Lambda) \to \D_\nuc(X,\Lambda)$ computes the nuclearized derived pushforward of pro-étale $\Lambda$-sheaves; in particular if $X = \Spa K$ then $f_*$ computes nuclearized pro-étale cohomology. If $f$ is qcqs then $f_*$ preserves all small colimits, so for example we have $f_* \Q_\ell = (\varprojlim_n f_*(\Z/\ell^n \Z))[1/\ell]$; this is not true if $f$ is not qcqs.

With $f$ as before, let us explain how to compute $f_!\colon \D_\nuc(Y,\Lambda) \to \D_\nuc(X,\Lambda)$. By the projection formula $f_!$ commutes with a change of coefficients, e.g. we have $f_! \Q_\ell = (f_! \Z_\ell)[1/\ell]$. This often reduces the computation of $f_!$ to the case of $\Z_\ell$-sheaves. If $f$ is qcqs then $f_!$ preserves $\ell$-adically complete sheaves and so in this case we have $f_! \Z_\ell = \varprojlim_n f_!(\Z/\ell^n \Z)$, where each $f_!(\Z/\ell^n \Z)$ is the usual lower shriek on étale sheaves. If $f$ is not qcqs then this is not true; instead one can then filter $Y$ by qcqs open subsets and compute $f_!$ as the colimit of the lower shrieks on these subsets. In particular $f_!$ agrees with previous definitions of ``$\Z_\ell$-cohomology with compact support'', cf. \cite[Remark 5.5]{perfectoid-spaces-survey}.
\end{example}

With our full 6-functor formalism at hand, we study so-called \emph{relatively dualizable} sheaves in this 6-functor formalism. This concept is not new and is known as universally locally acyclic sheaves in the case of discrete $\Lambda$. With the power of the magical 2-category found by Lu-Zheng \cite{lu-zhen-ula} we show that relatively dualizable sheaves satisfy all the expected properties and that this is completely formal. As an important special case of this theory we deduce that the following definition of $\ell$-cohomologically smooth maps is sensible:

\begin{definition}[{cf. \cref{def:cohom-smooth-maps}}]
An $\ell$-fine map $f\colon Y \to X$ is \emph{$\ell$-cohomologically smooth} if $f^! \Fld_\ell \in \D_\et(Y,\Fld_\ell)$ is invertible and its formation commutes with base-change along $f$.
\end{definition}

We emphasize that our definition of $\ell$-cohomological smoothness only depends on the étale theory with $\Fld_\ell$-coefficients and is therefore compatible with previously defined notions of $\ell$-cohomological smoothness in \cite{etale-cohomology-of-diamonds,fargues-scholze-geometrization}. This may seem very surprising, in particular because we still obtain the expected Poincaré duality (see below). This fact was independently observed by Bogdan Zavyalov, whose suggestions led to our pursuit of these ideas and which will also be used in Zavyalov's upcoming work on $p$-adic Poincaré duality in rigid-analytic geometry \cite{zavyalov-revisiting-poincare-duality}.

\begin{proposition}[{cf. \cref{rslt:Poincare-duality-for-smooth-maps}}]
Let $f\colon Y \to X$ be an $\ell$-cohomologically smooth map of small v-stacks. Then $f^! \Lambda \in \D_\nuc(Y,\Lambda)$ is invertible and the natural morphism
\begin{align*}
	f^!\Lambda \tensor f^* \isoto f^!
\end{align*}
is an isomorphism of functors $\D_\nuc(X,\Lambda) \to \D_\nuc(Y,\Lambda)$.
\end{proposition}

Of course, $\ell$-cohomologically smooth maps also satisfy all the other expected properties for nuclear sheaves, see \cref{sec:smoothness}; here the magic of Lu-Zheng's construction really shines. We can use similar ideas to get a robust notion of $\ell$-cohomologically proper maps (which are more subtle to define than one might think at first), which we study in \cref{sec:properness}.

In the final part of this paper we apply the 6-functor formalism to classifying stacks, in which case we recover a category of representations. More concretely, suppose that $G$ is a locally profinite group which has locally finite $\ell$-cohomological dimension (e.g. $G$ is locally pro-$p$). Then there is a natural equivalence
\begin{align*}
	\D_\nuc(*/G,\Lambda) = \D_\nuc(\Lambda)^{BG},
\end{align*}
where the right-hand side denotes the $\infty$-category of \emph{nuclear $G$-representations}, i.e. continuous $G$-representations on nuclear $\Lambda$-modules. If $\Lambda$ is concentrated in degree $0$ (which is probably always true in practice) then $\D_\nuc(\Lambda)^{BG}$ is the derived $\infty$-category of its heart. For example, if $\Lambda = \Q_\ell$ then this heart is the abelian category of continuous $G$-representations on filtered colimits of $\Q_\ell$-Banach spaces.

We get the following result on the interaction of the 6-functor formalism with classifying stacks:

\begin{theorem}[{cf. \cref{rslt:classifying-stacks-are-smooth}}]
Let $G$ be a locally pro-$p$ group. Then the map $*/G \to *$ is $\ell$-fine and $\ell$-cohomologically smooth. If $G$ is pro-$p$ then this map is additionally $\ell$-cohomologically proper.
\end{theorem}

Note that the pushforward along $*/G \to *$ computes nuclearized continuous $G$-cohomology, so the above theorem implies a Poincaré duality statement for continuous $G$-cohomology on nuclear $\Lambda$-modules. With this knowledge at hand, one can also introduce a good notion of \emph{admissible nuclear $G$-representations} (see \cref{def:admissible-representation}) and deduce some basic properties for them (see \cref{rslt:properties-of-admissible-reps}). If $\Lambda$ is discrete then this recovers the usual notion of admissible representations. If $\Lambda$ is not discrete then the notion of admissible representations seems less common in the literature and has probably not been defined in the generality presented here.

\paragraph{Background and Motivation.} In recent years, 6-functor formalisms have proven to be an extremely powerful tool in the study of various cohomology theories. With recent advances in the geometric Langlands program, 6-functor formalisms have also found tremendous success in applications to representation theory. In the case of rigid-analytic geometry, there have been introduced two 6-functor formalisms: A 6-functor formalism for étale $\Lambda$-sheaves for any discrete ring $\Lambda$ which is killed by some power of $\ell$ \cite{etale-cohomology-of-diamonds} and an analog for mod-$p$ coefficients \cite{mann-mod-p-6-functors}. These 6-functor formalisms have then been extended to certain stacky maps in \cite{mod-ell-stacky-6-functors,mod-p-stacky-6-functors}, which makes them applicable to representation theory. However, as far as we know there has not yet been introduced a satisfying 6-functor formalism for sheaves of modules over non-discrete rings like $\Z_\ell$, $\Q_\ell$ or $\overline{\Q_\ell}$. There have been the following attempts:
\begin{enumerate}[1.]
	\item The most naive way of defining a 6-functor formalism for $\Z_\ell$-sheaves is by formally taking $\ell$-adic completions, i.e. associating to every small v-stack $X$ the $\infty$-category $\D_\naive(X,\Z_\ell) := \varprojlim_n \D_\et(X,\Z/\ell^n \Z)$ (cf. \cite[\S26]{etale-cohomology-of-diamonds}). This works well for most purposes: The $\infty$-category $\D_\naive(X,\Z_\ell)$ satisfies v-descent and embedds fully faithfully into $\D(X_\vsite,\Z_\ell)$, even into $\D_\nuc(X,\Z_\ell)$. One also gets a full 6-functor formalism for these sheaves. In the case of qcqs spaces this 6-functor formalism mostly recovers the nuclear 6-functor formalism. For non-qcqs spaces and maps, the naive 6-functor formalism forces $\ell$-adic completions and thus loses information compared to the nuclear version. This is especially relevant in rigid-analytic geometry, where one is often interested in cohomologies of non-qcqs spaces like those appearing in the Drinfeld tower.

	The main issue with the naive 6-functor formalism for $\Z_\ell$-sheaves is that it does not produce a good theory for $\Q_\ell$-sheaves. One can attempt to define $\D_\naive(X,\Q_\ell) := \D_\naive(X,\Z_\ell) \tensor \Q$, i.e. by formally inverting $\ell$. This roughly amounts to studying those $\Q_\ell$-sheaves which admit a $\Z_\ell$-model, but that is not precisely correct. In fact, $\D_\naive(X,\Q_\ell)$ does not even embedd into $\D(X_\vsite,\Q_\ell)$!

	\item In \cite[\S VII]{fargues-scholze-geometrization} an $\infty$-category $\D_\solid(X,\Z_\ell) \subset \D(X_\vsite,\Z_\ell)$ of \emph{solid $\Z_\ell$-sheaves} on $X$ is defined. This $\infty$-category satisfies v-descent and one can easily define $\D_\solid(X,\Lambda)$ for every solid $\Z_\ell$-algebra $\Lambda$, so in particular for every nuclear $\Z_\ell$-algebra $\Lambda$. The theory of solid sheaves has excellent geometric properties; in fact it forms the foundation upon which we build our nuclear theory. Fargues-Scholze provide a ``5-functor formalism'' for $\D_\solid(X,\Lambda)$, which in many applications is sufficiently strong to replace an actual 6-functor formalism. Still, $\D_\solid(X,\Lambda)$ does not have a 6-functor formalism (it is far too big for that) and hence does not generalize the étale theory for discrete $\Lambda$. See below for more information on the comparison between the solid 5-functor formalism and our nuclear 6-functor formalism.

	\item In \cite[\S VII.6]{fargues-scholze-geometrization} a full subcategory $\D_\lis(X,\Lambda) \subset \D_\solid(X,\Lambda)$ is defined in order to remedy some of the downsides of $\D_\solid(X,\Lambda)$. While $\D_\lis(X,\Lambda)$ seems to work very well with regards to the Langlands conjecture, it feels rather unnatural from a purely geometric point of view. Some of the main issues are that it is not stable under various operations and does not satisfy v-descent.
\end{enumerate}
The nuclear 6-functor formalism has all the expected formal properties of a 6-functor formalism and generalizes the étale version for discrete coefficients. We therefore believe it to be the ``right'' way of working with sheaves over nuclear $\Z_\ell$-algebras. While the existence of the nuclear 6-functor formalism is a very neat abstract result and has many formal consequences (e.g. Poincaré duality for $\overline{\Q_\ell}$-local systems), we do not provide any actual applications of the 6-functor formalism in our paper. From our point of view, the main scientific value of this paper comes from the following motivations:
\begin{enumerate}[(a)]
	\item We found the definition of $\D_\lis$ rather unsatisfying and were therefore looking for different ways of understanding it. It seems plausible that one can describe $\D_\lis \subset \D_\nuc$ in terms of some abstract properties, like $\Ind$-compact or $\Ind$-perfect objects.

	\item With $\D_\nuc$ appearing so naturally in geometry, we would be surprised if it does not play some role in the local Langlands program. For example, one could hope that the categorical Langlands conjecture generalizes to a conjecture describing $\D_\nuc(\Bun_G,\Lambda)$. So far $\ell$-adic Banach representations have not received much attention in the Langlands program, but in \cite{vigneras-banach-reps} it was shown that they seem to have some connection to Langlands.

	\item We hope that the nuclear 6-functor formalism sheds some light on the $p$-adic Langlands program. In fact, we have previously been trying to generalize our mod-$p$ 6-functor formalism to a $\Z_p$- or $\Q_p$-version without much success. The nuclear $\ell$-adic 6-functor formalism provides a lot of insight on how its $p$-adic analog should work.
	\item A large part of this paper is rather formal. One can therefore use this paper as a general recipe on how to construct a 6-functor formalism and study its basic properties like smoothness and representations.
\end{enumerate}

\paragraph{Why Nuclear Sheaves.} Our definition of nuclear sheaves, while being rather explicit, comes somewhat out of nowhere. Let us therefore explain why one should expect such a definition to appear in a 6-functor formalism for $\Z_\ell$-sheaves. There are two major motivations:
\begin{enumerate}[(i)]
	\item In condensed mathematics, the $\infty$-category $\D_\nuc(\Z_\ell)$ of nuclear $\Z_\ell$-modules should be seen as an analog of ``discrete'' modules over $\Z_\ell$ (since $\Z_\ell$ is equipped with a topology, the actual category of discrete $\Z_\ell$-modules is rather small). Since étale $\Fld_\ell$-sheaves on a small v-stack $X$ are a relative version of discrete $\Fld_\ell$-modules, it is believable that the correct $\Z_\ell$-analog should be a relative version of nuclear $\Z_\ell$-modules.

	\item One of the main reasons why the solid 5-functor formalism cannot be upgraded to a 6-functor formalism is that for proper maps $f\colon Y \to X$ the functor $f_*\colon \D_\solid(Y,\Lambda) \to \D_\solid(X,\Lambda)$ often does not satisfy the projection formula. A geometric example of this phenomenon is given in \cite[Warning VII.2.5]{fargues-scholze-geometrization}. Let us discuss a different example, which also demonstrates where the nuclearity condition comes from: Suppose $X = \Spa C$ is a geometric point and $Y = \underline S$ is the pro-étale space over $X$ given by some profinite set $S$. Then $\D_\solid(X,\Lambda) = \D_\solid(\Lambda)$ is the $\infty$-category of solid condensed $\Lambda$-modules. If $f_*$ satisfied the projection formula then in particular for all solid $\Lambda$-modules $M$ we would have
	\begin{align*}
		f_* f^* M = f_* \Lambda \tensor_\Lambda M,
	\end{align*}
	where the tensor product on the right is the solid tensor product. Note that $f_* \Lambda = \cts(S, \Lambda) = \Lambda_\solid[S]^\vee$ and $f_* f^* M = \IHom(\Lambda_\solid[S], M)$. Thus for fixed $M$ the above condition (for all $S$) amounts precisely to the condition that $M$ is nuclear (as defined in \cite[Definition 13.10]{scholze-analytic-spaces}).
\end{enumerate}

\paragraph{Relation to the Solid 5-Functor Formalism.} In \cite[\S VII]{fargues-scholze-geometrization} Fargues-Scholze define a ``5-functor formalism'' for the $\infty$-category $\D_\solid(X,\Lambda)$ of solid $\Lambda$-sheaves on any small v-stack $X$. Apart from the usual four functors $\tensor$, $\IHom$, $f^*$ and $f_*$ they introduce a functor $f_\natural$ defined to be the left adjoint of $f^*$. The goal of the solid 5-functor formalism is to approximate a 6-functor formalism for $\Lambda$-sheaves, but there has been some confusion as to what extent this is possible. With the nuclear 6-functor formalism at hand we can shed some light on this and in particular try to answer the questions raised in \cite[Remark VII.3.6]{fargues-scholze-geometrization}.

Namely, in \cref{rslt:comparison-of-natural-and-shriek-functor-fdcs-smooth} we show the following: Suppose $f\colon Y \to X$ is an $\ell$-fine and $\ell$-cohomologically smooth map of small v-stacks which is representable in locally spatial diamonds (this latter condition can be relaxed, cf. \cref{rslt:comparison-of-natural-and-shriek-functor-ell-fine-smooth}). Then $f_\natural\colon \D_\solid(Y,\Lambda) \to \D_\solid(X,\Lambda)$ preserves nuclear sheaves and there is a natural isomorphism
\begin{align*}
	f_\natural = f_! (- \tensor f^! \Lambda)
\end{align*}
of functors $\D_\nuc(Y,\Lambda) \to \D_\nuc(X,\Lambda)$. In other words, along $\ell$-cohomologically smooth maps $f$, $f_\natural$ computes $f_!$ up to a twist. In particular this allows one to simulate $f_!$ by $f_\natural$ whenever one is in a geometric situation that can be built out of smooth maps, which explains why the 5-functor formalism is so useful in \cite{fargues-scholze-geometrization}. Note that it also follows that $\D_\lis(X,\Lambda) \subset \D_\nuc(X,\Lambda)$.

Of course, for non-smooth $f$ the functors $f_\natural$ and $f_!$ are very different and $f_\natural$ does usually not preserve nuclear sheaves.

\paragraph{Structure of the Paper.} This paper is roughly structured into four parts:
\begin{itemize}
	\item Part I consists of \cref{sec:w1-solid,sec:nuclear} and introduces the $\infty$-category of nuclear sheaves. These sections are very specific to the concrete setting at hand. In \cref{sec:w1-solid} we introduce $\omega_1$-solid sheaves as a particularly nice full subcategory of all solid sheaves. Their main advantage is that it is easier to control their compact objects (which is crucial for studying nuclear sheaves later on) while still maintaining most of the nice properties of solid sheaves. In \cref{sec:nuclear} we then introduce nuclear sheaves. We provide different characterizations of nuclearity (thereby also motivating the terminology) and show that they satisfy v-descent.

	\item Part II consists of \cref{sec:pushforward,sec:6functors} and constructs the 6-functor formalism for nuclear sheaves. While not being fully formal, most of the ideas in these sections should be applicable to many different geometric settings. It can therefore be used as a general recipe for constructing 6-functor formalisms.

	\item Part III consists of \cref{sec:perf,sec:reldual,sec:smoothness,sec:properness} and is almost completely formal. Here we introduce dualizable and relatively dualizable sheaves (the latter being also known as universally locally acyclic sheaves in the étale context) and use them to define cohomologically smooth and proper maps. The main insight is that by using the magic of Lu-Zheng's ideas one can show that cohomological smoothness is a condition that only depends on very little data of the 6-functor formalism (one only needs to understand its behavior on the monoidal unit). Since all of the sections in part III are very formal, they should apply to every 6-functor formalism.

	\item Part IV consists of \cref{sec:representations}, where we apply the nuclear 6-functor formalism to classifying stacks in order to obtain a 6-functor formalism for representation theory. The main result here is that under mild assumptions on the locally profinite group $G$ the classifying stack $*/G$ is $\ell$-cohomologically smooth.
\end{itemize}

\paragraph{Notation and Conventions.} The whole paper is written in the modern language of $\infty$-categories, as this is the most natural and clean way of describing our ideas (for example $\D_\nuc(X,\Lambda)$ does not have a $t$-structure and is therefore not that easily accessible with conventional methods; also the 6-functor formalism for stacky maps requires an $\infty$-categorical framework). In particular every functor, sheaf, ring, module, representation etc. is always assumed to be derived if not explicitly specified differently. In the presence of a $t$-structure on a stable $\infty$-category we usually denote by $\pi_n M = H^{-n}(M)$ the homology objects of $M$. We say that $M$ is \emph{static} if it is concentrated in degree 0, i.e. has no derived structure. For a pro-étale map $j\colon U \to X$ of diamonds and a pro-étale sheaf $\mathcal M$ on $U$ we denote $\mathcal M[U] := j_! \mathcal M$ (in the sense of sites), which is a pro-étale sheaf on $X$. All sheaves will always be considered as part of the derived $\infty$-category of v-sheaves. For example by an étale sheaf we mean a sheaf in $\D_\et(X,\Z) \subset \D(X_\vsite,\Z)$. We denote by $(-)_\et\colon \D(X_\vsite,\Z) \to \D_\et(X,\Z)$ the right adjoint to the inclusion. For a $\Z_\ell$-module $M$ we denote $M/\ell^n M := \cofib(M \xto{\ell^n} M)$ and similarly for sheaves of $\Z_\ell$-modules. We warn the reader that by $\D_\et(X,\Z_\ell)$ we denote the $\infty$-category of those pro-étale $\Z_\ell$-sheaves on $X$ whose underlying sheaf of abelian groups is étale (this is different from the notation in \cite[\S26]{etale-cohomology-of-diamonds}).

We are aware of the fact that $\infty$-categories can be intimidating to the uninitiated, so we refer the reader to \cite[\S1.5]{mann-mod-p-6-functors} for a quick down-to-earth introduction to $\infty$-categories with a special focus on the terminology used in the algebraic setting.

\paragraph{Acknowledgements.} I heartily thank David Hansen for many very helpful discussions regarding this material and for reading preliminary versions of it. I similarly thank Peter Scholze for useful discussions (and some clarification regarding dualizable objects...) and for checking whether the main results of this paper are plausible. Special thanks go to Johannes Anschütz, Arthur-César Le Bras, Alexander Ivanov and Bogdan Zavyalov who provided valuable input and feedback to some of the ideas appearing in this paper.

\section{\texorpdfstring{$\omega_1$}{w1}-Solid Sheaves} \label{sec:w1-solid}

Fix a prime $\ell \ne p$. The $\infty$-category $\D_\solid(X,\Z_\ell)$ constructed in \cite[\S VII]{fargues-scholze-geometrization} lacks some formal properties which we need for our construction of nuclear sheaves, most prominently the unit object is usually not compact. We will remedy this by considering a much smaller subcategory $\D_\solid(X,\Z_\ell)_{\omega_1} \subset \D_\solid(X,\Z_\ell)$ of \emph{$\omega_1$-solid sheaves}. This subcategory is generated under small colimits by \emph{countable} limits of qcqs étale sheaves. Since countable limits have finite cohomological dimension, this gives us a lot of control. In particular, the $\infty$-category of $\omega_1$-solid sheaves has similar properties as $\D_\solid(X,\Z_\ell)$, but under mild assumptions on $X$ the unit is compact. This ``mild assumption'' on $X$ is the following:

\begin{definition}
A spatial diamond $X$ is called \emph{$\ell$-bounded} if there is some integer $d$ such that for all static étale $\Fld_\ell$-modules $\mathcal M \in \D_\et(X,\Fld_\ell)$ on $X$ we have $H^k(X, \mathcal M) = 0$ for $k > d$.
\end{definition}

\begin{definition}
Let $X$ be a spatial diamond. A quasi-pro-étale map $U \to X$ is called \emph{basic} if it can be written as a cofiltered limit $U = \varprojlim_i U_i$ such that all $U_i \to X$ are étale, quasicompact and separated.
\end{definition}

\begin{remark}
Suppose that $X$ is an $\ell$-bounded spatial diamond and let $\mathcal M \in \D_\et(X, \Z)$ be any static étale sheaf on $X$ which is killed by some power of $\ell$. Then for the same $d$ as in the definition of $\ell$-boundedness and all basic $U \in X_\proet$ we have $H^k(U, \mathcal M) = 0$ for $k > d+1$. Namely, if $U = X$ then this follows by writing $\mathcal M = \varprojlim_n \mathcal M/\ell^n \mathcal M$ and noting that for each $\mathcal M/\ell^n\mathcal M$ the result follows by repeatedly applying fiber sequences of the form $\mathcal M/\ell^{n-1}\mathcal M \to \mathcal M/\ell^n \mathcal M \to \mathcal M/\ell \mathcal M$. If $U$ is étale over $X$ then the result follows from the observation that the étale pushforward along $U \to X$ is $t$-exact (this can be checked after pullback to a strictly totally disconnected space, in which case $U$ is also strictly totally disconnected and therefore has vanishing étale cohomology). For general $U = \varprojlim_i U_i$ use that $H^k(U, \mathcal M) = \varinjlim_i H^k(U_i, \mathcal M)$. In the following we will implicitly make use of these facts.
\end{remark}

Let us now come to the definition of $\omega_1$-solid sheaves. Here by $\omega_1$ we mean the first uncountable cardinal.

\begin{definition}
Let $X$ be an $\ell$-bounded spatial diamond. A solid $\Z_\ell$-module $\mathcal M \in \D_\solid(X, \Z_\ell)$ on $X$ is called \emph{$\omega_1$-solid} if for every $\omega_1$-cofiltered limit $U = \varprojlim_i U_i$ of basic objects in $X_\proet$ the natural map
\begin{align*}
	\varinjlim_i \Gamma(U_i, \mathcal M) \isoto \Gamma(U, \mathcal M)
\end{align*}
is an isomorphism. We denote by
\begin{align*}
	\D_\solid(X,\Z_\ell)_{\omega_1} \subset \D_\solid(X,\Z_\ell)
\end{align*}
the full subcategory spanned by the $\omega_1$-solid sheaves.
\end{definition}

The basic properties of $\D_\solid(X,\Z_\ell)_{\omega_1}$ are summarized by the following result, which appeared to be surprisingly subtle.

\begin{proposition} \label{rslt:general-properties-of-omega-1-solid-sheaves}
Let $X$ be an $\ell$-bounded spatial diamond.
\begin{propenum}
	\item \label{rslt:compact-generators-of-etale-sheaves} Let $\D_\et(X, \Z_\ell) \subset \D_\solid(X, \Z_\ell)$ denote the full subcategory spanned by those sheaves whose underlying abelian sheaf lies in $\D_\et(X, \Z)$. Then $\D_\et(X, \Z_\ell)$ is compactly generated and the compact objects are generated under finite (co)limits and retracts by the objects $\Fld_\ell[U]$ for quasicompact separated $U \in X_\et$. Moreover, the $t$-structure on $\D_\et(X,\Z_\ell)$ restricts to a $t$-structure on $\D_\et(X,\Z_\ell)^\omega$.

	\item \label{rslt:compact-generators-of-w1-solid-sheaves} $\D_\solid(X, \Z_\ell)_{\omega_1}$ is compactly generated. The compact objects are $\ell$-adically complete and generated under finite (co)limits and retracts by the objects
	\begin{align*}
		\Z_{\ell,\solid}[U] = \varprojlim_n \Z_\ell[U_n] = \varprojlim_n (\Z/\ell^n\Z)[U_n]
	\end{align*}
	for sequential limits $U = \varprojlim_n U_n$ with all $U_n \to X$ being étale, quasicompact and separated. This identifies $\D_\solid(X, \Z_\ell)^\omega_{\omega_1}$ with a full subcategory
	\begin{align*}
		\D_\solid(X, \Z_\ell)_{\omega_1}^\omega \subset \Pro(\D_\et(X,\Z_\ell)^\omega).
	\end{align*}

	\item \label{rslt:stability-of-w1-solid-sheaves} $\D_\solid(X, \Z_\ell)_{\omega_1}$ is stable under all colimits and countable limits in $\D_\solid(X,\Z_\ell)$ and contains $\D_\et(X,\Z_\ell)$. It admits a complete $t$-structure and a symmetric monoidal structure by restricting the ones on $\D_\solid(X,\Z_\ell)$. Moreover, the compact objects in $\D_\solid(X,\Z_\ell)_{\omega_1}$ are stable under tensor product.
\end{propenum}
\end{proposition}
\begin{proof}
We first prove (i). We claim that the forgetful functor $\D_\et(X,\Z_\ell) \to \D_\et(X,\Z)$ is fully faithful. Indeed, the right adjoint of the forgetful functor $\D(X_\proet, \Z_\ell) \to \D(X_\proet, \Z)$ is given by $\IHom_\Z(\Z_\ell, -)$ and one observes that this functor preserves étale sheaves; in fact, if $\mathcal M \in \D_\et(X, \Z)$ then $\IHom_\Z(\Z_\ell, -) = \varinjlim_n \IHom_\Z(\Z/\ell^n \Z, \mathcal M)$ by the usual Breen resolution argument (cf. the proof of \cite[Proposition VII.1.12]{fargues-scholze-geometrization}). One deduces from the same formula that if $\mathcal M \in \D_\et(X,\Z_\ell)$ then $\IHom_\Z(\Z_\ell, \mathcal M) = \mathcal M$, proving the desired fully faithfulness.

The embedding $\D_\et(X,\Z_\ell) \injto \D_\et(X,\Z)$ is $t$-exact. Moreover, if $\mathcal P \in \D_\et(X,\Z_\ell)$ is perfect constructible as an object in $\D_\et(X,\Z)$ then it is compact as an object in $\D_\et(X,\Z_\ell)$: By the proof of \cite[Proposition 20.17]{etale-cohomology-of-diamonds} this reduces to showing that $\Hom(\Z[U], -) = \Gamma(U, -)$ preserves small colimits of objects in $\D_\et(X,\Z_\ell)$, which follows from $\ell$-boundedness of $X$ (using the fact that every $\mathcal M \in \D_\et(X,\Z_\ell)$ can be written as $\mathcal M = \varinjlim_n \IHom(\Z/\ell^n\Z, \mathcal M)$ by the previous paragraph). We now prove that conversely every compact object in $\D_\et(X,\Z_\ell)$ is perfect construcible as an object in $\D_\et(X,\Z)$. First observe that for every quasicompact separated $U \in X_\et$, $\Fld_\ell[U] \in \D_\et(X,\Z_\ell)$ is compact because it is perfect constructible in $\D_\et(X,\Z)$. We now claim that the objects $\Fld_\ell[U]$ generate $\D_\et(X,\Z_\ell)$. Note that these objects generate $(\Z/\ell^n\Z)[U][k]$ for all integers $n \ge 1$ and $k$; by abstract nonsense it is therefore enough to show that the family of functors $\Hom((\Z/\ell^n\Z)[U][k], -)$ is conservative. We can ignore the shifts by $k$ by instead taking spectra-enriched $\Hom$'s. We now claim that for every $\mathcal M \in \D_\et(X,\Z_\ell)$ the natural map
\begin{align*}
	\varinjlim_n \Hom((\Z/\ell^n\Z)[U], \mathcal M) \isoto \Hom(\Z_\ell[U], \mathcal M) = \Gamma(U, \mathcal M),
\end{align*}
is an isomorphism of spectra. Using the $\ell$-boundedness of $X$, a standard Postnikov limit argument reduces this claim to the case that $\mathcal M$ is left-bounded. Since both sides commute with colimits in $\mathcal M$ we can further reduce to the case that $\mathcal M$ is static. But then the usual Breen-Deligne resolution works (cf. the proof of \cite[Proposition VII.1.12]{fargues-scholze-geometrization}). We are thus reduced to showing that the family of functors $\Gamma(U, -)$ is conservative; but this is clear. In particular we deduce the claimed description of compact objects in $\D_\et(X,\Z_\ell)$ and identify them as precisely those objects which are perfect constructible in $\D_\et(X,\Z)$. By \cite[Proposition 20.12]{etale-cohomology-of-diamonds} the $t$-structure on $\D_\et(X,\Z)$ restricts to a $t$-structure on perfect constructible sheaves; this finishes the proof of (i).

We now prove (ii) and (iii). The fact that $\D_\solid(X,\Z_\ell)_{\omega_1}$ is stable under countable limits in $\D_\solid(X,\Z_\ell)$ follows immediately from the definition by using the fact that countable limits commute with $\omega_1$-filtered colimits in spectra. Also, it is obvious that $\D_\solid(X,\Z_\ell)_{\omega_1}$ contains $\D_\et(X,\Z_\ell)$ because étale sheaves $\mathcal M$ on $X$ satisfy $\Gamma(U, \mathcal M) = \varinjlim_i \Gamma(U_i, \mathcal M)$ for all cofiltered limits $U = \varprojlim_i U_i$ (this holds for unbounded $\mathcal M$ by the $\ell$-boundedness of $X$).

Let now $U = \varprojlim_n U_n$ be given as in (ii). We first check that $\Z_{\ell,\solid}[U]$ is $\omega_1$-solid. Since $\omega_1$-solid sheaves are stable under countable limits we reduce to showing this for $\Z_\ell[U_n]$ for all $n$. Since étale sheaves are $\omega_1$-solid, we further reduce to showing that the natural map $\Z_\ell[U_n] \isoto \varprojlim_n (\Z/\ell^n \Z)[U_n]$ is an isomorphism. Both sides of this claimed isomorphism are static, so we can check this by applying Yoneda in the heart, i.e. we need to see that for every $\mathcal M \in \D_\solid(X,\Z_\ell)^\heartsuit$ the natural map
\begin{align*}
	\Hom(\varprojlim_n (\Z/\ell^n \Z)[U_n], \mathcal M) \isoto \Hom(\Z_\ell[U_n], \mathcal M)
\end{align*}
is an isomorphism. Since both $\Z_\ell[U_n]$ and $\Z_\ell/\ell^n \Z_\ell)[U_n]$ are compact objects in the heart (by the description of finitely presented, i.e. compact, objects in \cite[Theorem VII.1.3]{fargues-scholze-geometrization}), so we can reduce to $\mathcal M$ being of the form $\mathcal M = \varprojlim_i \mathcal M_i$ for qcqs étale sheaves $\mathcal M_i$. Pulling out limits from both sides of the claimed $\Hom$-identity, we reduce to the case that $\mathcal M$ is étale. But then by the proof of \cite[Theorem VII.1.3]{fargues-scholze-geometrization} we have
\begin{align*}
	&\Hom(\varprojlim_n (\Z/\ell^n \Z)[U_n], \mathcal M) = \varinjlim_n \Hom((\Z/\ell^n \Z)[U_n], \mathcal M) = \varinjlim_n \Hom(\Z/\ell^n\Z, \restrict{\mathcal M}U)\\
	&\qquad= \Hom(\Z_\ell, \restrict{\mathcal M}U) = \Hom(\Z_\ell[U], \mathcal M),
\end{align*}
as desired. This finishes the proof that $\Z_{\ell,\solid}[U]$ is $\omega_1$-solid. Now let $\mathcal C \subset \D_\solid(X,\Z_\ell)_{\omega_1}$ be the full subcategory generated under finite (co)limits and retracts by the objects $\Z_{\ell,\solid}[U]$ for $U = \varprojlim_n U_n$ as in (ii). This induces a natural functor
\begin{align*}
	\alpha\colon \Ind(\mathcal C) \to \D_\solid(X,\Z_\ell).
\end{align*}
Our goal will be to show that $\alpha$ induces an equivalence of $\Ind(\mathcal C)$ and $\D_\solid(X,\Z_\ell)_{\omega_1}$. As a first step towards seeing this, let us prove that the $t$-structure on $\D_\solid(X,\Z_\ell)$ restricts to a $t$-structure on the essential image of $\alpha$. Namely, this reduces to showing that for every $\mathcal P \in \mathcal C$, all $\pi_k \mathcal P$ lie in the essential image of $\alpha$. By \cite[Theorem VII.1.3]{fargues-scholze-geometrization}, $\pi_k \mathcal P$ is finitely presented and can be written as a countable limit $\pi_k \mathcal P = \varprojlim_n \mathcal P_{n,k}$ of qcqs étale sheaves. The latter property implies that $\pi_k \mathcal P$ is $\omega_1$-solid, which in turn implies that the sections of $\pi_k P$ on any basic $U \in X_\proet$ are determined by the sections of $\pi_k P$ on the $U = \varprojlim_n U_n$ as in (ii). But note that basic $U$ form a basis of $X_\proet$ (this is clear if $X$ is a perfectoid space and follows from \cite[Proposition 11.24, 11.23.(iii)]{etale-cohomology-of-diamonds} in general), hence there is a surjection of the form $\bigdsum_i \Z_{\ell,\solid}[U_i] \surjto \pi_k \mathcal P$ with all $U_i$ as in (ii). Note that this direct sum is still $\omega_1$-solid, hence so is the kernel of the map to $\pi_k \mathcal P$ so that we get a two-term resolution $\bigdsum_j \Z_{\ell,\solid}[U_j] \to \bigdsum_i \Z_{\ell,\solid}[U_i] \surjto \pi_k \mathcal P$. By writing this map of infinite direct sums as a filtered colimit of maps of finite direct sums and using that $\pi_k \mathcal P$ is finitely presented, we deduce that one arrange that both direct sums are finite, i.e. lie in $\mathcal C$. Continuing this argument, we obtain a resolution of $\pi_k \mathcal P$ by objects in $\mathcal C$ which implies that $\pi_k \mathcal P$ is a geometric resolution of objects in $\mathcal C$ and in particular a filtered colimit of objects in $\mathcal C$ (as every finite step of the geometric resolution lies in $\mathcal C$). This implies that $\pi_k \mathcal P$ lies in the essential image of $\alpha$, as desired.

Next we claim that for $d$ as in the definition of $\ell$-boundedness, for every $\mathcal M$ in the heart of the essential image of $\alpha$ and for every basic $U$ we have $H^k(U, \mathcal M) = 0$ for $k > d + 2$. Indeed, since $H^k(U, \mathcal M)$ commutes with filtered colimits of static $\mathcal M$, we can reduce to the case that $\mathcal M = \pi_m \mathcal P$ for some $\mathcal P \in \mathcal C$. As seen in the previous paragraph this implies that $\mathcal M$ is a sequential limit $\mathcal M = \varprojlim_n \mathcal M_n$ for qcqs static $\mathcal M_n \in \D_\et(X,\Z_\ell)$. By factoring the countable limit out of $H^k(U, -)$ and noting that countable limits have cohomological dimension 1 in spectra, we reduce to showing that $H^k(U, \mathcal M_n) = 0$ for $k > d$, which follows immediately from the definition of $\ell$-boundedness.

We are now ready to prove that $\alpha$ is fully faithful. This amounts to showing the following: Given $U$ as in (ii) and any filtered diagram $(\mathcal P_i)_i$ of objects in $\mathcal C$, the natural map
\begin{align*}
	\varinjlim_i \Hom(\Z_{\ell,\solid}[U], \mathcal P_i) \isoto \Hom(\Z_{\ell,\solid}[U], \varinjlim_i \mathcal P_i)
\end{align*}
is an isomorphism (here the colimit $\varinjlim_i \mathcal P_i$ is formed in $\D_\solid(X,\Z_\ell)$). Using the identification $\Hom(\Z_{\ell,\solid}[U], -) = \Gamma(U, -)$ and hence the fact that this functor has finite cohomological dimension on the essential image of $\alpha$, we can employ a standard Postnikov limit argument to reduce to the claim that $\Gamma(U, -)$ preserves filtered colimits of uniformly left-bounded sheaves.

We have established that $\alpha\colon \Ind(\mathcal C) \injto \D_\solid(X,\Z_\ell)$ is an embedding. We now show that the essential image is precisely $\D_\solid(X,\Z_\ell)_{\omega_1}$. First observe that the essential image of $\alpha$ is contained in $\D_\solid(X,\Z_\ell)_{\omega_1}$. Indeed, this follows from the fact that every $\mathcal P \in \mathcal C$ is $\omega_1$-solid and that for all basic $U$ the functor $\Gamma(U,-)$ preserves filtered colimits in the essential image of $\alpha$ because it has finite cohomological dimension on that image (as used in the previous paragraph). It remains to see that every $\omega_1$-solid $\Z_\ell$-module on $X$ lies in the essential image of $\alpha$. By standard arguments (e.g. looking at the right adjoint to the embedding $\Ind(\mathcal C) \injto \D_\solid(X,\Z_\ell)_{\omega_1}$) this reduces to showing that the family of functors $\Hom(\Z_{\ell,\solid}[U], -) = \Gamma(U,-)$, for $U$ as in (ii), is conservative on $\D_\solid(X,\Z_\ell)_{\omega_1}$. But by definition of $\omega_1$-solid objects, these functors determine $\Gamma(U,-)$ for \emph{all} basic $U \in X_\proet$, and as the basic $U$ form a basis of $X_\proet$ (as noted above), this family of functors is indeed conservative.

We have finally shown the equivalence $\Ind(\mathcal C) \isom \D_\solid(X,\Z_\ell)_{\omega_1}$. The rest of (ii) goes as follows: Everything is clear except that the compact objects of $\D_\solid(X,\Z_\ell)_{\omega_1}$ are $\ell$-adically complete and embed into $\Pro(\D_\et(X,\Z_\ell)^\omega)$. The $\ell$-adic completeness was shown for $\Z_\ell[U]$ with qcqs $U \in X_\et$ above and follows immediately for all compact objects because $\ell$-adic completeness is stable under limits. To get the embedding into the $\Pro$-category, we use the following more general statement: Equip $\Pro(\D_\et(X,\Z_\ell)^\omega)$ with the natural $t$-structure induced from the one on $\D_\et(X,\Z_\ell)^\omega$ (cf. \cite[Lemma C.2.4.3]{lurie-spectral-algebraic-geometry}); then the natural functor
\begin{align*}
	\Pro(\D_\et(X,\Z_\ell)^\omega)^b \injto \D_\solid(X,\Z_\ell)
\end{align*}
is fully faithful. Namely, the functor $\Pro(\D_\et(X,\Z_\ell)^\omega) \to \D_\solid(X,\Z_\ell)$ is $t$-exact by \cite[Proposition VII.1.6]{fargues-scholze-geometrization} and so the claimed fully faithfulness reduces to showing that $\Hom(\varprojlim_i \mathcal P_i, \mathcal Q) = \varinjlim_i \Hom(\mathcal P_i, \mathcal Q)$ for compact static $\mathcal P_i, \mathcal Q \in \D_\et(X,\Z_\ell)$; this follows from the usual Breen resolution argument (as in \cite[Proposition VII.1.12]{fargues-scholze-geometrization}).

The rest of (iii) is easy: We have already seen that $\D_\solid(X,\Z_\ell)_{\omega_1}$ is stable under countable limits and contains all étale sheaves. By the proof of (ii) it follows that it is also stable under all small colimits and that the $t$-structure on $\D_\solid(X,\Z_\ell)$ restricts to a $t$-structure on $\D_\solid(X,\Z_\ell)_{\omega_1}$. It remains to see that $\D_\solid(X,\Z_\ell)_{\omega_1}$ is stable under tensor product: This can be checked on compact generators and thus follows from the observation $\Z_{\ell,\solid}[U] \tensor \Z_{\ell,\solid}[V] = \Z_{\ell,\solid}[U \cprod_X V]$; this also proves that compact objects are stable under tensor product.
\end{proof}

In order to work with $\omega_1$-solid sheaves it is important to understand how they behave under a change of spatial diamond:

\begin{proposition} \label{rslt:pullback-pushforward-for-w1-solid-sheaves}
Let $f\colon Y \to X$ be a map of $\ell$-bounded spatial diamonds.
\begin{propenum}
	\item \label{rslt:pullback-for-w1-solid-sheaves} The pullback $\D(X_\vsite,\Z_\ell) \to \D(Y_\vsite,\Z_\ell)$ restricts to a $t$-exact symmetric monoidal functor
	\begin{align*}
		f^*\colon \D_\solid(X,\Z_\ell)_{\omega_1} \to \D_\solid(Y,\Z_\ell)_{\omega_1}
	\end{align*}
	which preserves all small colimits and all countable limits.

	\item \label{rslt:pushforward-for-w1-solid-sheaves} The pushforward $f_{\vsite*}\colon \D(Y_\vsite,\Z_\ell) \to \D(X_\vsite,\Z_\ell)$ restricts to a functor
	\begin{align*}
		f_{\solid*}\colon \D_\solid(Y,\Z_\ell)_{\omega_1} \to \D_\solid(X,\Z_\ell)_{\omega_1}
	\end{align*}
	which has finite cohomological dimension, is right adjoint to $f^*$ and preserves all small limits and colimits.
\end{propenum}
\end{proposition}
\begin{proof}
Since the pullback on the v-site is $t$-exact and preserves all small limits and colimits, part (i) follows immediately from \cref{rslt:general-properties-of-omega-1-solid-sheaves} and the fact that $f^* \Z_\ell[U] = \Z_\ell[U \cprod_X Y]$ for every quasicompact separated $U \in X_\et$ (for the proof that $f^*$ is symmetric monoidal see e.g. \cite[Proposition VII.2.2]{fargues-scholze-geometrization}). This proves (i).

To prove (ii), first note that the v-pushforward has finite cohomological dimension when restricted to $\D_\solid(Y,\Z_\ell)_{\omega_1}$: Recall that for a static v-sheaf $\mathcal M$ on $X$, $H^k(f_*\mathcal M)$ is the sheafification of the presheaf $U \mapsto H^k(U \cprod_X Y, \mathcal M)$; but if $\mathcal M$ is $\omega_1$-solid then $H^k(U, \mathcal M) = 0$ for $k > d+1$ with $d$ as in the definition of $\ell$-boundedness for $Y$ (see the proof of \cref{rslt:general-properties-of-omega-1-solid-sheaves}). It follows that $f_{\vsite*}$ preserves small colimits on $\D_\solid(Y,\Z_\ell)_{\omega_1}$, hence to show that it maps this $\infty$-category to $\D_\solid(X,\Z_\ell)_{\omega_1}$ we are reduced to the compact generators, i.e. we need to show that $f_{\vsite*}\Z_{\ell,\solid}[U] \in \D_\solid(X,\Z_\ell)_{\omega_1}$ for all $U = \varprojlim_n U_n$ as in \cref{rslt:compact-generators-of-w1-solid-sheaves}. But $f_{\vsite*}\Z_{\ell,\solid}[U] = \varprojlim_n f_{\vsite*}(\Z/\ell^n\Z)[U_n]$ so since $\D_\solid(X,\Z_\ell)_{\omega_1}$ is stable under countable limits it is enough to show that $f_{\vsite*}(\Z/\ell^n\Z)[U_n]$ is étale -- this follows from \cite[Corollary 16.8.(ii)]{etale-cohomology-of-diamonds}.
\end{proof}

\begin{corollary} \label{rslt:descent-for-w1-solid-sheaves}
The assignment
\begin{align*}
	X \mapsto \D_\solid(X,\Z_\ell)_{\omega_1}
\end{align*}
defines a hypercomplete sheaf of symmetric monoidal $\infty$-categories on the v-site of $\ell$-bounded spatial diamonds.
\end{corollary}
\begin{proof}
By using the embedding $\D_\solid(X,\Z_\ell) \subset \D(X_\vsite,\Z_\ell)$ and using that the right-hand category satisfies hypercomplete v-descent for formal reasons, we are left with showing that for any v-cover $f\colon Y \to X$ of $\ell$-bounded spatial diamonds and any $\mathcal M \in \D(X_\vsite,\Z_\ell)$ we have $f^*\mathcal M \in \D_\solid(Y,\Z_\ell)_{\omega_1}$ if and only if $\mathcal M \in \D_\solid(X,\Z_\ell)_{\omega_1}$. The ``if'' part follows immediately from \cref{rslt:pullback-for-w1-solid-sheaves}. It remains to prove the ``only if'' part, so assume that $f^* \mathcal M$ is $\omega_1$-solid. Extend $f$ to a v-hypercover $f_\bullet\colon Y_\bullet \to X$ such that all $Y_n$ are $\ell$-bounded spatial diamonds. Then by v-hyperdescent for v-sheaves we have $\mathcal M = \varprojlim_{n\in\Delta} f_{n\vsite*} f_n^* \mathcal M$, which we want to show to be $\omega_1$-solid. Since $\omega_1$-solid sheaves are stable under countable limits by \cref{rslt:stability-of-w1-solid-sheaves} it is enough to show that each $f_{n\vsite*} f_n^* \mathcal M$ is $\omega_1$-solid. But this follows immediately from \cref{rslt:pullback-pushforward-for-w1-solid-sheaves} since $f_0^* \mathcal M$ is $\omega_1$-solid by assumption.
\end{proof}

It is convenient to know that the tensor product of $\ell$-adically complete sheaves is again $\ell$-adically complete:

\begin{proposition} \label{rslt:solid-tensor-product-preserves-complete-sheaves}
Let $X$ be an $\ell$-bounded spatial diamond and let $\mathcal M, \mathcal N \in \D^-_\solid(X,\Z_\ell)$ be $\ell$-adically complete. Then $\mathcal M \tensor \mathcal N$ is $\ell$-adically complete.
\end{proposition}
\begin{proof}
By choosing projective resolutions we can write $\mathcal M$ and $\mathcal N$ as geometric realizations of direct sums of copies of the compact projective generators $\Z_{\ell,\solid}[U]$ for w-contractible $U \in X_\proet$. Since the $\ell$-adic completion functor on $\D_\solid(X,\Z_\ell)$ has homological dimension $1$ and therefore preserves geometric realizations of uniformly left-bounded sheaves, we can reduce to the case that $\mathcal M$ and $\mathcal N$ are completed direct sums of copies of $\Z_{\ell,\solid}[U]$ for w-contractible $U \in X_\proet$; in particular both $\mathcal M$ and $\mathcal N$ are static (note that $\Z_{\ell,\solid}[U]$ is $\ell$-adically complete by the proof of \cref{rslt:general-properties-of-omega-1-solid-sheaves}). Since $\ell$-adic completion commutes with $\omega_1$-filtered colimits, we can further reduce to the case of \emph{countable} completed direct sums, i.e. we have $\mathcal M = \cplbigdsum{n} \Z_{\ell,\solid}[U_n]$ and $\mathcal N = \cplbigdsum{m} \Z_{\ell,\solid}[V_m]$ for w-contractible $U_n, V_m \in X_\proet$ and integers $m, n > 0$. Note that we can write $\mathcal M$ as a union of subsheaves of the form $\mathcal M_\alpha = \prod_n \ell^{\alpha_n} \Z_{\ell,\solid}[U_n]$ for all sequences $\alpha_n \in \Z_{\ge0}$ converging to $\infty$ (and similarly for $\mathcal N$ in terms of $\mathcal N_\beta$ for sequences $\beta$). Using \cite[Proposition VII.2.3]{fargues-scholze-geometrization} we deduce
\begin{align*}
	&\mathcal M \tensor \mathcal N = \varinjlim_{\alpha,\beta} \mathcal M_\alpha \tensor \mathcal N_\beta = \varinjlim_{\alpha,\beta} \prod_{m,n} (\ell^{\alpha_n} \Z_{\ell,\solid}[U_n] \tensor \ell^{\beta_m} \Z_{\ell,\solid}[V_m])\\
	&\qquad = \varinjlim_{\alpha,\beta} \prod_{m,n} \ell^{\alpha_n + \beta_m} \Z_{\ell,\solid}[U_n \cprod_X V_m].
\end{align*}
By a simple cofinality argument one checks that this is just the completed direct sum of $\Z_{\ell,\solid}[U_n \cprod_X V_m]$ for all $n, m > 0$. In particular this sheaf is $\ell$-adically complete, as desired.
\end{proof}

\section{Nuclear Sheaves} \label{sec:nuclear}

Fix the prime $\ell \ne p$. In the following we will define a full subcategory $\D_\nuc(X,\Z_\ell) \subset \D_\solid(X,\Z_\ell)$ of nuclear sheaves on any small v-stack $X$. In general this $\infty$-category will be defined by v-descent, but on $\ell$-bounded spatial diamonds it admits a much more explicit description:

\begin{definition}
Let $X$ be an $\ell$-bounded spatial diamond and let $\mathcal M \in \D_\solid(X, \Z_\ell)$ be a solid $\Z_\ell$-sheaf on $X$.
\begin{defenum}
	\item $\mathcal M$ is called a \emph{Banach sheaf} if $\mathcal M$ is $\ell$-adically complete and $\mathcal M/\ell \mathcal M$ is étale, i.e. lies in $\D_\et(X, \Fld_\ell)$.

	\item \label{def:nuclear-module-on-spat-diamond} $\mathcal M$ is called \emph{nuclear} if it is a filtered colimit of Banach sheaves. We denote by $\D_\nuc(X, \Z_\ell) \subset \D_\solid(X, \Z_\ell)$ the full subcategory spanned by the nuclear sheaves.
\end{defenum}
\end{definition}

\begin{remarks}
\begin{remarksenum}
	\item It follows immediately from \cref{rslt:stability-of-w1-solid-sheaves} that every nuclear $\Z_\ell$-sheaf on an $\ell$-bounded spatial diamond $X$ is $\omega_1$-solid, i.e. we have $\D_\nuc(X,\Z_\ell) \subset \D_\solid(X,\Z_\ell)_{\omega_1}$.

	\item The $t$-structure on $\D_\solid(X,\Z_\ell)$ does usually not restrict to a $t$-structure on $\D_\nuc(X,\Z_\ell)$. The problem is that for a Banach sheaf $\mathcal M$ it is generally not true that $\tau_{\ge0} \mathcal M$ is still a Banach sheaf: While it is true that $\tau_{\ge0} \mathcal M$ is still $\ell$-adically complete, it may happen that it is not étale mod $\ell$ anymore. As a counterexample, suppose that $X = \Z_\ell$ (viewed as a pro-finite étale space over some algebraically closed non-archimedean field) and let $\mathcal M$ be the $\ell$-adic completion of the following étale sheaf $\mathcal N$: $\mathcal N$ is the étale product of sheaves $\mathcal N_n$ for $n \ge 1$, where $\mathcal N_n$ is the constant sheaf $\Z/\ell^n\Z$ supported on $X \setminus \ell^n \Z_\ell$. Then $\mathcal M$ is a Banach sheaf concentrated in homological degrees $0$ and $1$, but by a direct computation one checks that $\pi_1 M / \ell \pi_1 M$ is not étale: It does not have any étale sections, but it has a non-trivial value on the pro-étale map $\{ 0 \} \to X$.
\end{remarksenum}
\end{remarks}

\begin{warning}
In \cite[Definition 8.5]{condensed-complex-geometry} there is a very general definition of nuclear objects in a compactly generated symmetric monoidal $\infty$-category. However, our \cref{def:nuclear-module-on-spat-diamond} differs from that general definition. In fact, if we apply the general definition from \cite{condensed-complex-geometry} to $\D_\solid(X,\Z_\ell)_{\omega_1}$ then we end up with a subcategory $\mathcal C \subset \D_\nuc(X,\Z_\ell)$ spanned only by the overconvergent sheaves. There are still good reasons to call our sheaves nuclear, as we will see in the following.
\end{warning}

The above definition of nuclear sheaves may seem a bit ad-hoc and it is not at all clear from this definition why the $\infty$-category of nuclear sheaves satisfies v-descent. Our first goal will therefore be to provide an equivalent definition in terms of \emph{trace-class} maps. The general definition of trace-class maps provided in \cite[\S8]{condensed-complex-geometry} does not work in our setting, as it is far too restrictive. We will therefore provide our own definition tailored to the specific setting at hand:

\begin{definition}
Let $X$ be an $\ell$-bounded spatial diamond.
\begin{defenum}
	\item We denote by $\IHom_\solid(-,-)$ the internal hom functor in the symmetric monoidal $\infty$-category $\D_\solid(X,\Z_\ell)_{\omega_1}$ (this exists because this $\infty$-category is presentable).

	\item We define a functor
	\begin{align*}
		\IHom_\solid^\tr(-, -)\colon \D_\solid(X,\Z_\ell)_{\omega_1}^\opp \cprod \D_\solid(X,\Z_\ell)_{\omega_1} \to \D_\solid(X,\Z_\ell)_{\omega_1}
	\end{align*}
	as follows: This functor will preserve limits in the first argument and for compact first argument it will preserve colimits in the second argument, so that it is enough to construct its restriction to $\D_\solid(X,\Z_\ell)_{\omega_1}^{\omega,\opp} \cprod \D_\solid(X,\Z_\ell)_{\omega_1}^\omega$. In this case we define it as the $\ell$-adic completion
	\begin{align*}
		\IHom_\solid^\tr(-, -) = \widehat{\IHom_\solid'(-, -)},
	\end{align*}
	where $\IHom_\solid'(-,-)$ is the following composition of functors:
	\begin{align*}
		\D_\solid(X,\Z_\ell)_{\omega_1}^{\omega,\opp} \cprod \D_\solid(X,\Z_\ell)_{\omega_1}^\omega \injto \Ind(\D_\et(X,\Z_\ell)^{\omega,\opp}) \cprod \D_\solid(X,\Z_\ell)_{\omega_1}^\omega \to \D_\solid(X,\Z_\ell)_{\omega_1}
	\end{align*}
	Here the first functor is induced by \cref{rslt:compact-generators-of-w1-solid-sheaves} and the second functor is the unique colimit-preserving functor which restricts to $\IHom_\solid(\mathcal P, \mathcal Q)$ for $\mathcal P \in \D_\et(X,\Z_\ell)^\omega$ and $\mathcal Q \in \D_\solid(X,\Z_\ell)_{\omega_1}^\omega$.

	\item There is a natural transformation
	\begin{align*}
		\IHom_\solid^\tr(-, -) \to \IHom_\solid(-, -)
	\end{align*}
	of functors $\D_\solid(X,\Z_\ell)_{\omega_1}^\opp \cprod \D_\solid(X,\Z_\ell)_{\omega_1} \to \D_\solid(X,\Z_\ell)_{\omega_1}$. Namely, it is enough to construct this on compact arguments, where it boils down to a natural transformation $\IHom_\solid' \to \IHom_\solid$, which comes from the fact that $\IHom_\solid'$ is a left Kan extension of the restriction of $\IHom_\solid$ to $\D_\et(X,\Z_\ell)^{\omega,\opp}$ in the first argument.

	\item For all $\mathcal M, \mathcal N \in \D_\solid(X,\Z_\ell)_{\omega_1}$ we denote
	\begin{align*}
		\Hom^\tr(\mathcal M, \mathcal N) := \Gamma(X, \IHom_\solid^\tr(\mathcal M, \mathcal N)).
	\end{align*}
	A map $\mathcal M \to \mathcal N$ is called \emph{trace-class} if it lies in the image of the map $\Hom^\tr(\mathcal M, \mathcal N) \to \Hom(\mathcal M, \mathcal N)$ induced by the natural transformation above.
\end{defenum}
\end{definition}

\begin{remark}
The definition of $\IHom_\solid^\tr(-, -)$ may seem a bit intimidating, so let us provide a more intuitive (albeit less formal) description: Pick any compact objects $\mathcal P, \mathcal Q \in \D_\solid(X,\Z_\ell)_{\omega_1}^\omega$. By \cref{rslt:compact-generators-of-w1-solid-sheaves} we can write $\mathcal P = \varprojlim_n \mathcal P_n$ for compact étale sheaves $\mathcal P_n \in \D_\et(X,\Z_\ell)^\omega$. We then define
\begin{align*}
	\IHom_\solid^\tr(\mathcal P, \mathcal Q) = (\varinjlim_n \IHom_\solid(\mathcal P_n, \mathcal Q))\cplt,
\end{align*}
where $\widehat{}$ denotes $\ell$-adic completion. For general $\mathcal M, \mathcal N \in \D_\solid(X,\Z_\ell)_{\omega_1}$ we write $\mathcal M = \varinjlim_i \mathcal P_i$ and $\mathcal N = \varinjlim_j \mathcal Q_j$ for compact $\mathcal P_i$ and $\mathcal Q_j$ and then define
\begin{align*}
	\IHom_\solid^\tr(\mathcal M, \mathcal N) = \varprojlim_i \varinjlim_j \IHom_\solid^\tr(\mathcal P_i, \mathcal Q_j).
\end{align*}
Here the limit is taken in $\D_\solid(X,\Z_\ell)_{\omega_1}$.
\end{remark}

Before getting back to nuclear sheaves, let us study some of the properties of trace-class maps. It turns out that they enjoy almost the same formal properties as the abstract trace-class maps studied in \cite[\S8]{condensed-complex-geometry}:

\begin{lemma} \label{rslt:properties-of-trace-class-maps}
Let $X$ be an $\ell$-bounded spatial diamond.
\begin{lemenum}
	\item \label{rslt:trace-class-maps-form-ideal} Let $f\colon \mathcal M \to \mathcal N$, $g\colon \mathcal M' \to \mathcal M$ and $h\colon \mathcal N \to \mathcal N'$ be maps in $\D_\solid(X,\Z_\ell)_{\omega_1}$. If $f$ is trace-class then so is $h \comp f \comp g$.

	\item \label{rslt:trace-class-maps-stable-under-tensor-product} If $f\colon \mathcal M \to \mathcal N$ and $f'\colon \mathcal M' \to \mathcal N'$ are trace-class maps in $\D_\solid(X,\Z_\ell)_{\omega_1}$, then so is $f \tensor f'\colon \mathcal M \tensor \mathcal M' \to \mathcal N \tensor \mathcal N'$.

	\item Let $f\colon \mathcal M \to \mathcal N$ be a trace-class map in $\D_\solid(X,\Z_\ell)_{\omega_1}$. Then for every $\mathcal L \in \D_\solid(X,\Z_\ell)_{\omega_1}$ the commutative square
	\begin{center}\begin{tikzcd}
		\IHom_\solid^\tr(\mathcal L, \mathcal M) \arrow[r] \arrow[d] & \IHom_\solid^\tr(\mathcal L, \mathcal N) \arrow[d]\\
		\IHom_\solid(\mathcal L, \mathcal M) \arrow[r] & \IHom_\solid(\mathcal L, \mathcal N)
	\end{tikzcd}\end{center}
	admits a diagonal map $\IHom_\solid(\mathcal L, \mathcal M) \to \IHom_\solid^\tr(\mathcal L, \mathcal N)$ making both triangles commute.

	\item Let $f\colon \mathcal P \to \mathcal M$ be a trace-class map in $\D_\solid(X,\Z_\ell)_{\omega_1}$ with $\mathcal P$ compact. Then there is a compact object $\mathcal Q$ in $\D_\solid(X,\Z_\ell)_{\omega_1}$ such that $f$ factors as $\mathcal P \to \mathcal Q \to \mathcal M$, where $\mathcal P \to \mathcal Q$ is also trace-class.
\end{lemenum}
\end{lemma}
\begin{proof}
We first make the following general observation: For any $\mathcal M, \mathcal N, \mathcal L \in \D_\solid(X,\Z_\ell)_{\omega_1}$ the natural maps
\begin{align*}
	\IHom_\solid(\mathcal M, \mathcal N) \tensor \IHom_\solid^\tr(\mathcal N, \mathcal L) &\to \IHom_\solid(\mathcal M, \mathcal L),\\
	\IHom_\solid^\tr(\mathcal M, \mathcal N) \tensor \IHom_\solid(\mathcal N, \mathcal L) &\to \IHom_\solid(\mathcal M, \mathcal L)
\end{align*}
factor over $\IHom_\solid^\tr(\mathcal M, \mathcal L)$. To see this, consider the $\infty$-category $\mathcal C$ of morphisms $g\colon \mathcal N \to \mathcal N'$ in $\D_\solid(X,\Z_\ell)_{\omega_1}$, where morphisms $g_1 \to g_2$ in $\mathcal C$ are given by commuting squares
\begin{center}\begin{tikzcd}
	\mathcal N_1 \arrow[r,"g_1"] \arrow[d] & \mathcal N'_1\\
	\mathcal N_2 \arrow[r,"g_2"] & \mathcal N'_2 \arrow[u]
\end{tikzcd}\end{center}
(This can easily be constructed as simplicial sets.) Then we can view all of the above expressions as functors in $\mathcal M$, $g\colon \mathcal N \to \mathcal N'$ and $\mathcal L$ (where we replace the second appearing $\mathcal N$ in the above expression by $\mathcal N'$) and the above morphisms as natural transformations of such functors. Since both $\IHom_\solid(\mathcal M, \mathcal L)$ and $\IHom_\solid^\tr(\mathcal M, \mathcal L)$ transform colimits in $\mathcal M$ into limits, we can formally reduce the construction of the desired natural transformation to the case that $\mathcal M$ is compact (use that right Kan extension is a right adjoint functor). Then all functors preserve colimits in $\mathcal N$, so we can reduce to the case that $\mathcal N$ is compact as well. By factoring out colimits in $\mathcal N'$ (like for $\mathcal M$) we can reduce to the case that $\mathcal N'$ is compact. Then we can finally also factor out colimits in $\mathcal L$ to reduce to the case that all of $\mathcal M$, $\mathcal N$, $\mathcal N'$ and $\mathcal L$ are compact. In this case $\IHom_\solid^\tr(\mathcal M, \mathcal L)$ is $\ell$-adically complete, so we can ignore $\ell$-adic completions. We therefore end up with constructing natural transformations
\begin{align*}
	\IHom_\solid(\mathcal M, \mathcal N) \tensor \IHom_\solid'(\mathcal N', \mathcal L) &\to \IHom_\solid'(\mathcal M, \mathcal L),\\
	\IHom_\solid'(\mathcal M, \mathcal N) \tensor \IHom_\solid(\mathcal N', \mathcal L) &\to \IHom_\solid'(\mathcal M, \mathcal L),
\end{align*}
functorial in $\mathcal M$, $g\colon \mathcal N \to \mathcal N'$ and $\mathcal L$. For the second natural transformation, we can write $\mathcal M = \varprojlim_n \mathcal M_n$ for qcqs étale $\mathcal M_n$ and pull out this limit on both sides (by definition of $\IHom_\solid'$), transforming it into colimits. This reduces the construction of the natural transformation to the case that $\mathcal M$ is qcqs étale, in which case $\IHom_\solid'(\mathcal M, -) = \IHom_\solid(\mathcal M, -)$, so we are done. For the first natural transformation we can similarly reduce to the case that $\mathcal N'$ is qcqs étale. Then if we write $\mathcal N = \varprojlim_n \mathcal N_n$ for qcqs étale $n$ we know that $g$ factors over some $\mathcal N_n$ and is thus a limit of maps $\mathcal N_{n'} \to \mathcal N'$ in $\mathcal C$. By the same argument as above we can pull out this limit and thus assume that $\mathcal N$ is also qcqs étale. But then we can further pull out a limit in $\mathcal M$ to reduce to the case that $\mathcal M$ is qcqs étale, in which case the desired natural transformation is evident.

With the above preparations at hand, let us now prove the actual claims. Part (i) follows immediately from the just constructed natural transformations by applying $\Gamma(X,-)$. To prove (iii), note that the given map $f$ provides a map $\Z_\ell \to \IHom_\solid^\tr(\mathcal M, \mathcal N)$, so that the desired diagonal map can be constructed as follows:
\begin{align*}
	\IHom_\solid(\mathcal L, \mathcal M) \to \IHom_\solid(\mathcal L, \mathcal M) \tensor \IHom_\solid^\tr(\mathcal M, \mathcal N) \to \IHom_\solid^\tr(\mathcal L, \mathcal N),
\end{align*}
where the second map is the one constructed above. To prove (ii) we need to show that the natural map
\begin{align*}
	\IHom_\solid^\tr(\mathcal M, \mathcal N) \tensor \IHom_\solid^\tr(\mathcal M', \mathcal N') \to \IHom_\solid(\mathcal M \tensor \mathcal M', \mathcal N \tensor \mathcal N')
\end{align*}
factors over $\IHom_\solid^\tr(\mathcal M \tensor \mathcal M', \mathcal N \tensor \mathcal N')$. This can be done in a similar manner as above by reducing to the case that all sheaves are qcqs étale; we leave the details to the reader. To prove (iv), write $\mathcal M$ as a filtered colimit $\mathcal M = \varinjlim_i \mathcal Q_i$ with compact $\mathcal Q_i$ and note that $\Hom^\tr(\mathcal P, \mathcal M) = \varinjlim_i \Hom^\tr(\mathcal P, \mathcal Q_i)$, which easily implies the claim.
\end{proof}

In order to relate trace-class maps to nuclear sheaves (as defined in \cref{def:nuclear-module-on-spat-diamond}) we need further properties of $\IHom_\solid^\tr$:

\begin{lemma}
Let $X$ be an $\ell$-bounded spatial diamond and let $\mathcal P \in \D_\solid(X,\Z_\ell)_{\omega_1}$ be compact.
\begin{lemenum}
	\item \label{rslt:IHom-tr-is-bounded} The functor
	\begin{align*}
		\IHom_\solid^\tr(\mathcal P, -)\colon \D_\solid(X,\Z_\ell)_{\omega_1} \to \D_\solid(X,\Z_\ell)_{\omega_1}
	\end{align*}
	is bounded, i.e. there are integers $a \le b$ such that for every static $\mathcal M \in \D_\solid(X,\Z_\ell)_{\omega_1}$ the sheaf $\IHom_\solid^\tr(\mathcal P, \mathcal M)$ is concentrated in homological degrees $[a, b]$.

	\item \label{rslt:IHom-tr-preserves-completeness} For every $\ell$-adically complete $\mathcal M \in \D_\solid(X,\Z_\ell)_{\omega_1}$ the sheaf $\IHom_\solid^\tr(\mathcal P, \mathcal M)$ is $\ell$-adically complete.
\end{lemenum}
\end{lemma}
\begin{proof}
By \cref{rslt:compact-generators-of-w1-solid-sheaves} we can assume that $\mathcal P = \Z_{\ell,\solid}[U]$ for some basic $U = \varprojlim_n U_n \in X_\proet$. We first prove (i) in the case that $\mathcal M = \Z_{\ell,\solid}[V]$ for some basic $\varprojlim_m V_m \in X_\proet$. Then we have
\begin{align*}
	\IHom_\solid^\tr(\mathcal P, \mathcal M) = (\varinjlim_n \varprojlim_m \IHom_\solid((\Z/\ell^n\Z)[U_n], \Z_\ell[V_m]))\cplt.
\end{align*}
Thus the desired boundedness boils down to showing that $\IHom_\solid((\Z/\ell^n\Z)[U_n], \Z_\ell[V_m])$ is bounded independent of $V_m$ (since countable limits are bounded by $1$), which follows immediately from the $\ell$-boundedness of $X$. We can now deduce that the functor $\IHom_\solid^\tr(\mathcal P, -)$ is right-bounded by writing any static $\mathcal M$ as a sifted colimit of static compact objects (i.e. choosing a resolution by direct sums of compact objects) and using that $\IHom_\solid^\tr(\mathcal P, -)$ preserves this colimit.

Now let $\mathcal M$ be general (and still static). By writing $\mathcal M$ as a filtered colimit of finitely presented static objects, we can reduce (i) to the case that $\mathcal M$ is finitely presented, i.e. a countable limit $\mathcal M = \varprojlim_n \mathcal M_n$ of qcqs étale sheaves $\mathcal M_n \in \D_\et(X,\Z_\ell)^\omega$. Then $\IHom_\solid^\tr(\mathcal P, \mathcal M)$ is $\ell$-adically complete: Since $\ell$-adic completion is stable under uniformly right-bounded geometric realizations and $\IHom_\solid^\tr(\mathcal P, \mathcal M)$ is right-bounded by the above, we can pick a resolution of $\mathcal M$ by compact objects in order to reduce to the case that $\mathcal M$ is compact -- then the $\ell$-adic completeness is clear by definition. It now follows that the natural map
\begin{align*}
	(\varinjlim_n \IHom_\solid((\Z/\ell^n\Z)[U_n], \mathcal M))\cplt \isoto \IHom_\solid^\tr(\mathcal P, \mathcal M)
\end{align*}
is an isomorphism (by again passing to geometric realizations in $\mathcal M$). In the same way as for compact $\mathcal M$ we deduce that $\IHom_\solid^\tr(\mathcal P, \mathcal M)$ is bounded independent of $\mathcal M$. This finishes the proof of (i).

We now prove (ii), so let the $\ell$-adically complete sheaf $\mathcal M$ be given. By the boundedness of $\IHom_\solid^\tr(\mathcal P, -)$ this functor commutes with Postnikov limits, so we can assume that $\mathcal M$ is left-bounded. Moreover, since $\ell$-adic completeness can be checked on homotopy sheaves and $\IHom_\solid^\tr(\mathcal P, -)$ is bounded, we can use $\mathcal M = \varinjlim_n \tau_{\ge n} \mathcal M$ in order to reduce to the case that $\mathcal M$ is bounded. We can then even assume that $\mathcal M$ is static. By choosing a resolution of $\mathcal M$ in terms of direct sums of copies of $\Z_{\ell,\solid}[V]$ for varying basic $V = \varprojlim_n V_n \in X_\proet$ and using that both $\IHom_\solid^\tr(\mathcal P, -)$ and $\ell$-adic completion commute with uniformly right-bounded geometric realizations we can reduce to the case that $\mathcal M$ is a completed direct sum of copies of $\Z_{\ell,\solid}[V]$. Since $\ell$-adic completion commutes with $\omega_1$-filtered colimits we can further reduce to a countable such sum, i.e. we have $\mathcal M = \cplbigdsum{k} \Z_{\ell,\solid}[V_k]$ for various basic $V_k = \varprojlim_n V_{k,n} \in X_\proet$ and integers $k > 0$. To abbreviate notation let us denote $\mathcal Q_k := \Z_{\ell,\solid}[V_k]$. For every sequence of integers $\alpha_m \ge 0$ converging to $\infty$ we write $\mathcal M_\alpha := \prod_k \ell^{\alpha_k} \mathcal Q_k$, so that $\mathcal M = \varinjlim_\alpha \mathcal M_\alpha$. We now claim that the natural map
\begin{align*}
	\IHom_\solid^\tr(\mathcal P, \mathcal M) = \varinjlim_\alpha \IHom_\solid^\tr(\mathcal P, \mathcal M_\alpha) \isoto \varinjlim_\alpha \prod_k \ell^{\alpha_k} \IHom_\solid^\tr(\mathcal P, \mathcal Q_k)
\end{align*}
is an isomorphism. Here by $\prod_k \ell^{\alpha_k} \IHom_\solid^\tr(\dots)$ we mean the object $\prod_k \IHom_\solid^\tr(\dots)$ but where the $\ell^{\alpha_k}$ determine the transitions maps in the filtered colimit. To prove the claimed isomorphism of the colimits over $\alpha$, it is enough to find suitable sections. Concretely, fix any sequence $\alpha$ and choose another sequence $\alpha'$ such that $\alpha' \le \alpha$ and the sequence $\alpha - \alpha'$ still converges to $\infty$. Then we get natural maps
\begin{align*}
	\prod_k \ell^{\alpha_k} \IHom_\solid^\tr(\mathcal P, \mathcal Q_k) \to \cplbigdsum{k} \ell^{\alpha'_k} \IHom_\solid^\tr(\mathcal P, \mathcal Q_k) \to \IHom_\solid^\tr(\mathcal P, \prod_k \ell^{\alpha'_k} \mathcal Q_k).
\end{align*}
The first map exists because $\alpha - \alpha'$ converges to $\infty$; it is the one which multiplies the $k$-th part of the product/sum by $\ell^{\alpha_k - \alpha'_k}$. The second map exists because $\IHom_\solid^\tr(\mathcal P, \prod_k \ell^{\alpha'_k} \mathcal Q_k)$ is $\ell$-adically complete (by the same argument as in the proof of (i) using that $\prod_k \mathcal Q_k$ is static and finitely presented), so that we can replace the completed direct sum by an ordinary direct sum for the construction. One checks that the thus constructed map is indeed the desired section of the map of filtered systems, proving the above isomorphism. To finish the proof that $\IHom_\solid^\tr(\mathcal P, \mathcal M)$ is $\ell$-adically complete, it remains to see that the natural map
\begin{align*}
	\varinjlim_\alpha \prod_k \ell^{\alpha_k} \IHom_\solid^\tr(\mathcal P, \mathcal Q_k) \isoto \cplbigdsum{k} \IHom_\solid^\tr(\mathcal P, \mathcal Q_k)
\end{align*}
is an isomorphism, or equivalently that the source of this map is $\ell$-adically complete. This can be checked on homotopy sheaves and by the usual arguments involving geometric realizations we reduce to showing this statement after replacing each $\IHom_\solid^\tr(\mathcal P, \mathcal Q_k)$ by a completed direct sum of static compact objects. Then everything is static and the claimed isomorphism can easily be checked on sections $\Gamma(U, -)$ for w-contractible $U$ (reducing the problem to an easy problem about classical abelian groups).
\end{proof}

We are finally in the position to provide the promised characterization of nuclear sheaves in terms of trace-class maps. This will immediately enable us to prove v-descent as well, leading to our first main result on nuclear sheaves in this paper.

\begin{definition}
Let $X$ be an $\ell$-bounded spatial diamond. A sheaf $\mathcal N \in \D_\solid(X,\Z_\ell)_{\omega_1}$ is called \emph{basic nuclear} if it can be written as a sequential colimit $\mathcal N = \varinjlim_n \mathcal P_n$ such that all $\mathcal P_n$ are compact and all transition maps $\mathcal P_n \to \mathcal P_{n+1}$ are trace-class.
\end{definition}

\begin{theorem} \label{rslt:main-results-on-nuclear-Z-ell-sheaves}
\begin{thmenum}
	\item \label{rslt:properties-of-nuclear-subcategory} Let $X$ be an $\ell$-bounded spatial diamond. Then for an $\omega_1$-solid sheaf $\mathcal M$ on $X$ the following are equivalent:
	\begin{enumerate}[(a)]
		\item $\mathcal M$ is nuclear.
		\item $\mathcal M$ is an $\omega_1$-filtered colimit of basic nuclear sheaves.
		\item For every compact $\mathcal P \in \D_\solid(X,\Z_\ell)_{\omega_1}$ the natural map $\Hom^\tr(\mathcal P, \mathcal M) \to \Hom(\mathcal P, \mathcal M)$ is an isomorphism.
	\end{enumerate}
	The $\infty$-category $\D_\nuc(X,\Z_\ell)$ is $\omega_1$-compactly generated with the $\omega_1$-compact objects being precisely the basic nuclear sheaves. Moreover, this $\infty$-category contains all étale sheaves and is stable under all small colimits and under the tensor product in $\D_\solid(X,\Z_\ell)_{\omega_1}$.

	\item \label{rslt:v-descent-for-nuclear-sheaves} The assignment
	\begin{align*}
		X \mapsto \D_\nuc(X,\Z_\ell)
	\end{align*}
	defines a hypercomplete sheaf of presentable symmetric monoidal $\infty$-categories on the v-site of $\ell$-bounded spatial diamonds.
\end{thmenum}
\end{theorem}
\begin{proof}
We first prove (i), so let $X$ and $\mathcal M$ be given. The equivalence of (b) and (c) follows formally from \cref{rslt:properties-of-trace-class-maps}, see \cite[Theorem 8.6]{condensed-complex-geometry}. Now assume that $\mathcal M$ is nuclear; we will show that it satisfies condition (c). Since (c) is stable under colimits in $\mathcal M$ we can assume that $\mathcal M$ is a Banach module. Then by \cref{rslt:IHom-tr-preserves-completeness} both $\IHom_\solid^\tr(\mathcal P, \mathcal M)$ and $\IHom_\solid(\mathcal P, \mathcal M)$ are $\ell$-adically complete, so it is enough to show that $\mathcal N := \mathcal M/\ell \mathcal M$ satisfies (c). Now $\mathcal N$ is étale, so by \cref{rslt:compact-generators-of-etale-sheaves} we can write it as a colimit of copies of $\Fld_\ell[U]$ for varying basic $U \in X_\et$. It is therefore enough to show that $\Fld_\ell[U]$ satisfies (c), which follows by explicit computation (using that if we write $\mathcal P = \varprojlim_n \mathcal P_n$ for qcqs étale sheaves $\mathcal P_n$ then $\Hom(\mathcal P, \Fld_\ell[U]) = \varinjlim_n \Hom(\mathcal P_n, \Fld_\ell[U])$).

To finish the equivalence of (a), (b) and (c) it remains to show that (b) implies (a), so from now on assume that $\mathcal M = \varinjlim_n \mathcal P_n$ is basic nuclear; we need to show that $\mathcal M$ is nuclear. This follows if we can show that every map $\mathcal P_n \to \mathcal P_{n+1}$ factors over some Banach sheaf because then $\mathcal M$ is the colimit of these Banach sheaves. Let $\mathcal P'_{n+1}$ be the $\ell$-adic completion of the pushforward of $\mathcal P_{n+1}$ to the étale site. We claim that the natural map
\begin{align*}
	\Hom(\mathcal P_n, \mathcal P'_{n+1}) = \Hom^\tr(\mathcal P_n, \mathcal P'_{n+1}) \isoto \Hom^\tr(\mathcal P_n, \mathcal P_{n+1})
\end{align*}
is an isomorphism (here the first identity follows from (c) because $\mathcal P'_{n+1}$ is clearly nuclear), which provides the desired factorization $\mathcal P_n \to \mathcal P'_{n+1} \to \mathcal P_{n+1}$. To prove this isomorphism, note that by \cref{rslt:IHom-tr-preserves-completeness} both sides of the claimed isomorphism are $\ell$-adically complete, so we can check the isomorphism modulo $\ell$. Then it boils down to the following claim: Let $\mathcal M$ ($= \mathcal P'_{n+1}/\ell \mathcal P'_{n+1}$) be an $\Fld_\ell$-module in $\D_\solid(X,\Z_\ell)_{\omega_1}$ and let $\mathcal M_\et$ denote its pushforward to the étale site. Then the natural map
\begin{align*}
	\Hom(\mathcal P_n, \mathcal M_\et) = \Hom^\tr(\mathcal P_n, \mathcal M_\et) \isoto \Hom^\tr(\mathcal P_n, \mathcal M)
\end{align*}
is an isomorphism (here the first identity follows because $\mathcal M_\et$ is nuclear and hence satisfies (c) by the above). Both sides of the claimed isomorphism preserve all small colimits (for the functor $\mathcal M \mapsto \mathcal M_\et$ this follows from the $\ell$-boundedness of $X$, as this implies that $\Gamma(U, -)$ is cohomologically bounded for all basic $U \in X_\proet$), so we can assume that $\mathcal M$ is compact. Then we have
\begin{align*}
	\Hom^\tr(\mathcal P_n, \mathcal M) = \varinjlim_k \Hom(\mathcal P_{n,k}, \mathcal M) = \varinjlim_k \Hom(\mathcal P_{n,k}, \mathcal M_\et) = \Hom(\mathcal P_n, \mathcal M_\et)
\end{align*}
for any representation $\mathcal P_n = \varprojlim_k \mathcal P_{n,k}$ with all $\mathcal P_{n,k}$ being qcqs étale (in the first identity there is no $\ell$-adic completion required because all terms are killed by $\ell$). This finishes the proof that (a), (b) and (c) are equivalent.

We finish the proof of (i): It is formal that $\D_\nuc(X,\Z_\ell)$ is $\omega_1$-compactly generated with the $\omega_1$-compact generators being precisely the basic nuclear objects (see \cite[Theorem 8.6]{condensed-complex-geometry}). We have already shown that all étale sheaves are nuclear, and from (c) it follows immediately that nuclear sheaves are stable under all small colimits in $\D_\solid(X,\Z_\ell)_{\omega_1}$. To show that nuclear sheaves are also stable under tensor products, it is now enough to see this for basic nuclear objects, where it follows from \cref{rslt:trace-class-maps-stable-under-tensor-product} (to get the require right-boundedness, note that the above $\mathcal P'_{n+1}$ are clearly bounded).

We now prove (ii). It is clear that nuclear sheaves are stable under pullback. As in the proof of \cref{rslt:descent-for-w1-solid-sheaves}, the claimed v-descent now reduces to the following claim: Let $f\colon Y \to X$ be a map of $\ell$-bounded spatial diamonds and let $\mathcal M \in \D_\solid(X,\Z_\ell)_{\omega_1}$ such that $f^* \mathcal M$ is nuclear; then $\mathcal M$ is nuclear. Let $f_\bullet\colon Y_\bullet \to X$ be a hypercover of $\ell$-bounded spatial diamonds extending $f$. Then all $f_n^* \mathcal M$ are nuclear. By \cref{rslt:pushforward-for-w1-solid-sheaves} the pushforwards $f_{n\solid*}$ preserve small colimits and Banach sheaves and thus nuclear sheaves, so that we deduce that all $f_{n\solid*} f_n^* \mathcal M$ are nuclear. Moreover, by v-descent for v-sheaves we have $\mathcal M = \varprojlim_{n\in\Delta} f_{n\solid*} f_n^* \mathcal M$. Fix any compact $\mathcal P \in \D_\solid(X,\Z_\ell)_{\omega_1}$. Then by \cref{rslt:IHom-tr-is-bounded} the functor $\Hom^\tr(\mathcal P, -)$ is bounded and therefore preserves Postnikov limits and uniformly left-bounded totalizations. We deduce
\begin{align*}
	\Hom^\tr(\mathcal P, \mathcal M) &= \varprojlim_k \Hom^\tr(\mathcal P, \tau_{\le k} \mathcal M) = \varprojlim_k \Hom^\tr(\mathcal P, \varprojlim_{n\in\Delta} f_{n\solid*} f_n^* \tau_{\le k} \mathcal M)\\
	&= \varprojlim_k \varprojlim_{n\in\Delta} \Hom^\tr(\mathcal P, f_{n\solid*} f_n^* \tau_{\le k} \mathcal M) = \varprojlim_{n\in\Delta} \varprojlim_k \Hom^\tr(\mathcal P, f_{n\solid*} f_n^* \tau_{\le k} \mathcal M),
\intertext{and since $f_{n\solid*}$ has finite cohomological dimension by \cref{rslt:pushforward-for-w1-solid-sheaves} we get}
	&= \varprojlim_{n\in\Delta} \Hom^\tr(\mathcal P, \varprojlim_k f_{n\solid*} f_n^* \tau_{\le k} \mathcal M) = \varprojlim_{n\in\Delta} \Hom^\tr(\mathcal P, f_{n\solid*} f_n^* \mathcal M),
\intertext{and finally by nuclearity of $f_{n\solid*} f_n^* \mathcal M$,}
	&= \varprojlim_{n\in\Delta} \Hom(\mathcal P, f_{n\solid*} f_n^* \mathcal M) = \Hom(\mathcal P, \mathcal M),
\end{align*}
proving that $\mathcal M$ is indeed nuclear.
\end{proof}

\begin{remark}
In the case that $X = \Spa(C)$ is a geometric point, we recover the classical $\infty$-category $\D_\nuc(X,\Z_\ell) = \D_\nuc(\Z_\ell)$ of nuclear condensed $\Z_\ell$-modules. In fact, for any profinite set $X$ (viewed as a diamond over $\Spa(C)$) and all $\mathcal P, \mathcal M \in \D_\solid(X,\Z_\ell)_{\omega_1}$ with $\mathcal P$ compact we have
\begin{align*}
	\IHom_\solid^\tr(\mathcal P, \mathcal M) = \IHom_\solid(\mathcal P, \Z_\ell) \tensor \mathcal M,
\end{align*}
i.e. we recover the classical notion of nuclearity. Note that for non-compact $\mathcal P$ the above identity is false, because by definition $\IHom_\solid^\tr(\mathcal P, -)$ transforms colimits in $\mathcal P$ into limits, whereas $\IHom_\solid(\mathcal P, \Z_\ell) \tensor -$ does not. This does not affect the notion of nuclear sheaves, because this only every uses $\IHom_\solid^\tr(\mathcal P, -)$ for compact $\mathcal P$ (some subtleties arise if compact objects are not stable under tensor products, though).
\end{remark}

While many natural constructions of sheaves preserve the nuclear category, some (like internal hom) do not. We therefore need a way of ``nuclearizing'' sheaves, which can be done as follows.

\begin{definition}
Let $X$ be an $\ell$-bounded spatial diamond. By the adjoint functor theorem the inclusion $\D_\nuc(X,\Z_\ell) \injto \D_\solid(X,\Z_\ell)_{\omega_1}$ admits a right adjoint
\begin{align*}
	(-)_\nuc\colon \D_\solid(X,\Z_\ell)_{\omega_1} \to \D_\nuc(X,\Z_\ell), \qquad \mathcal M \mapsto \mathcal M_\nuc.
\end{align*}
For every $\mathcal M \in \D_\solid(X,\Z_\ell)_{\omega_1}$ we call $\mathcal M_\nuc$ the \emph{nuclearization} of $\mathcal M$.
\end{definition}

\begin{proposition} \label{rslt:properties-of-nuclearization}
Let $X$ be an $\ell$-bounded spatial diamond. Then the nuclearization functor on $X$ is characterized by the following properties:
\begin{propenum}
	\item \label{rslt:nuclearization-preserves-colims} The functor $(-)_\nuc\colon \D_\solid(X,\Z_\ell)_{\omega_1} \to \D_\nuc(X,\Z_\ell)$ preserves all small colimits.

	\item \label{rslt:nuclearization-of-complete-module} If $\mathcal M \in \D_\solid(X,\Z_\ell)_{\omega_1}$ is $\ell$-adically complete, then $\mathcal M_\nuc$ is the $\ell$-adic completion of $\mathcal M_\et$. In particular, if $\mathcal M$ is killed by some power of $\ell$ then $\mathcal M_\nuc = \mathcal M_\et$.

	\item \label{rslt:nuclearization-is-bounded} The functor $(-)_\nuc$ is bounded, i.e. there are integers $a \le b$ such that for every static $\mathcal M \in \D_\solid(X,\Z_\ell)_{\omega_1}$ the sheaf $\mathcal M_\nuc$ lies in homological degrees $[a, b]$ (with respect to the solid $t$-structure).
\end{propenum}
\end{proposition}
\begin{proof}
We construct a colimit-preserving functor
\begin{align*}
	F\colon \D_\solid(X,\Z_\ell)_{\omega_1} \to \D_\nuc(X,\Z_\ell)
\end{align*}
as follows: It is enough to construct it on compact objects and for compact $\mathcal Q \in \D_\solid(X,\Z_\ell)_{\omega_1}$ we define $F(\mathcal Q)$ to be the $\ell$-adic completion of $\mathcal Q_\et$. After composing $F$ with the inclusion $\iota\colon \D_\nuc(X,\Z_\ell) \to \D_\solid(X,\Z_\ell)_{\omega_1}$ we get a natural map $\iota F \to \id$. It follows that there is a natural map $F \to (-)_\nuc$. We claim that this is an isomorphism. This can be checked on sections from basic nuclear objects, i.e. for every $\mathcal M \in \D_\solid(X,\Z_\ell)_{\omega_1}$ and every sequential colimit $\mathcal N = \varinjlim_n \mathcal P_n$ of compact $\mathcal P_n \in \D_\solid(X,\Z_\ell)_{\omega_1}$ with trace-class transition maps, we need to verify that the natural map
\begin{align*}
	\Hom(\mathcal N, F(\mathcal M)) \isoto \Hom(\mathcal N, \mathcal M_\nuc)
\end{align*}
is an isomorphism. Let us first compute the left-hand side. We claim that for any $\mathcal L \in \D_\solid(X,\Z_\ell)_{\omega_1}$ the natural map
\begin{align*}
	\Hom(\mathcal L, \iota F(\mathcal M)) = \Hom^\tr(\mathcal L, \iota F(\mathcal M)) \isoto \Hom^\tr(\mathcal L, \mathcal M)
\end{align*}
is an isomorphism. This can be checked for compact $\mathcal L$; then both sides commute with colimits in $\mathcal M$, so we can also reduce to the case that $\mathcal M$ is compact. Factoring out the $\ell$-adic completion on both sides, we can assume that $\mathcal M$ is killed by $\ell$. But then both sides evaluate to $\varinjlim_n \Hom(\mathcal L_n, \mathcal M)$ for any representation $\mathcal L = \varprojlim_n \mathcal L_n$ with qcqs étale $\mathcal L_n$.

Let us get back to the claimed isomorphism of $F$ and $(-)_\nuc$, so let $\mathcal M$ and $\mathcal N = \varinjlim_n \mathcal P_n$ be as before. The claim now reduces to showing that the natural map
\begin{align*}
	&\varprojlim_n \Hom^\tr(\mathcal P_n, \mathcal M) = \Hom^\tr(\mathcal N, \mathcal M) = \Hom(\mathcal N, F(\mathcal M))\\
	&\qquad\qquad\isoto \Hom(\mathcal N, \mathcal M_\nuc) = \Hom(\mathcal N, \mathcal M) = \varprojlim_n \Hom(\mathcal P_n, \mathcal M)
\end{align*}
is an isomorphism. This follows by constructing sections $\Hom(\mathcal P_{n+1}, \mathcal M) \to \Hom^\tr(\mathcal P_n, \mathcal M)$, which exist by \cref{rslt:trace-class-maps-form-ideal}.

Having established the isomorphism $F = (-)_\nuc$, we can now easily prove the claims (i)--(iii). Part (i) is obvious since $F$ preserves all small colimits by construction. For part (ii) note that if $\mathcal M$ is killed by $\ell$ then clearly $F(\mathcal M) = \mathcal M_\et$. Since $(-)_\nuc$ preserves limits (as it is a right adjoint functor) we deduce that $(-)_\nuc$ preserves $\ell$-adically complete objects. For part (iii), we can immediately reduce to the case that $\mathcal M$ is static and finitely presented and hence $\ell$-adically complete. Then the claim follows from the fact that both $\ell$-adic completion and pushforward to the étale site are bounded.
\end{proof}

We now have a good understanding of nuclear $\Z_\ell$-sheaves on $\ell$-bounded spatial diamonds. It is formal to extend this notion to all small v-stacks and to modules over any nuclear $\Z_\ell$-algebra $\Lambda$, where by ``nuclear $\Z_\ell$-algebra'' we mean an (animated condensed) $\Z_\ell$-algebra whose underlying $\Z_\ell$-module is nuclear. Prominent examples of nuclear $\Z_\ell$-algebras are $\Fld_\ell$, $\Z_\ell$, $\Q_\ell$, $\overline{\Q_\ell}$ and $\Cpx_\ell$.

\begin{definition}
For any small v-stack $X$ we define
\begin{align*}
	\D_\nuc(X,\Z_\ell) \subset \D_\solid(X,\Z_\ell)_{\solid} \subset \D(X_\vsite,\Z_\ell)
\end{align*}
by descent from the case of $\ell$-bounded spatial diamonds using \cref{rslt:descent-for-w1-solid-sheaves,rslt:v-descent-for-nuclear-sheaves} (note that every strictly totally disconnected space is an $\ell$-bounded spatial diamond, so that $\ell$-bounded spatial diamonds form a basis for the v-site). In other words, a $\Z_\ell$-sheaf $\mathcal M \in \D(X_\vsite,\Z_\ell)$ lies in $\D_\nuc(X,\Z_\ell)$ resp. $\D_\solid(X,\Z_\ell)_{\omega_1}$ if and only if this is true after pullback to every $\ell$-bounded spatial diamond.
\end{definition}

\begin{definition}
Let $\Lambda$ be a nuclear $\Z_\ell$-algebra.
\begin{defenum}
	\item By pullback $\Lambda$ defines a v-sheaf of connective $\mathbb E_\infty$-rings on the big v-site of all small v-stacks, which we still denote $\Lambda$. In particular, for every small v-stack $X$, we get a stable $\infty$-category $\D(X_\vsite, \Lambda)$ of $\Lambda$-modules on $X$. This $\infty$-category admits a complete $t$-structure and the forgetful functor $\D(X_\vsite,\Lambda) \to \D(X_\vsite,\Z_\ell)$ is $t$-exact, conservative and preserves all small limits and colimits. Moreover, for any map of small v-stacks, the pullback and pushforward functors on $\Z_\ell$-modules can be enhanced to functors on $\Lambda$-modules.

	\item Let $X$ be a small v-stack. We define
	\begin{align*}
		\D_\nuc(X,\Lambda) \subset \D_\solid(X,\Lambda)_{\omega_1} \subset \D(X_\vsite,\Lambda)
	\end{align*}
	to be the full subcategories spanned by those $\Lambda$-modules which lie in $\D_\nuc(X,\Z_\ell)$ resp. $\D_\solid(X,\Z_\ell)_{\omega_1}$ after applying the forgetful functor. Equivalently we can view $\Lambda$ as an $\mathbb E_\infty$-algebra in $\D_\nuc(X,\Z_\ell)$ and define $\D_\nuc(X,\Lambda)$ as the $\infty$-category of $\Lambda$-modules in $\D_\nuc(X,\Z_\ell)$. The objects of $\D_\nuc(X,\Lambda)$ are called the \emph{nuclear $\Lambda$-modules on $X$} and the objects of $\D_\solid(X,\Lambda)_{\omega_1}$ are called the \emph{$\omega_1$-solid $\Lambda$-modules on $X$}.
\end{defenum}
\end{definition}

We can formally extend most of the results on nuclear and $\omega_1$-solid $\Z_\ell$-modules on $\ell$-bounded spatial diamonds to nuclear and $\omega_1$-solid $\Lambda$-modules on small v-stacks:

\begin{proposition} \label{rslt:properties-of-w1-solid-Lambda-sheaves}
Let $\Lambda$ be a nuclear $\Z_\ell$-algebra.
\begin{propenum}
	\item \label{rslt:v-descent-for-w1-solid-Lambda-sheaves} The assignment
	\begin{align*}
		X \mapsto \D_\solid(X,\Lambda)_{\omega_1}
	\end{align*}
	defines a hypercomplete sheaf of presentable symmetric monoidal stable $\infty$-categories with complete $t$-structure on the v-site of all small v-stacks.

	\item Let $X$ be a small v-stack. The inclusion $\D_\solid(X,\Lambda)_{\omega_1} \injto \D(X_\vsite,\Lambda)$ is $t$-exact and commutes with all pullbacks, colimits and countable limits.

	\item \label{rslt:compact-generators-of-w1-solid-Lambda-sheaves} Let $X$ be an $\ell$-bounded spatial diamond. Then $\D_\solid(X,\Lambda)_{\omega_1}$ is compactly generated. The compact objects are generated under finite (co)limits and retracts by the objects $\mathcal P \tensor_{\Z_\ell} \Lambda$ for compact $\mathcal P \in \D_\solid(X,\Z_\ell)_{\omega_1}$.
\end{propenum}
\end{proposition}
\begin{proof}
Part (i) is clear by definition for $\Lambda = \Z_\ell$ (for presentability use that presentable $\infty$-categories are stable under limits and use (iii)). For general $\Lambda$ it follows by repeating the argument in \cref{rslt:descent-for-w1-solid-sheaves} and noting that everything commutes with the forgetful functor along $\Z_\ell \to \Lambda$. For part (ii) we can similarly reduce to the case $\Lambda = \Z_\ell$. Then the commutation with pullbacks is true by design and the commutation with colimits and countable limits follows from \cref{rslt:stability-of-w1-solid-sheaves} because everything can be reduced to $\ell$-bounded spatial diamonds (using that pullbacks preserve limits and colimits). Part (iii) follows easily from \cref{rslt:compact-generators-of-w1-solid-sheaves}.
\end{proof}

\begin{proposition}
Let $\Lambda$ be a nuclear $\Z_\ell$-algebra.
\begin{propenum}
	\item \label{rslt:v-descent-for-nuclear-Lambda-modules} The assignment
	\begin{align*}
		X \mapsto \D_\nuc(X,\Lambda)
	\end{align*}
	defines a hypercomplete sheaf of presentable symmetric monoidal $\infty$-categories on the v-site of all small v-stacks.

	\item Let $X$ be a small v-stack. The inclusion $\D_\solid(X,\Lambda)_\nuc \injto \D_\solid(X,\Lambda)_{\omega_1}$ is symmetric monoidal and commutes with all pullbacks and colimits.

	\item \label{rslt:nuclear-Lambda-modules-are-w1-compactly-generated} Let $X$ be an $\ell$-bounded spatial diamond. Then $\D_\nuc(X,\Lambda)$ is $\omega_1$-compactly generated. The $\omega_1$-compact objects are generated under countable colimits by the objects $\mathcal N \tensor_{\Z_\ell} \Lambda$ for basic nuclear $\mathcal N \in \D_\nuc(X,\Z_\ell)$.
\end{propenum}
\end{proposition}
\begin{proof}
We can argue as in \cref{rslt:properties-of-w1-solid-Lambda-sheaves} to reduce everything to \cref{rslt:main-results-on-nuclear-Z-ell-sheaves}.
\end{proof}

\begin{proposition}
Let $\Lambda \to \Lambda'$ be a map of nuclear $\Z_\ell$-algebras and let $X$ be a small v-stack.
\begin{propenum}
	\item \label{rslt:def-and-compatibility-of-Lambda-base-change-and-forget} There is a natural pair of adjoint functors
	\begin{align*}
		- \tensor_\Lambda \Lambda'\colon \D_\nuc(X,\Lambda) \rightleftarrows \D_\nuc(X,\Lambda')\noloc \mathrm{Forget},
	\end{align*}
	both of which commute with all pullbacks and colimits. Moreover, the functor $- \tensor_\Lambda \Lambda'$ is symmetric monoidal and the forgetful functor is conservative and preserves all limits.

	The same is true for the $\infty$-categories of $\omega_1$-solid sheaves instead of nuclear sheaves, in which case the forgetful functor is additionally $t$-exact.

	\item \label{rslt:Lambda-base-change-commutes-with-inclusion-of-nuclear-sheaves} Both $- \tensor_\Lambda \Lambda'$ and the forgetful functor commute with the inclusion of nuclear sheaves into $\omega_1$-solid sheaves.
\end{propenum}
\end{proposition}
\begin{proof}
This follows immediately from the fact that $\D_\nuc(X,\Lambda)$ is the $\infty$-category of $\Lambda$-modules in $\D_\nuc(X,\Z_\ell)$ (and similarly for $\D_\solid(X,\Lambda)_{\omega_1}$) and the fact that pullback functors and the inclusion of nuclear sheaves into $\omega_1$-solid sheaves are symmetric monoidal.
\end{proof}

It is convenient to also define the nuclearization functor in the general setting of nuclear $\Lambda$-modules on small v-stacks, even though on general small v-stacks it will not have the same nice properties.

\begin{definition}
Let $\Lambda$ be a nuclear $\Z_\ell$-algebra and $X$ a small v-stack. We define the \emph{nuclearization} functor
\begin{align*}
	(-)_\nuc\colon \D_\solid(X,\Lambda)_{\omega_1} \to \D_\nuc(X,\Lambda)
\end{align*}
to be the right adjoint of the inclusion.
\end{definition}

\begin{proposition}
Let $\Lambda$ be a nuclear $\Z_\ell$-algebra and let $X$ be a small v-stack.
\begin{propenum}
	\item \label{rslt:nuclearization-for-Lambda-modules-preserves-colimits} If $X$ is an $\ell$-bounded spatial diamond then the functor $(-)_\nuc\colon \D_\solid(X,\Lambda)_{\omega_1} \to \D_\nuc(X,\Lambda)$ preserves all small colimits and is bounded with respect to the $t$-structure on $\D_\solid(X,\Lambda)_{\omega_1}$.
	\item If $\Lambda \to \Lambda'$ is a map of nuclear $\Z_\ell$-algebras then nuclearization on $X$ commutes with the forgetful functor along $\Lambda \to \Lambda'$.
\end{propenum}
\end{proposition}
\begin{proof}
Part (ii) follows immediately from \cref{rslt:Lambda-base-change-commutes-with-inclusion-of-nuclear-sheaves} by passing to right adjoints. Part (i) follows from (ii) and \cref{rslt:properties-of-nuclearization} (use the forgetful functor along $\Z_\ell \to \Lambda$).
\end{proof}

In the case that the nuclear $\Z_\ell$-algebra $\Lambda$ is discrete (e.g. $\Lambda = \Fld_\ell$), we recover the classical theory of étale $\Lambda$-modules:

\begin{proposition} \label{rslt:nuclear-over-discrete-Lambda-equiv-etale}
Let $\Lambda$ be a discrete $\Z_\ell$-algebra and $X$ a small v-stack. Then we have
\begin{align*}
	\D_\nuc(X,\Lambda) = \D_\et(X,\Lambda)
\end{align*}
as full subcategories of $\D(X_\vsite,\Lambda)$.
\end{proposition}
\begin{proof}
Since both $\infty$-categories are defined by descent, we can assume that $X$ is an $\ell$-bounded spatial diamond. By \cref{rslt:properties-of-nuclear-subcategory} every étale $\Lambda$-module is nuclear (for any nuclear $\Z_\ell$-algebra $\Lambda$), so we have $\D_\et(X,\Lambda) \subset \D_\nuc(X,\Lambda)$. To get the other inclusion, note that the discreteness of $\Lambda$ implies that $\Hom(\Z_\ell,\Lambda) = \varinjlim_n \Hom(\Z/\ell^n\Z,\Lambda)$ (a priori this holds in the $\infty$-category of $\Z_\ell$-modules, but then also follows in the $\infty$-category of rings). Applying this to the structure map $\Z_\ell \to \Lambda$ we deduce that $\Lambda$ is a $\Z/\ell^n \Z$-algebra for some $n$. In particular every $\mathcal M \in \D_\nuc(X,\Lambda)$ is killed by $\ell^n$, which implies that $\mathcal M$ is étale by \cref{rslt:nuclearization-of-complete-module}.
\end{proof}

\section{Pushforward and Base-Change} \label{sec:pushforward}

Fix a prime $\ell \ne p$ and a nuclear $\Z_\ell$-algebra $\Lambda$. In the previous section we have constructed the $\infty$-category $\D_\nuc(X,\Lambda)$ of nuclear $\Lambda$-modules on every small v-stack $X$. We will now introduce the pushforward of nuclear sheaves and study its behavior.

\begin{definition}
Let $f\colon Y \to X$ be a map of small v-stacks.
\begin{defenum}
	\item We denote the pullback functor on nuclear sheaves by
	\begin{align*}
		f^*\colon \D_\nuc(X,\Lambda) \to \D_\nuc(Y,\Lambda).
	\end{align*}
	It is symmetric monoidal, preserves all small colimits and coincides with the v-pullback by design.

	\item We denote the \emph{nuclear pushforward} functor by
	\begin{align*}
		f_*\colon \D_\nuc(Y,\Lambda) \to \D_\nuc(X,\Lambda).
	\end{align*}
	It is defined to be the right adjoint of $f^*$, which exists by the adjoint functor theorem. It may occasionally be necessary to also consider pushforward functors on $\omega_1$-solid sheaves and on all v-sheaves (defined to be the right adjoints of the pullback functor on the respective $\infty$-categories), which we will denote $f_{\solid*}$ and $f_{\vsite*}$, respectively.
\end{defenum}
\end{definition}

In general we cannot expect the nuclear pushforward to coincide with the v-pushforward because the latter will in general not preserve nuclear sheaves. However, this usually works under a qcqs assumption, as follows.

\begin{definition}
Let $X$ be a small v-stack.
\begin{defenum}
	\item We denote by
	\begin{align*}
		\D^b_\nuc(X,\Lambda), \D^+_\nuc(X,\Lambda), \D^-_\nuc(X,\Lambda) \subset \D_\nuc(X,\Lambda)
	\end{align*}
	the full subcategories of those v-sheaves whose pullback to every $\ell$-bounded spatial diamond is bounded, resp. left bounded, resp. right bounded with respect to the $t$-structure on $\omega_1$-solid sheaves. The elements of these subcategories are called the \emph{locally bounded}, resp. \emph{locally left-bounded}, resp. \emph{locally right-bounded} nuclear $\Lambda$-modules on $X$.

	\item Similar definitions as in (a) apply to $\omega_1$-solid sheaves and to v-sheaves in place of nuclear sheaves.
\end{defenum}
\end{definition}

\begin{lemma}
For any $? \in \{ b, +, - \}$ the assignment $X \mapsto \D^?_\nuc(X,\Lambda)$ defines a hypercomplete sheaf of stable $\infty$-categories on the v-site of all small v-stacks.
\end{lemma}
\begin{proof}
This follows immediately from the fact that pullbacks are $t$-exact and that $\D_\nuc(-,\Lambda)$ is a hypercomplete v-sheaf.
\end{proof}

\begin{remark}
We warn the reader that there is no $t$-structure on $\D_\nuc(X,\Lambda)$ (unless $\Lambda$ is discrete, see \cref{rslt:nuclear-over-discrete-Lambda-equiv-etale}), but it still makes sense to speak of (left/right) bounded nuclear sheaves.
\end{remark}

We can now prove the following characterizations of the nuclear pushforward, generalizing the ones in \cite[Proposition 17.6]{etale-cohomology-of-diamonds}.

\begin{proposition} \label{rslt:v-pushforward-vs-solid-and-nuclear-pushforward}
Let $f\colon Y \to X$ be a qcqs map of small v-stacks.
\begin{propenum}
	\item The v-pushforward $f_{\vsite*}\colon \D(Y_\vsite,\Lambda) \to \D(X_\vsite,\Lambda)$ preserves $\omega_1$-solid sheaves, i.e. we have $f_{\vsite*} = f_{\solid*}$.

	\item The v-pushforward preserves locally left-bounded nuclear sheaves, so that we have $f_{\vsite*} = f_*$ on locally left-bounded sheaves.

	\item \label{rslt:properties-of-pushforward-for-l-bounded-maps} Suppose that the functor $f_{\solid*}\colon \D_\solid(Y,\Z_\ell)_{\omega_1} \to \D_\solid(X,\Z_\ell)_{\omega_1}$ has finite cohomological dimension. Then (ii) also hold for unbounded sheaves. Moreover, in this case both $f_{\solid*}$ and $f_*$ preserve all small colimits.
\end{propenum}
\end{proposition}
\begin{proof}
All statements can be checked after pullback to any v-cover, so we can assume that $X$ is an $\ell$-bounded spatial diamond; in particular $Y$ is qcqs. Moreover, since the forgetful functor along $\Z_\ell \to \Lambda$ commutes with pushforwards we can further reduce to the case that $\Lambda = \Z_\ell$. We fix a hypercover $g_\bullet\colon Y_\bullet \to Y$ by $\ell$-bounded spatial diamonds, so that $f_{\vsite*}$ is computed as the totalization $f_{\vsite*} = \varprojlim_{n\in\Delta} f_{n\vsite*} g_n^*$.

Part (i) now follows from that fact that all $f_{n\vsite*}$ preserve $\omega_1$-solid sheaves (see \cref{rslt:pushforward-for-w1-solid-sheaves} and that $\omega_1$-solid sheaves on $X$ are stable under countable limits (see \cref{rslt:stability-of-w1-solid-sheaves}). Part (ii) follows by the same argument as in the proof of \cref{rslt:v-descent-for-nuclear-sheaves} by exploiting the fact that $\Hom^\tr(\mathcal P, -)$ preserves uniformly left-bounded totalizations. Part (iii) follows similarly by additionally taking limits over Postnikov truncations (the fact that $f_{\solid*}$ and $f_*$ commute with colimits follows easily from the facts that $f$ is qcqs and has bounded cohomological dimension by the usual Postnikov limit argument).
\end{proof}

\begin{corollary} \label{rslt:base-change-for-solid-and-nuclear-pushforward}
Let
\begin{center}\begin{tikzcd}
	Y' \arrow[r,"g'"] \arrow[d,"f'"] & Y \arrow[d,"f"]\\
	X' \arrow[r,"g"] & X
\end{tikzcd}\end{center}
be a cartesian diagram of small v-stacks and assume that $f$ is qcqs.
\begin{corenum}
	\item \label{rslt:base-change-for-solid-pushforward} The natural morphism $g^* f_{\solid*} \isoto f'_{\solid*} g'^*$ is an isomorphism of functors $\D_\solid(Y,\Lambda)_{\omega_1} \to \D_\solid(X',\Lambda)_{\omega_1}$.

	\item \label{rslt:left-bounded-base-change-for-nuclear-pushforward} The natural morphism $g^* f_* \isoto f'_* g'^*$ is an isomorphism of functors $\D^+_\nuc(Y,\Lambda) \to \D^+_\nuc(X',\Lambda)$.

	\item \label{rslt:base-change-for-unbounded-nuclear-pushforward} Suppose that the functor $f_{\solid*}\colon \D_\solid(Y,\Z_\ell)_{\omega_1} \to \D_\solid(X,\Z_\ell)_{\omega_1}$ has finite cohomological dimension. Then (ii) also holds for unbounded sheaves.
\end{corenum}
\end{corollary}
\begin{proof}
All claims follows immediately from \cref{rslt:v-pushforward-vs-solid-and-nuclear-pushforward} by using the fact that v-pushforward satisfies arbitrary base-change (for (iii) note additionally that the functor $f'_{\vsite*} g'^*$ preserves nuclear sheaves because so does $g^* f_{\vsite*}$).
\end{proof}

The condition that $f_{\solid*}$ has finite cohomological dimension may look a bit hard to grasp at first, but it turns out that this condition is usually satisfied in practice:

\begin{proposition} \label{rslt:fdcs-map-has-bounded-pushforward}
Let $f\colon Y \to X$ be a map of small v-stacks which is locally compactifiable, representable in spatial diamonds and has locally finite $\dimtrg$. Then $f_{\solid*}\colon \D_\solid(Y,\Lambda)_{\omega_1} \to \D_\solid(X,\Lambda)_{\omega_1}$ has finite cohomological dimension.
\end{proposition}
\begin{proof}
By \cref{rslt:base-change-for-solid-pushforward} we can assume that $X$ is an $\ell$-bounded spatial diamond; in particular $Y$ is also a spatial diamond. By passing to a finite cover of $Y$ we can assume that $f$ is compactifiable. Then \cite[Theorem 22.5]{etale-cohomology-of-diamonds} implies that $f_*\colon \D_\et(Y,\Fld_\ell) \to \D_\et(X,\Fld_\ell)$ has finite cohomological dimension, which implies that $Y$ is $\ell$-bounded. We conclude by \cref{rslt:pushforward-for-w1-solid-sheaves}.
\end{proof}

\section{The 6-Functor Formalism} \label{sec:6functors}

Fix a prime $\ell \ne p$ and a nuclear $\Z_\ell$-algebra $\Lambda$. We will finally construct the 6-functor formalism for nuclear $\Lambda$-sheaves. In the previous subsection we have already introduced the pullback and pushforward functors. The next pair of functors are tensor product and internal $\Hom$, which are easily defined:

\begin{definition}
Let $X$ be a small v-stack. We denote by $\tensor_\Lambda$ the symmetric monoidal structure on $\D_\nuc(X,\Lambda)$ and by $\IHom_\Lambda(-,-)$ the associated internal hom functor (which exists because $\D_\nuc(X,\Lambda)$ is presentable). We often drop the subscript $\Lambda$ if there is no room for confusion.
\end{definition}

The last pair of functors -- the shriek functors -- are the hardest to construct. The general idea is as follows: Given a nice compactifiable map $f\colon Y \to X$ of small v-stacks we want to define the functor $f_!\colon \D_\nuc(Y,\Lambda) \to \D_\nuc(X,\Lambda)$ as the composition $f_! = g_* \comp j_!$ for any decomposition of $f$ into an open immersion $j$ and a proper map $g$; here $j_!$ will be left adjoint of $j^*$.

As a first step towards constructing $f_!$, let us show that $j_!$ exists and satisfies the expected properties:

\begin{lemma} \label{rslt:properties-of-etale-lower-shriek}
For every étale map $j\colon U \to X$ of small v-stacks the pullback $j^*\colon \D_\solid(X,\Lambda)_{\omega_1} \to \D_\solid(U,\Lambda)_{\omega_1}$ admits a left adjoint
\begin{align*}
	j_!\colon \D_\solid(U,\Lambda)_{\omega_1} \to \D_\solid(X,\Lambda)_{\omega_1}
\end{align*}
with the following properties:
\begin{lemenum}
	\item For every map $g\colon X' \to X$ of small v-stacks with base-change $g'\colon U' := U \cprod_X X' \to U$, $j'\colon U' \to X'$ the natural morphism
	\begin{align*}
		j'_! g'^* \isoto g^* j_!
	\end{align*}
	is an isomorphism of functors $\D_\solid(U,\Lambda)_{\omega_1} \to \D_\solid(X',\Lambda)_{\omega_1}$.

	\item For all $\mathcal M \in \D_\solid(X,\Lambda)_{\omega_1}$ and $\mathcal N \in \D_\solid(U,\Lambda)_{\omega_1}$ the natural map
	\begin{align*}
		j_!(\mathcal N \tensor j^* \mathcal M) \isoto j_! \mathcal N \tensor \mathcal M
	\end{align*}
	is an isomorphism.

	\item \label{rslt:etale-lower-shriek-is-functorial-in-Lambda} $j_!$ commutes with the forgetful functor and the base-change functor along any map $\Lambda \to \Lambda'$ of nuclear $\Z_\ell$-algebras.

	\item If $j$ is quasicompact then $j_!$ is $t$-exact and preserves $\ell$-adically complete sheaves.

	\item $j_!$ preserves nuclear sheaves and hence restricts to a functor
	\begin{align*}
		j_!\colon \D_\nuc(U,\Lambda) \to \D_\nuc(X,\Lambda)
	\end{align*}
	which is left adjoint to $j^*$.
\end{lemenum}
\end{lemma}
\begin{proof}
By \cite[Proposition VII.3.1]{fargues-scholze-geometrization} the pullback functor $j^*\colon \D_\solid(X,\Lambda) \to \D_\solid(U,\Lambda)$ admits a left adjoint $j_\natural$ which satisfies the analogous properties (i), (ii) and (iii) (the result in loc. cit. is only stated for static $\Lambda$ but the proof works in general; in fact, the proof is completely formal). We now show that $j_\natural$ preserves $\omega_1$-solid sheaves and hence restricts to the desired functor $j_!$. Since $j_\natural$ commutes with the forgetful functor along $\Z_\ell \to \Lambda$, we can assume that $\Lambda = \Z_\ell$. Moreover, since $j_\natural$ commutes with any base-change, we can assume that $X$ is a strictly totally disconnected space. Then $U$ is a perfectoid space and hence admits a basis given by open subsets which are quasicompact and separated over $X$ (e.g. take affinoid perfectoid open subsets). It is formal that $j_\natural$ is computed as the colimit over this basis, so we can assume that $j$ is quasicompact and separated. In particular $U$ is strictly totally disconnected and hence an $\ell$-bounded spatial diamond. Since $j_\natural$ preserves all small colimits, it is now enough to verify that for every compact $\mathcal P \in \D_\solid(U,\Z_\ell)_{\omega_1}$ we have $j_\natural \mathcal P \in \D_\solid(X,\Z_\ell)_{\omega_1}$. By \cref{rslt:compact-generators-of-w1-solid-sheaves} we can assume that $\mathcal P = \Z_{\ell,\solid}[V]$ for some countable basic $V = \varprojlim_n V_n \in X_\proet$. But then for all $\mathcal M \in \D_\solid(X,\Z_\ell)$ we have
\begin{align*}
	\Hom(j_\natural \Z_{\ell,\solid}[V], \mathcal M) = \Hom(\Z_{\ell,\solid}[V], j^* \mathcal M) = \Gamma(V, j^* \mathcal M) = \Gamma(V, \mathcal M),
\end{align*}
hence $j_\natural \Z_{\ell,\solid}[V] = \Z_{\ell,\solid}[V]$ (where on the right-hand side we define it as an sheaf on $X$), which is indeed $\omega_1$-solid.

We have now shown the existence of $j_!$ by restricting $j_\natural$ to $\omega_1$-solid sheaves. The claims (i), (ii) and (iii) now follow immediately from the analogous properties of $j_\natural$. It remains to prove (iv) and (v), for which we can assume that $\Lambda = \Z_\ell$ (by (iii)), that $X$ is strictly totally disconnected (by (i)) and that $j$ is quasicompact and separated (by passing to an open cover as above).

We first observe that $j_!$ preserves étale sheaves: By passing to right adjoints, this reduces to the observation that $j^*$ commutes with the pushforward to the étale site, which follows immediately from the fact that $j$ is étale. Therefore claim (v) reduces to claim (iv), because every nuclear $\Z_\ell$-sheaf on $U$ is a colimit of Banach sheaves. To prove (iv) we need to make one more computation: Suppose $\mathcal P \in \D_\solid(U,\Z_\ell)_{\omega_1}$ is static and can be written as a sequential limit $\mathcal P = \varprojlim_n \mathcal P_n$ with static qcqs étale $\mathcal P_n \in \D_\et(U,\Z_\ell)^\omega$; then
\begin{align}
	j_!\mathcal P = \varprojlim_n j_! \mathcal P_n. \label{eq:j-shriek-for-limit-of-qcqs-etale}
\end{align}
To prove this, pick a surjective map $\mathcal Q_0 \surjto \mathcal P$ for some static compact $\omega_1$-solid $\Z_\ell$-module $\mathcal Q_0 $ on $U$. We can write $\mathcal Q_0 = \varprojlim_n \mathcal Q_{0,n}$ for some qcqs étale sheaves $\mathcal Q_{0,n}$ and by the usual Breen resolution argument the map $\mathcal Q_0 \to \mathcal P$ can be obtained from a map of $\Pro$-systems $(\mathcal Q_{0,n})_n \to (\mathcal P_n)_n$. In particular the kernel of the map $\mathcal Q_0 \to \mathcal P$ is again a limit of qcqs étale sheaves, so we can iterate the process in order to obtain a resolution of the $\Pro$-system $(\mathcal P_n)_n$ in terms of $\Pro$-systems $(\mathcal Q_{k,n})_n$ of static qcqs étale sheaves such that each $\mathcal Q_k := \varprojlim_n \mathcal Q_{k,n}$ is compact. By passing to the associated simplicial objects (whose geometric realization is $\mathcal P$) and using that both $j_!$ and countable limits commute with geometric realizations, we reduce the claim \cref{eq:j-shriek-for-limit-of-qcqs-etale} to the case that $(\mathcal P_n)_n = (\mathcal Q_{k,n})_n$ for some $k$. But since $\mathcal Q_k$ is compact, the claim follows from the above computation of $j_!$ on compact generators.

With \cref{eq:j-shriek-for-limit-of-qcqs-etale} at hand, claim (iv) is now straightforward: Note that it follows immediately that $j_!$ is $t$-exact, because every static $\omega_1$-solid sheaf can be written as a filtered colimits of objects of the form $\mathcal P$ in \cref{eq:j-shriek-for-limit-of-qcqs-etale} (use also \cite[Proposition VII.1.6]{fargues-scholze-geometrization}). Thus to show that $j_!$ preserves $\ell$-adically complete sheaves, we can immediately reduce to the static case. From now on let $\mathcal M \in \D_\solid(U,\Z_\ell)_{\omega_1}$ be static and $\ell$-adically complete. To show that $j_!\mathcal M$ is $\ell$-adically complete, we can as in the proof of \cref{rslt:solid-tensor-product-preserves-complete-sheaves} pass to a geometric realization of $\mathcal M$ in order to reduce to the case that $\mathcal M = \cplbigdsum{k} \Z_{\ell,\solid}[V_k]$ for countably many basic $V_k = \varprojlim_n V_{k,n} \in U_\proet$ (here we use that $j_!$ preserves colimits and is right $t$-exact). Now filter $\mathcal M$ by the subsheaves $\mathcal M_\alpha = \prod_k \ell^{\alpha_k} \Z_{\ell,\solid}[V_k]$ for sequences of integers $\alpha_k \ge 0$ converging to $\infty$. By \cref{eq:j-shriek-for-limit-of-qcqs-etale} we get
\begin{align*}
	j_! \mathcal M = \varinjlim_\alpha j_! \prod_k \ell^{\alpha_k} \Z_{\ell,\solid}[V_k] = \varinjlim_\alpha \prod_k j_!(\ell^{\alpha_k} \Z_{\ell,\solid}[V_k]) = \cplbigdsum{k} \ell^{\alpha_k} j_!(\Z_{\ell,\solid}[V_k]),
\end{align*}
as desired.
\end{proof}

Having a good understanding of $j_!$ for étale $j$, it remains to study $g_*$ for nice proper maps $g$. We already know base-change for $g_*$ by \cref{rslt:base-change-for-unbounded-nuclear-pushforward}. It remains to check the projection formula. Also, in order to show that the construction $f_! = g_* \comp j_!$ is independent of the factorization, we need a compatibility of $g_*$ and $j_!$:

\begin{lemma}
Let $g\colon Y \to X$ be a proper map of small v-stacks which has locally bounded dimension (see \cite[Definition 3.5.3]{mann-mod-p-6-functors}) and assume that the functor $g_{\solid*}\colon \D_\solid(Y,\Z_\ell) \to \D_\solid(X,\Z_\ell)$ has finite cohomological dimension. Then:
\begin{lemenum}
	\item \label{rslt:projection-formula-for-bounded-proper-pushforward} For all $\mathcal M \in \D_\nuc(X,\Lambda)$ and $\mathcal N \in \D_\nuc(Y,\Lambda)$ the natural map
	\begin{align*}
		g_* \mathcal N \tensor \mathcal M \isoto g_*(\mathcal N \tensor g^* \mathcal M)
	\end{align*}
	is an isomorphism.

	\item \label{rslt:compatibility-of-etale-shriek-and-proper-pushforward} Let $j\colon U \to X$ be an open immersion with base-change $j'\colon V \to Y$ and $g'\colon V \to U$. Then the natural map
	\begin{align*}
		j_! g'_* \isoto g_* j'_!
	\end{align*}
	is an isomorphism of functors $\D_\nuc(V,\Lambda) \to \D_\nuc(X,\Lambda)$.
\end{lemenum}
\end{lemma}
\begin{proof}
Part (ii) follows formally from part (i) together with \cref{rslt:properties-of-etale-lower-shriek} and \cref{rslt:base-change-for-unbounded-nuclear-pushforward}, see \cite[Lemma 3.6.8]{mann-mod-p-6-functors}. To prove (i), we note that both sides of the claimed isomorphism commute with base-change (by \cref{rslt:base-change-for-unbounded-nuclear-pushforward}), so we can assume that $X$ is a strictly totally disconnected space. Moreover, both sides of the claimed isomorphism also commute with colimits in $\mathcal M$, so we can assume $\mathcal M = \mathcal M' \tensor_{\Z_\ell} \Lambda$ for some basic nuclear $\mathcal M' \in \D_\nuc(X,\Z_\ell)$ (see \cref{rslt:nuclear-Lambda-modules-are-w1-compactly-generated}). But then after applying the forgetful functor along $\Z_\ell \to \Lambda$ (and using that $g_*$ commutes with this forgetful functor) we get
\begin{align*}
	g_* \mathcal N \tensor_\Lambda \mathcal M &= g_* \mathcal N \tensor_\Lambda (\Lambda \tensor_{\Z_\ell} \mathcal M') = g_* \mathcal N \tensor_{\Z_\ell} \mathcal M',\\
	g_*(\mathcal N \tensor_\Lambda g^* \mathcal M) &= g_*(\mathcal N \tensor_\Lambda (\Lambda \tensor_{\Z_\ell} g^* \mathcal M')) = g_*(\mathcal N \tensor_{\Z_\ell} g^* \mathcal M').
\end{align*}
We can therefore assume that $\Lambda = \Z_\ell$. Now write $\mathcal M = \varinjlim_i \mathcal P_i$ for compact objects $\mathcal P_i \in \D_\solid(X,\Z_\ell)_{\omega_1}$. Then by \cref{rslt:nuclearization-preserves-colims} we have $\mathcal M = \varinjlim_i (\mathcal P_i)_\nuc$ and by the other claims in \cref{rslt:properties-of-nuclearization} each $(\mathcal P_i)_\nuc$ is a bounded Banach sheaf. By pulling out colimits on both sides of the claimed isomorphism we can thus reduce to the case that $\mathcal M$ is a bounded Banach sheaf.

We will first prove the claim in the case that $Y$ is a spatial diamond (it is automatically $\ell$-bounded by \cite[Theorem 22.5]{etale-cohomology-of-diamonds}). Then, since both sides of the claimed isomorphism commute with colimits in $\mathcal N$ (for $g_*$ this was shown in \cref{rslt:properties-of-pushforward-for-l-bounded-maps}), we can use the same strategy as for $\mathcal M$ in order to reduce to the case that $\mathcal N$ is also a bounded Banach sheaf. Now by \cref{rslt:solid-tensor-product-preserves-complete-sheaves} both sides of the claimed isomorphism are $\ell$-adically complete, hence we can check the isomorphism after reducing modulo $\ell$. But then the claim follows immediately from \cite[Proposition 22.11]{etale-cohomology-of-diamonds}.

To prove the claim for general $Y$, pick a hypercover $h_\bullet\colon Y_\bullet \to Y$ such that all $Y_n$ are spatial diamonds and all maps $g_n\colon Y_n \to X$ are proper and of finite $\dimtrg$ (e.g. start with any hypercover of $Y$ by affinoid perfectoid spaces of finite $\dimtrg$ over $X$ and then take relative compactifications). Note that the functor
\begin{align*}
	F\colon \D_\solid(Y,\Z_\ell)_{\omega_1} \to \D_\solid(X,\Z_\ell)_{\omega_1}, \qquad \mathcal N' \mapsto g_* \mathcal N'_\nuc \tensor \mathcal M
\end{align*}
is bounded. Indeed, it follows formally from adjunctions that $g_* \mathcal N'_\nuc = (g_{\solid*} \mathcal N')_\nuc$, so it is enough to see that the functors $g_{\solid*}$, $(-)_\nuc$ (on $X$) and $- \tensor \mathcal M$ are bounded. For the first functor this holds by assumption, for the second functor this was shown in \cref{rslt:nuclearization-is-bounded} and for the third functor it follows from the boundedness of $\mathcal M$ together with \cite[Proposition VII.2.3]{fargues-scholze-geometrization}. By a similar argument one sees that the composition of functors $F \comp h_{n\solid*} h_n^*$ is bounded (this boils down to the functor $g_{n\solid*}$ having finite cohomological dimension). Thus by taking Postnikov limits and commuting them with totalizations, we deduce that $F(\varprojlim_{n\in\Delta} h_{n\solid*} h_n^* \mathcal N') = \varprojlim_{n\in\Delta} F(h_{n\solid*} h_n^* \mathcal N')$ for all $\mathcal N'$. Taking $\mathcal N' = \mathcal N$ we get
\begin{align*}
	g_* \mathcal N \tensor \mathcal M &= F(\mathcal N) = \varprojlim_{n\in\Delta} (g_{n*} h_n^* \mathcal N \tensor \mathcal M),
\intertext{and using the fact that we already know that the projection formula holds for each $g_n$ in place of $g$ by the above argument,}
	&= \varprojlim_{n\in\Delta} g_{n*} (h_n^* \mathcal N \tensor g_n^* \mathcal M) = g_{\solid*} \varprojlim_{n\in\Delta} h_{n\solid*} h_n^* (\mathcal N \tensor g^* \mathcal M).
\end{align*}
But by v-descent for $\omega_1$-solid sheaves we have $\varprojlim_{n\in\Delta} h_{n\solid*} h_n^* \mathcal N' = \mathcal N'$ for every $\mathcal N' \in \D_\solid(Y,\Z_\ell)_{\omega_1}$. By applying this to $\mathcal N' = \mathcal N \tensor g^* \mathcal M$ we conclude
\begin{align*}
	g_* \mathcal N \tensor \mathcal M &= g_{\solid*} \varprojlim_{n\in\Delta} h_{n\solid*} h_n^* (\mathcal N \tensor g^* \mathcal M) = g_*(\mathcal N \tensor g^* \mathcal M),
\end{align*}
as desired.
\end{proof}

With the above results at hand, it is now formal to construct the 6-functor formalism. We first construct it for maps which are representable in locally spatial diamonds and afterwards extend it to certain ``stacky'' maps. In the case of locally spatial diamonds, we will define shriek functors for the following class of maps:

\begin{definition}
A map $f\colon Y \to X$ of small v-stacks is called \emph{fdcs}\footnote{The name comes from ``Finite Dimension, Compactifiable and Spatial''.} if it is locally compactifiable (i.e. there is some analytic cover of $Y$ on which $f$ is compactifiable) and representable in locally spatial diamonds and has locally finite $\dimtrg$.
\end{definition}

\begin{lemma} \label{rslt:properties-of-fdcs-maps}
\begin{lemenum}
	\item The property of being fdcs is analytically local on both source and target.
	\item Fdcs maps are stable under composition and base-change.
	\item Every étale map is fdcs.
	\item \label{rslt:fdcs-maps-have-2-out-of-3} Let $f\colon Y \to X$ and $g\colon Z \to Y$ be maps of small v-stacks. If $f$ and $f \comp g$ are fdcs then so is $g$.
\end{lemenum}
\end{lemma}
\begin{proof}
Claims (i), (ii) and (iii) are obvious (using \cite[Proposition 22.3]{etale-cohomology-of-diamonds} to handle the compactifiability condition). For (iv) we can argue in the same way as for the similar bdcs condition, see \cite[Lemma 3.6.10.(iv)]{mann-mod-p-6-functors}.
\end{proof}

The next result constructs the 6-functor formalism for fdcs maps. The result will freely make use of the theory of abstract 6-functor formalisms developed in \cite[\S A.5]{mann-mod-p-6-functors}, so the reader is advised to take a look at that. In particular, recall the definition of the $\infty$-operad $\Corr(\mathcal C)_{E,all}$ of correspondences and of 6-functor formalisms, see \cite[Definitions A.5.2.(b), A.5.4, A.5.7]{mann-mod-p-6-functors}. Also, in the following we will denote by $\vStacks$ the 2-category of small v-stacks.

\begin{proposition} \label{rslt:6-functor-formalism-for-diamonds}
There is a 6-functor formalism
\begin{align*}
	\D_\nuc(-,\Lambda)\colon \Corr(\vStacks)_{fdcs,all} \to \infcatinf
\end{align*}
with the following properties:
\begin{propenum}
	\item Restricted to the symmetric monoidal subcategory $\vStacks^\opp$ (equipped with the coproduct monoidal structure), $\D_\nuc(-,\Lambda)$ coincides with the functor constructed in \cref{rslt:v-descent-for-nuclear-Lambda-modules}.

	\item For every fdcs map $f\colon Y \to X$ the functor
	\begin{align*}
		f_! := \D_\nuc([Y \xfrom{\id} Y \xto{f} X],\Lambda)\colon \D_\nuc(Y,\Lambda) \to \D_\nuc(X,\Lambda)
	\end{align*}
	preserves all small colimits. If $f = j$ is étale then $j_!$ is left adjoint to $j^*$ and if $f$ is proper then $f_! = f_*$.
\end{propenum}
\end{proposition}
\begin{proof}
The construction is very similar to \cite[Theorem 3.6.12]{mann-mod-p-6-functors}; in fact, we copied most of the proof with only minor adjustments. The most prominent difference is that we make use of the notion of prespatial diamonds from \cite{mod-ell-stacky-6-functors} to get a replacement for $p$-boundedness.

We first construct the 6-functor formalism for morphisms in the collection $fdcqc$ of those maps $f\colon Y \to X$ which are representable in prespatial diamonds (see \cite[Definition 3.3]{mod-ell-stacky-6-functors}), compactifiable and quasicompact and have finite $\dimtrg$. One checks that the class of maps $fdcqc$ satisfies the ``2-out-of-3 property'' (i.e. the analog og \cref{rslt:fdcs-maps-have-2-out-of-3}) by a similar argument as in the proof of \cref{rslt:fdcs-maps-have-2-out-of-3} (note that it follows easily from the definition of prespatial diamonds that a quasicompact subdiamond of a prespatial diamond is prespatial). We denote by $P \subset fdcqc$ the subclass of those maps which are additionally proper and by $I \subset fdcqc$ the class of quasicompact open immersions. It follows easily from the 2-out-of-3 property for $fdcqc$ that both $P$ and $I$ also have the 2-out-of-3 property. Thus by \cite[Proposition 3.6]{mod-ell-stacky-6-functors} the pair $I,P \subset fdcqc$ is a suitable decomposition in the sense of \cite[Definition A.5.9]{mann-mod-p-6-functors}. We can therefore apply \cite[Proposition A.5.10]{mann-mod-p-6-functors}: Condition (a) follows from \cref{rslt:properties-of-etale-lower-shriek}, condition (b) follows from \cref{rslt:base-change-for-unbounded-nuclear-pushforward,rslt:projection-formula-for-bounded-proper-pushforward} using \cite[Proposition 3.7]{mod-ell-stacky-6-functors} (applied as in \cref{rslt:fdcs-map-has-bounded-pushforward}), condition (c) follows from \cref{rslt:compatibility-of-etale-shriek-and-proper-pushforward}, conditions (d) and (e) are clear because $\D_\nuc(X,\Lambda)$ is presentable and condition (f) reduces to the observation that for $f \in P$ the functor $f_*$ preserves small colimits by \cref{rslt:properties-of-pushforward-for-l-bounded-maps}. We obtain a 6-functor formalism
\begin{align*}
	\D_{qc}\colon \Corr(\vStacks)_{fdcqc,all} \to \infcatinf
\end{align*}
mapping $X \in \vStacks$ to $\D_\nuc(X,\Lambda)$ with the correct pullback functors. We will now extend $\D_{qc}$ from $fdcqc$ to $fdcs$ by using the extension results in \cite[\S A.5]{mann-mod-p-6-functors}.
\begin{enumerate}[1.]
	\item In the first step, we first restrict $\D_{qc}$ to $\Corr(\vStacks)_{fdcsqc,all}$, where $fdcsqc \subset fdcqc$ is the subset of those edges which are representable in spatial diamonds. We now wish to extend $\D_{qc}$ from $fdcsqc$ to the class of edges $fdcss$ consiting of those fdcs maps which are separated. By \cite[Proposition A.5.16]{mann-mod-p-6-functors} and \cref{rslt:v-descent-for-nuclear-Lambda-modules} this extension can be performed on the full subcategory $\mathcal C_1 \subset \vStacks$ consisting of separated locally spatial diamonds, i.e. we need to extend the $\infty$-operad map
	\begin{align*}
		\D_{qc}\colon \Corr(\mathcal C_1)_{fdcsqc,all} \to \infcatinf
	\end{align*}
	to an $\infty$-operad map
	\begin{align*}
		\D_s\colon \Corr(\mathcal C_1)_{fdcss,all} \to \infcatinf.
	\end{align*}
	We first apply \cite[Proposition A.5.12]{mann-mod-p-6-functors}, which allows us to extend $\D_{qc}$ from $fdcsqc$ to the collection $E_1$ of edges of the form $\bigdunion_i Y_i \to X$, where each $Y_i \to X$ lies in $fdcsqc$; let us denote the new 6-functor formalism by $\D'_{qc}$. We now apply \cite[Proposition A.5.14]{mann-mod-p-6-functors} to extend $\D = \D'_{qc}$ from $E_1$ to $E'_1 := fdcss$. Here we use the collection $S_1 \subset E_1$ of edges of the form $\bigdunion_i U_i \to X$ for covers $X = \bigunion_i U_i$ by quasicompact open immersions $U_i \injto X$. Then for $j \in S_1$ we have $j^! = j^*$, hence condition (b) of loc. cit. follows from the sheafiness of $\D(-)$. Condition (c) amounts to saying that every separated fdcs map $Y \to X$ of separated locally spatial diamonds admits a cover $Y = \bigunion_i V_i$ by quasicompact open immersions $V_i \injto Y$ such that each map $V_i \to X$ is quasicompact (it is automatically compactifiable by \cite[Proposition 22.3.(v)]{etale-cohomology-of-diamonds}); but this is easily satisfied, e.g. pick the $V_i$ to be any open cover of $Y$ by quasicompact open subsets (then the maps $V_i \injto Y$ and $V_i \to X$ are quasicompact because both $X$ and $Y$ are separated). Finally, condition (d) follows easily from the fact that all the spaces in $\mathcal C_1$ are separated. This finishes the construction of the 6-functor formalism
	\begin{align*}
		\D_s\colon \Corr(\vStacks)_{fdcss,all} \to \infcatinf
	\end{align*}
	(where we implicitly used \cite[Proposition A.5.16]{mann-mod-p-6-functors} to extend from $\mathcal C_1$ to $\vStacks$).

	\item In the second extension step we extend $\D_s$ to the desired $\infty$-operad map
	\begin{align*}
		\D\colon \Corr(\vStacks)_{fdcs,all} \to \infcatinf.
	\end{align*}
	This extension is similar to the previous one, albeit somewhat simpler: We can perform the extension directly on $\mathcal C_2 = \vStacks$ by applying \cite[Proposition A.5.14]{mann-mod-p-6-functors} to $E_2 = fdcss$ and $E_2' = fdcs$ with $S_2 \subset E_2$ being the collection of all open immersions.
\end{enumerate}
We have now constructed a 6-functor formalism $\D(-,\Lambda) = \D$ on fdcs maps. It remains to verify that it satisfies claims (i) and (ii). Claim (i) is obvious from the construction. For claim (ii), it follows from the definition of 6-functor formalisms that $f_!$ preserves all small colimits and it follows immediately from the construction that $f_! = f_*$ for proper $f$. It remains to see that for every étale map $j\colon U \to X$ of small v-stacks the just constructed functor $j_!$ is left adjoint to $j^*$ and thus coincides with the functor $j_!$ from \cref{rslt:properties-of-etale-lower-shriek}. To see this, we first apply \cite[Proposition A.5.10]{mann-mod-p-6-functors} to the case that $E = I = et$ is the collection of étale maps in $\vStacksCoeff$ and $P$ consists only of degenerate edges; then conditions (b) and (c) are vacuous and condition (a) is satisfied by \cref{rslt:properties-of-etale-lower-shriek}. We thus obtain a 6-functor formalism
\begin{align*}
	\D_\et\colon \Corr(\vStacks)_{et,all} \to \infcatinf, \qquad X \mapsto \D_\nuc(X,\Lambda).
\end{align*}
We need to show that $\D_\et$ is equivalent to the restriction of $\D$ to $\Corr(\vStacks)_{et,all}$. By the uniqueness of the extension results \cite[Proposition A.5.12, A.5.14, A.5.16]{mann-mod-p-6-functors} we can show this equivalence on the full subcategory $\mathcal C \subset \vStacks$ consisting of locally spatial diamonds and we can then restrict to the subset $etsqc \subset et$ of separated quasicompact étale maps. We can now further reduce to the full subcategory $\mathcal C' \subset \mathcal C$ consisting of strictly totally disconnected spaces. But note that every map of strictly totally disconnected spaces which lies in $etsqc$ is of the form $\bigdunion_{i=1}^n U_i \to X$ for quasicompact open immersions $U_i \injto X$, so we can further replace $etsqc$ by the collection of quasicompact open immersions. But in this case $\D_\et$ and $\D$ agree by construction.
\end{proof}

We now want to extend the 6-functor formalism from \cref{rslt:6-functor-formalism-for-diamonds} to certain ``stacky'' maps. The relevant definitions are as follows:

\begin{definition}
We say that an fdcs map $f\colon Y \to X$ of small v-stacks admits \emph{universal $\ell$-codescent} if it satisfies the following property: Given any small v-stack $X'$ which admits a map to some strictly totally disconnected space and given any map $X' \to X$ with base-change $f'\colon Y' \to X'$ and Čech nerve $Y'_\bullet \to X'$ the natural functor
\begin{align*}
	\D_\nuc^!(X',\Z_\ell) \isoto \varprojlim_{n\in\Delta} \D_\nuc^!(Y'_n,\Z_\ell)
\end{align*}
is an equivalence. Here $\D_\nuc^!(-,\Lambda)$ denotes the functor $Z \mapsto \D_\nuc(Z,\Lambda)$, $h \mapsto h^!$ obtained from the 6-functor formalism in \cref{rslt:6-functor-formalism-for-diamonds}.
\end{definition}

\begin{definition} \label{def:ell-fine-maps}
A map $f\colon Y \to X$ of small v-stacks is called \emph{$\ell$-fine} if there is a map $g\colon Z \to Y$ such that $g$ and $f \comp g$ are fdcs and $g$ admits universal $\ell$-codescent.
\end{definition}

\begin{remark}
There is a definition of \emph{fine} maps in \cite[Definition 1.3.i]{mod-ell-stacky-6-functors} which is closely related to our definition of $\ell$-fine maps. In fact, every fine map is $\ell$-fine, as follows from \cref{rslt:smooth-maps-admit-codescent} below. Moreover, we introduced a similar notion of \emph{$p$-fine} maps in \cite[Definition 2.4]{mod-p-stacky-6-functors} which only takes into account the universal codescent for $p$-\emph{torsion} coefficients -- this seems to be a mistake, as it is probably not enough to imply universal codescent for $p$-adic non-torsion coefficients (once the corresponding 6-functor formalism has been worked out).
\end{remark}

\begin{lemma} \label{rslt:stabilities-of-ell-fine-maps}
\begin{lemenum}
	\item The condition of being $\ell$-fine is étale local on source and target.
	\item The collection of $\ell$-fine maps is stable under composition and base-change.
	\item Every fdcs map is $\ell$-fine.
	\item Let $f\colon Y \to X$ and $g\colon Z \to Y$ be maps of small v-stacks. If $f$ and $f \comp g$ are $\ell$-fine then so is $g$.
\end{lemenum}
\end{lemma}
\begin{proof}
This is formal, see \cite[Lemma 2.5]{mod-p-stacky-6-functors} for a $p$-torsion analog (and use \cref{rslt:fdcs-maps-have-2-out-of-3} in place of \cite[Lemma 3.6.10.(iv)]{mann-mod-p-6-functors}).
\end{proof}

In \cref{sec:smoothness,sec:representations} we will provide many examples of $\ell$-fine maps which are not fdcs (and usually not $0$-truncated). We can finally formulate the main result of this paper:

\begin{theorem} \label{rslt:6-functor-formalism}
There is a 6-functor formalism
\begin{align*}
	\D_\nuc(-,\Lambda)\colon \Corr(\vStacks)_{lfine,all} \to \infcatinf
\end{align*}
with the following properties:
\begin{thmenum}
	\item Restricted to the symmetric monoidal subcategory $\vStacks^\opp$ (equipped with the coproduct monoidal structure), $\D_\nuc(-,\Lambda)$ coincides with the functor constructed in \cref{rslt:v-descent-for-nuclear-Lambda-modules}.

	\item For every $\ell$-fine map $f\colon Y \to X$ the functor
	\begin{align*}
		f_! := \D_\nuc([Y \xfrom{\id} Y \xto{f} X],\Lambda)\colon \D_\nuc(Y,\Lambda) \to \D_\nuc(X,\Lambda)
	\end{align*}
	preserves all small colimits. If $f = j$ is étale then $j_!$ is left adjoint to $j^*$ and if $f$ is proper then $f_! = f_*$.
\end{thmenum}
\end{theorem}
\begin{proof}
We first prove the following claim: Let $f\colon Y \to X$ be a map of small v-stacks with Čech nerve $f_\bullet\colon Y_\bullet \to X$ and assume that $f$ admits universal $\ell$-codescent and $X$ admits a map to some strictly totally disconnected space; then the natural functor
\begin{align*}
	f_\bullet^!\colon \D_\nuc^!(X,\Lambda) \isoto \varprojlim_{n\in\Delta} \D_\nuc^!(Y_n,\Lambda)
\end{align*}
is an equivalence. By definition of universal $\ell$-codescent this holds for $\Lambda = \Z_\ell$, so we only need to reduce the general case to this case. Note that the functor $f_\bullet^!$ has a left adjoint
\begin{align*}
	f_{\bullet!}\colon \varprojlim_{n\in\Delta} \D_\nuc^!(Y_n,\Lambda) \to \D_\nuc^!(X,\Lambda), \qquad (\mathcal M_n)_n \mapsto \varinjlim_{n\in\Delta} f_{n!} \mathcal M_n.
\end{align*}
Thus in order to prove the desired equivalence it is enough to show that the natural maps $f_{\bullet!} f_\bullet^! \isoto \id$ and $\id \isoto f_\bullet^! f_{\bullet!}$ are isomorphisms. This follows immediately from the $\Z_\ell$-case together with the fact that both functors commute with the forgetful functor along $\Z_\ell \to \Lambda$, the proof of which we delay to the end of this section (see \cref{rslt:functoriality-of-shriek-functors-in-Lambda} below).

With the above claim at hand, the construction of the desired 6-functor formalism for $\ell$-fine maps is now completely formal (cf. the proof of \cite[Proposition 2.6]{mod-p-stacky-6-functors}): Our goal is to show that the 6-functor formalism from \cref{rslt:6-functor-formalism-for-diamonds} extends uniquely to a 6-functor formalism for all $\ell$-fine maps. To do that, let $\mathcal C \subset \vStacks$ denote the full subcategory spanned by those small v-stacks which admit a map to some strictly totally disconnected space. Then $\mathcal C$ is a basis of $\vStacks$, hence by \cite[Proposition A.5.16]{mann-mod-p-6-functors} it is enough to construct the desired extension of the 6-functor formalism on $\mathcal C$, i.e. we need to construct a 6-functor formalism $\D(-,\Lambda)\colon \Corr(\mathcal C)_{lfine,all} \to \infcatinf$ extending the one from \cref{rslt:6-functor-formalism-for-diamonds}. We now apply \cite[Proposition A.5.14]{mann-mod-p-6-functors} with $E = fdcs$, $E' = lfine$ and $S \subset E$ being the subset of those maps which admit universal $\ell$-codescent. Condition (a) is clear, condition (b) was proved above, condition (c) holds by definition of $\ell$-fine maps and condition (d) follows from \cref{rslt:properties-of-fdcs-maps} (see \cite[Remark A.5.15.(ii)]{mann-mod-p-6-functors}).
\end{proof}

Let us extract the shriek functors from \cref{rslt:6-functor-formalism}, thereby completing the collection of six functors for nuclear sheaves:

\begin{definition}
Let $f\colon Y \to X$ be an $\ell$-fine map of small v-stacks.
\begin{defenum}
	\item We define $f_!\colon \D_\nuc(Y,\Lambda) \to \D_\nuc(X,\Lambda)$ to be the functor $f_! := \D_\nuc([Y \xfrom{\id} Y \xto{f} X],\Lambda)$, where $\D(-,\Lambda)$ is the 6-functor formalism from \cref{rslt:6-functor-formalism}.

	\item We define $f^!\colon \D_\nuc(X,\Lambda) \to \D_\nuc(Y,\Lambda)$ to be the right adjoint of $f_!$.
\end{defenum}
\end{definition}

\begin{remark} \label{rmk:explicit-construction-of-shriek-functors}
The construction of the functor $f_!$ in \cref{rslt:6-functor-formalism} is not very explicit, so we provide a more direct description:
\begin{enumerate}[1.]
	\item Suppose that $f\colon Y \to X$ is quasicompact, compactifiable and representable in spatial diamonds with finite $\dimtrg$. Then we define
	\begin{align*}
		f_! := (\overline f^{/X})_* \comp j_!,
	\end{align*}
	where $j$ denotes the open immersion $j\colon Y \injto \overline Y^{/X}$ and $j_!$ is the functor from \cref{rslt:properties-of-etale-lower-shriek}.

	\item Suppose that $f\colon Y \to X$ is an fdcs map of locally spatial diamonds. Let $\mathcal I$ be the category of open subsets $V \subset Y$ which are quasicompact and compactifiable over $X$. Then we have $\D_\nuc(Y,\Lambda) = \varinjlim_{V\in\mathcal I} \D_{\nuc,!}(V,\Lambda)$ in the $\infty$-category of presentable $\infty$-categories and colimit preserving functors, where $\D_{\nuc,!}$ is the functor mapping an inclusion $j\colon V \injto V'$ to $j_!$. We thus need to define $f_!$ as the functor
	\begin{align*}
		f_! \mathcal M := \varinjlim_{V\in\mathcal I} f_{V!} (\restrict{\mathcal M}V),
	\end{align*}
	where $f_V\colon V \to X$ denotes the composition $V \injto Y \to X$. The functors $f_{V!}$ were defined in the previous step.

	\item Suppose that $f\colon Y \to X$ is an arbitrary fdcs map of small v-stacks. Choose a hypercover $X_\bullet \to X$ such that all $X_n$ are locally spatial diamonds and let $f_\bullet\colon Y_\bullet \to X_\bullet$ be the base-change of $f$. Then we have $\D_\nuc(X,\Lambda) = \varprojlim_{n\in\Delta} \D_\nuc(X_n,\Lambda)$ and $\D_\nuc(Y,\Lambda) = \varprojlim_{n\in\Delta} \D_\nuc(Y_n,\Lambda)$ and with this representation we define
	\begin{align*}
		f_!(\mathcal M_n)_n := (f_{n!} \mathcal M_n)_n,
	\end{align*}
	where each $f_{n!}$ was defined in the previous step. This definition is possible because all $f_{n!}$ satisfy arbitrary base-change.

	\item Suppose that $f\colon Y \to X$ is an $\ell$-fine map such that $X$ admits a map to some strictly totally disconnected space. Pick an fdcs map $g\colon Z \to Y$ which admits universal $\ell$-codescent such that $f \comp g$ is fdcs and let $g_\bullet\colon Z_\bullet \to Y$ denote the Čech nerve of $g$. Then for every $\mathcal M \in \D_\nuc(Y,\Lambda)$ we have $\mathcal M = \varinjlim_{n\in\Delta} g_{n!} g_n^! \mathcal M$ and hence we must define
	\begin{align*}
		f_! \mathcal M := \varinjlim_{n\in\Delta} (f \comp g_n)_! g_n^! \mathcal M,
	\end{align*}
	where the functors $(f \comp g_n)_!$ and $g_n^!$ were defined in the previous step.

	\item Suppose that $f\colon Y \to X$ is a general $\ell$-fine map of small v-stacks. Then we can use the same descent technique as in step 3 to reduce the definition of $f_!$ to the previous step.
\end{enumerate}
Note that one could attempt to carry out the above construction directly instead of relying on the theory of abstract 6-functor formalisms from \cite[\S A.5]{mann-mod-p-6-functors}. This is possible to some extent, e.g. steps 1, 2 and 3 were carried out in \cite[\S22]{etale-cohomology-of-diamonds}. This has two downsides though: Firstly, the direct construction will not provide all the higher homotopies one expects the shriek functors to satisfy, so that thorough proofs involving the shriek functors require a lot of diagram checking; secondly, it seems very hard to carry out steps 4 and 5 using this approach.
\end{remark}

We have constructed the full 6-functor formalism for nuclear sheaves over a fixed nuclear $\Z_\ell$-algebra $\Lambda$. Sometimes it is useful to also understand how this 6-functor formalism behaves under a change of $\Lambda$. For the following result we denote by $\Ring_{\Z_\ell,\nuc}$ the $\infty$-category of nuclear $\Z_\ell$-algebras.

\begin{proposition} \label{rslt:6-functor-formalism-functorial-in-Lambda}
There is a 6-functor formalism
\begin{align*}
	\D_\nuc\colon \Corr(\vStacks \cprod \Ring_{\Z_\ell,\nuc}^\opp)_{lfine,all} \to \infcatinf, \qquad (X,\Lambda) \mapsto \D_\nuc(X,\Lambda),
\end{align*}
where $lfine$ denotes the class of those maps $(Y,\Lambda') \to (X,\Lambda)$ where the map $Y \to X$ is an $\ell$-fine map of small v-stacks and the map $\Lambda \to \Lambda'$ is an isomorphism. This 6-functor formalism has the following properties:
\begin{propenum}
	\item For every small v-stack $X$ and any map $\Lambda \to \Lambda'$ of nuclear $\Z_\ell$-algebras, the induced pullback functor $\D_\nuc(X,\Lambda) \to \D_\nuc(X,\Lambda')$ is $- \tensor_\Lambda \Lambda'$.

	\item For every nuclear $\Z_\ell$-algebra $\Lambda$, the restriction of $\D_\nuc$ to $\Corr(\vStacks \cprod \{ \Lambda \})_{lfine,all}$ coincides with the 6-functor formalism $\D_\nuc(-,\Lambda)$ from \cref{rslt:6-functor-formalism}.
\end{propenum}
\end{proposition}
\begin{proof}
Using the functor $\D_\nuc(-,\Z_\ell)\colon \vStacks \to \infcatinf^\tensor$ and abstract nonsense involving the usual straightening and unstraightening techniques and the generalized $\infty$-operad $\Mod(\mathcal C)^\tensor$ from \cite[Definition 4.5.1.1]{lurie-higher-algebra} we can construct a functor
\begin{align*}
	\D_\nuc\colon \vStacks \cprod \Ring_{\Z_\ell,\nuc}^\opp \to \infcatinf^\tensor, \qquad (X,\Lambda) \mapsto \D_\nuc(X,\Lambda)
\end{align*}
which satisfies (i) and restricts to the functor from \cref{rslt:v-descent-for-nuclear-Lambda-modules} for fixed $\Lambda$ (here $\infcatinf^\tensor$ denotes the $\infty$-category of symmetric monoidal $\infty$-categories). Now the construction of the desired 6-functor formalism can be carried out in the same way as in \cref{rslt:6-functor-formalism}. The only difference occurs at the very beginning of the construction in \cref{rslt:6-functor-formalism-for-diamonds}, where we additionally need to verify the compatibility of proper pushforward and étale lower shriek with the base-change $- \tensor_\Lambda \Lambda'$; this follows immediately from the projection formula.
\end{proof}

\begin{corollary} \label{rslt:functoriality-of-shriek-functors-in-Lambda}
Let $\Lambda \to \Lambda'$ be a map of nuclear $\Z_\ell$-algebras and $f\colon Y \to X$ an $\ell$-fine map of small v-stacks.
\begin{corenum}
	\item The functor $f_!$ commutes naturally with the base-change and the forgetful functor along $\Lambda \to \Lambda'$.

	\item The functor $f^!$ commutes naturally with the forgetful functor along $\Lambda \to \Lambda'$.
\end{corenum}
\end{corollary}
\begin{proof}
The fact that $f_!$ commutes with the base-change functor $- \tensor_\Lambda \Lambda'$ follows immediately from proper base-change in the 6-functor formalism from \cref{rslt:6-functor-formalism-functorial-in-Lambda} for the cartesian diagram
\begin{center}\begin{tikzcd}
	(Y,\Lambda') \arrow[r] \arrow[d] & (Y,\Lambda) \arrow[d]\\
	(X,\Lambda') \arrow[r] & (X,\Lambda)
\end{tikzcd}\end{center}
Note that it is really necessary to use \cref{rslt:6-functor-formalism-functorial-in-Lambda} here even though on underlying $\Lambda$-modules the desired commutation of functors reduces to the projection formula. Namely, without the 6-functor formalism from \cref{rslt:6-functor-formalism-functorial-in-Lambda} we do not even get a natural morphism between the functors $f_!(-) \tensor_\Lambda \Lambda'$ and $f_!(- \tensor_\Lambda \Lambda')$ from $\D_\nuc(Y,\Lambda)$ to $\D_\nuc(X,\Lambda')$ (only as functors to $\D_\nuc(X,\Lambda)$, but that is not enough). Claim (ii) follows immediately by passing to right adjoints.

It remains to see that $f_!$ commutes with the forgetful functor along $\Lambda \to \Lambda'$, which by the commutation of $f_!$ with $- \tensor_\Lambda \Lambda'$ is now a condition rather than an additional datum. Therefore this can be checked along all the steps in the direct computation of $f_!$ in \cref{rmk:explicit-construction-of-shriek-functors}, which ultimately reduces us to the case that either $f$ is an open immersion or proper. In the former case we apply \cref{rslt:etale-lower-shriek-is-functorial-in-Lambda}. In the latter case we have $f_! = f_*$, so that the claim follows by passing to left adjoints and using \cref{rslt:def-and-compatibility-of-Lambda-base-change-and-forget}.
\end{proof}

We finish this section by recording the following formal consequences of any 6-functor formalism, which are useful statements by themselves.

\begin{lemma}
Let $f\colon Y \to X$ be an $\ell$-fine map of small v-stacks.
\begin{lemenum}
	\item \label{rslt:IHom-adjunction-for-shriek-functors} For all $\mathcal M \in \D_\nuc(X, \Lambda)$ and $\mathcal N \in \D_\nuc(Y, \Lambda)$ there is a natural isomorphism
	\begin{align*}
		\IHom(f_! \mathcal N, \mathcal M) = f_* \IHom(\mathcal N, f^! \mathcal M).
	\end{align*}

	\item \label{rslt:upper-shriek-of-IHom-formula} For all $\mathcal M, \mathcal N \in \D_\nuc(X, \Lambda)$ there is a natural isomorphism
	\begin{align*}
		f^! \IHom(\mathcal N, \mathcal M) = \IHom(f^* \mathcal N, f^! \mathcal M).
	\end{align*}
\end{lemenum}
\end{lemma}
\begin{proof}
This is formal, see e.g. \cite[Proposition 23.3]{etale-cohomology-of-diamonds}.
\end{proof}

\section{Perfect, Dualizable and Overconvergent Sheaves} \label{sec:perf}

Fix a prime $\ell \ne p$ and a nuclear $\Z_\ell$-algebra $\Lambda$. Having developed a full 6-functor formalism for nuclear $\Lambda$-sheaves on small v-stacks, we now want to study some particular special cases of nuclear sheaves.

We start with dualizable sheaves. The notion of dualizable objects can be defined in any symmetric monoidal $\infty$-category:

\begin{definition} \label{def:dualizable-object}
Let $\mathcal C$ be a symmetric monoidal $\infty$-category. An object $P \in \mathcal C$ is called \emph{dualizable} if there are an object $P^*$, called the \emph{dual} of $P$, and morphisms
\begin{align*}
	\ev_P\colon P^* \tensor P \to 1, \qquad i_P\colon 1 \to P \tensor P^*,
\end{align*}
called the \emph{evaluation map} and \emph{coevaluation map} respectively, such that there are homotopy coherent diagrams
\begin{center}
	\begin{tikzcd}
		P \arrow[dr,"\id"] \arrow[r,"i_P \tensor \id"] & P \tensor P^* \tensor P \arrow[d,"\id \tensor \ev_P"] \\
		& P
	\end{tikzcd}
	\qquad
	\begin{tikzcd}
		P^* \arrow[dr,"\id"] \arrow[r,"\id \tensor i_P"] & P^* \tensor P \tensor P^* \arrow[d,"\ev_P \tensor \id"] \\
		& P^*
	\end{tikzcd}
\end{center}
We denote by $\mathcal C_\dlb \subset \mathcal C$ the full subcategory spanned by the dualizable objects in $\mathcal C$.
\end{definition}

It is easy to see that if $P$ is dualizable then the dual $P^*$ is unique up to unique isomorphism. We will use that fact without further mention below.

Note that by design the notion of dualizability in a symmetric monoidal $\infty$-category $\mathcal C$ only depends on the underlying symmetric monoidal $1$-category. In this case one of the first instances where it appears in the literature is in \cite[\S1]{dold-puppe-duality}. If $\mathcal C$ is closed, i.e. admits an internal hom functor $\IHom$, then there is a different way of defining the dual of an object $P \in \mathcal C$ by letting $P^\vee := \IHom(P, 1)$, where $1 \in \mathcal C$ is the monoidal unit. We say that $P$ is \emph{reflexive} if the natural map $P \isoto P^{\vee\vee}$ is an isomorphism. A natural question is how this notion of duals behaves with respect to dualizability. Here is the answer:

\begin{lemma} \label{rslt:characterization-of-dualizable}
Let $\mathcal C$ be a closed symmetric monoidal $\infty$-category and $P \in \mathcal C$ an object. Then the following are equivalent:
\begin{lemenum}
	\item $P$ is dualizable.

	\item $P$ is reflexive and the natural composed map $P \tensor P^\vee \to P^{\vee\vee} \tensor P^\vee \to (P \tensor P^\vee)^\vee$ is an isomorphism.

	\item The natural map $P \tensor P^\vee \isoto \IHom(P, P)$ is an isomorphism.

	\item For all $M, N \in \mathcal C$ the natural map
	\begin{align*}
		\pi_0\Hom(M, N \tensor P^\vee) \isoto \pi_0\Hom(M \tensor P, N),
	\end{align*}
	given by sending $f\colon M \to N \tensor P^\vee$ into the composition
	\begin{align*}
		M \tensor P \xto{f \tensor \id} N \tensor P^\vee \tensor P \xto{\id \tensor \ev} N,
	\end{align*}
	is an isomorphism.
\end{lemenum}
If this is the case then $P^\vee$ is the dual of $P$ in the sense of \cref{def:dualizable-object} and for all $M \in \mathcal C$ the natural map
\begin{align*}
	P^\vee \tensor M \isoto \IHom(P, M)
\end{align*}
is an isomorphism.
\end{lemma}
\begin{proof}
Note that all the statements only depend on the underlying $1$-category of $\mathcal C$, so we can assume that $\mathcal C$ is a $1$-category. Then the equivalence of (i), (ii) and (iv) is shown in \cite[Theorem 1.3]{dold-puppe-duality}. It goes as follows: The equivalence of (i) and (iv) follows easily by observing that both statements are equivalent to the fact that the functors $- \tensor P$ and $- \tensor P^*$ (resp. $- \tensor P^\vee$) are adjoint with the obvious counit. This also proves that necessarily $P^* = P^\vee$ in (i). Note that this implies that condition (iv) equivalently holds with the roles of $P$ and $P^\vee$ swapped. With this known, it is straightforward to see that (iv) implies (ii). To prove that (ii) implies (i), note that we always have the evaluation map $\ev\colon P \tensor P^\vee \to 1$ and if the isomorphism in (ii) holds then we also get a coevaluation map $i\colon 1 = 1^\vee \to (P \tensor P^\vee)^\vee = P \tensor P^\vee$. To show that the necessary triangles commute, we consider the following diagram in $\mathcal C$:
\begin{center}\begin{tikzcd}
	P \arrow[r,"i \tensor \id"] \arrow[ddr,bend right,swap,"\alpha"] & P \tensor P^\vee \tensor P \arrow[r,"\id \tensor \ev"] \isoarrow[d] & P \isoarrow[d]\\
	& P^{\vee\vee} \tensor P^\vee \tensor P \arrow[r,"\id \tensor \ev"] \isoarrow[d] & P^{\vee\vee}\\
	& (P \tensor P^\vee)^\vee \tensor P \arrow[ur,swap,"\beta"]
\end{tikzcd}\end{center}
Here the vertical isomorphisms are induced by the isomorphisms in (ii) (where in the middle column the maps act as the identity on the right-hand $P$). The map $\alpha$ is the natural map induced by the dualized evaluation map $1 = 1^\vee \to((P \tensor P^\vee)^\vee$ (intuitively this singles out the element $A \in (P \tensor P^\vee)^\vee$ given as the map $P \tensor P^\vee \to 1$, $(x, f) \mapsto f(x)$), so that the left-hand triangle of the diagram commutes by definition of $i$. Also clearly the upper right square commutes. The map $\beta$ is the map which by adjunction corresponds to the map
\begin{align*}
	(P \tensor P^\vee)^\vee = \IHom(P \tensor P^\vee, 1) \isoto \IHom(P, P^{\vee\vee})
\end{align*}
(Intuitively it sends a map $A\colon P \cprod P^\vee \to 1$ and an element $x \in P$ to the element $f \mapsto A(x,f)$ in $P^{\vee\vee}$.) One checks immediately that the lower right triangle commutes. It follows that the whole diagram commutes. Moreover, we see that $\beta \comp \alpha$ is the natural map $P \to P^{\vee\vee}$. By the commutativity of the diagram and the fact that $P \isoto P^{\vee\vee}$ is an isomorphism, we deduce that the composition of the upper two horizontal maps is the identity, as desired. The second diagram in the definition of dualizable maps can be verified similarly (or note that everything is symmetric in $P$ and $P^\vee$).

We have now shown that (i), (ii) and (iv) are all equivalent. To show that (ii) implies (iii), we simply note that if (ii) holds then
\begin{align*}
	P \tensor P^\vee = (P^\vee \tensor P)^\vee = \IHom(P, P^{\vee\vee}) = \IHom(P, P),
\end{align*}
as desired. Finally, assume that (iii) is satisfied. Then additionally to the natural evaluation map $\ev\colon P \tensor P^\vee \to 1$ there is a coevaluation map $1 \to \IHom(P, P) = P \tensor P^\vee$ given by the identity on $P$. One checks that these two maps exhibit $P$ as dualizable with dual $P^\vee$: The left-hand triangle is obviously commutative; for the right-hand triangle we use the fact that for every $M, N \in \mathcal C$ the evaluation map $M^\vee \tensor M \tensor N^\vee \to N^\vee$ factors over $M^\vee \tensor \IHom(N, M)$, where the map from this object to $N^\vee$ is a special case of the general map
\begin{align*}
	\IHom(X, Y) \tensor \IHom(Y, Z) \to \IHom(X, Z)
\end{align*}
for $X, Y, Z \in \mathcal C$. Now apply this to $N = M = P$ to conclude.

We proved the equivalence of (i), (ii), (iii) and (iv). The final claim follows immediately from (iv).
\end{proof}

\begin{remark}
We put a lot of care in the proof of \cref{rslt:characterization-of-dualizable}, even though the proof is rather straightforward and the results are not new. The reasons for this are twofold:
\begin{enumerate}[1.]
	\item We found it surprisingly hard to collect these results from the literature, and we were unable to find a reference where all displayed characterizations of dualizability are stated as cleanly as here.

	\item We have repeatedly been confused about what is true and what not. For example, we found a reference where it is suggested that the isomorphism $P \tensor P^\vee \isoto \IHom(P, P)$ is not enough to deduce dualizability, but it certainly is. Moreover, when characterizing dualizability via adjoint functors, one has to be careful. In particular, criterion (iv) of \cite[Lemma 3.7.4]{mann-mod-p-6-functors} seems to be false in general.
\end{enumerate}
\end{remark}

We also get the following abstract properties of dualizable objects. The first property will help us prove v-descent, while the second property is a sanity check.

\begin{lemma} \label{rslt:dualizable-preserved-by-symm-monoidal-functor}
Let $F\colon \mathcal C \to \mathcal D$ be a symmetric monoidal functor between symmetric monoidal $\infty$-categories and let $P \in \mathcal C$ be dualizable. Then $F(P)$ is dualizable. Moreover, if both $\mathcal C$ and $\mathcal D$ are closed then for all $M \in \mathcal C$ the natural map
\begin{align*}
	F(\IHom(P, M)) \isoto \IHom(F(P), F(M))
\end{align*}
is an isomorphism.
\end{lemma}
\begin{proof}
It follows immediately from the definitions that $F(P)$ is dualizable. The second claim follows because by \cref{rslt:characterization-of-dualizable} both sides can be identified with $F(P)^\vee \tensor F(M)$.
\end{proof}

\begin{lemma} \label{rslt:dualizable-objects-stable-under-finite-lim}
Let $\mathcal C$ be a closed symmetric monoidal stable $\infty$-category. Then the subcategory $\mathcal C_\dlb$ of dualizable objects is stable under finite (co)limits and retracts.
\end{lemma}
\begin{proof}
This follows easily from \cref{rslt:characterization-of-dualizable} by noting that both $P \tensor Q^\vee$ and $\IHom(Q, P)$ commute with finite (co)limits and retracts in $P$ and $Q$.
\end{proof}

Let us now come to the geometric setting. In the setting of categories of sheaves on geometric objects, one can usually characterize dualizability in terms of some ``perfectness'' condition. The same happens for our $\infty$-category of nuclear sheaves, so let us introduce the relevant definitions. We start with the notion of overconvergence, which naturally comes up when studying perfect sheaves:

\begin{definition}
Let $X$ be a strictly totally disconnected space.
\begin{defenum}
	\item We define
	\begin{align*}
		\D_\nuc(\pi_0(X),\Lambda) \subset \D_\solid(\pi_0(X),\Lambda)_{\omega_1} \subset \D(\pi_0(X)_\proet,\Lambda)
	\end{align*}
	in the obvious sense (e.g. view $\pi_0(X)$ as a strictly totally disconnected space by realizing it over some $\Spa C$).

	\item \label{def:map-to-pi-0} We denote by
	\begin{align*}
		\pi\colon X_\proet \to \pi_0(X)_\proet
	\end{align*}
	the natural map of sites.

	\item We denote $\Z_\ell(X) = \cts(X, \Z_\ell)$, which is a static $\Z_\ell$-Banach algebra. We further denote $\Lambda(X) := \Lambda \tensor_{\Z_\ell} \Z_\ell(X)$, which is a nuclear $\Z_\ell$-algebra.
	\end{defenum}
\end{definition}

\begin{lemma} \label{rslt:prpoerties-of-pi-0-push-pull}
Let $X$ be a strictly totally disconnected space.
\begin{lemenum}
	\item The functor $\pi_*\colon \D_\solid(X,\Lambda)_{\omega_1} \to \D_\solid(\pi_0(X),\Lambda)_{\omega_1}$ is bounded and preserves all small limits and colimits.

	\item The functor $\pi^*\colon \D_\solid(\pi_0(X),\Lambda)_{\omega_1} \injto \D_\solid(X,\Lambda)_{\omega_1}$ is $t$-exact and fully faithful and preserves all small colimits and countable limits.

	\item Both $\pi_*$ and $\pi^*$ preserve nuclear sheaves.

	\item \label{rslt:overconvergent-sheaves-on-std-space-are-modules} There is a natural equivalence of symmetric monoidal $\infty$-categories
	\begin{align*}
		\D_\nuc(\pi_0(X),\Lambda) = \D_\nuc(\Lambda(X)),
	\end{align*}
	where $\D_\nuc(\Lambda(X))$ denotes the $\infty$-category of nuclear $\Lambda(X)$-modules on $*_\proet$.
\end{lemenum}
\end{lemma}
\begin{proof}
All of the claims commute with the forgetful functor along $\Z_\ell \to \Lambda$, so that we can assume $\Lambda = \Z_\ell$. The claims (i) and (ii) are easy, see e.g. \cite[Lemma 5.7]{mann-werner-simpson} and the arguments in \cref{rslt:pushforward-for-w1-solid-sheaves}. Claim (iii) follows immediately by checking on generators.

We now prove (iv). Consider the natural morphism $\alpha\colon (\pi_0(X)_\proet,\Z_\ell) \to (*_\proet,\Z_\ell(X))$ of ringed sites. Again, it is easy to see that $\alpha_*$ preserves colimits and countable limits of $\omega_1$-solid sheaves and hence restricts to a functor $\alpha_*\colon \D_\solid(\pi_0(X),\Z_\ell)_{\omega_1} \to \D_\solid(\Z_\ell(X))_{\omega_1}$. This restricted functor is $t$-exact (check on countable limits of qcqs étale sheaves, reducing to the case of étale sheaves, where it is clear) and preserves all small limits (this can be checked after applying the forgetful functor to $\D_\solid(\Z_\ell)_{\omega_1}$, where it is clear). It thus admits a left adjoint $\alpha^*\colon \D_\solid(\Z_\ell(X))_{\omega_1} \to \D_\solid(\pi_0(X),\Z_\ell)_{\omega_1}$. Note that precomposing $\alpha^*$ with the functor $- \tensor_{\Z_\ell} \Z_\ell(X)$ gives the functor $s^*\colon \D_\solid(\Z_\ell)_{\omega_1} \to \D_\solid(\pi_0(X),\Z_\ell)_{\omega_1}$, where $s\colon \pi_0(X)_\proet \to *_\proet$ is the natural projection; this allows us to compute $\alpha^*$. First of all it implies that $\alpha^*$ maps discrete $\Lambda(X)$-modules to étale $\Z_\ell$-sheaves on $\pi_0(X)$, because the discrete $\Lambda$-modules are generated under colimits by the objects $M \tensor_{\Z_\ell} \Z_\ell(X)$ for discrete $\Z_\ell$-modules $M$. We can similarly show that $\alpha^*$ preserves right-bounded $\ell$-adically complete modules: We know that $\alpha^*$ is right $t$-exact, so by the usual arguments (see the proof of \cref{rslt:IHom-tr-preserves-completeness}) we can reduce this claim to countable products of a set of compact generators. But compact generators can be chosen as $\prod_I \Z_\ell \tensor_{\Z_\ell} \Z_\ell(X)$ and they get sent via $\alpha^*$ to $s^* \prod_I \Z_\ell = \prod_I \Z_\ell$ (where on the right hand side we view it as a sheaf on $\pi_0(X)$).

We now claim that $\alpha^*$ preserves nuclearity and is fully faithful on nuclear modules. The first claim follows from the fact that for every basic nuclear $M \in \D_\nuc(\Z_\ell)$, $\alpha^*(M \tensor_{\Z_\ell} \Z_\ell(X)) = s^* M$ is nuclear. For the second claim, we need to show that $\alpha_* \alpha^* N = N$ for all nuclear $\Z_\ell(X)$-modules $N$. Since everything commutes with colimits we can again assume that $N = M \tensor_{\Z_\ell} \Z_\ell(X)$ for some basic nuclear $\Z_\ell$-module $M$. Writing $M$ as a colimit of bounded Banach modules we can further reduce to the case that $M$ is bounded and $\ell$-adically complete. Since both $\alpha^*$ and $\alpha_*$ preserve (right-bounded) $\ell$-adically complete objects, the claimed isomorphism can be checked after reducing modulo $\ell$. But on discrete objects it is easy to see that $\alpha_* \alpha^* = \id$ (use again the trick via $s^*$). We have now shown that $\alpha^*$ induces an embedding
\begin{align*}
	\alpha^*\colon \D_\nuc(\Z_\ell(X)) \injto \D_\nuc(\pi_0(X),\Z_\ell).
\end{align*}
It remains to show that $\alpha^*$ is essentially surjective, i.e. contains every $\mathcal M \in \D_\nuc(\pi_0(X),\Z_\ell)$. By writing $\mathcal M$ as the colimit of the nuclearizations of compact objects (see \cref{rslt:properties-of-nuclearization}) we can assume that $\mathcal M$ is bounded and $\ell$-adically complete. But $\alpha^*$ is clearly essentially surjective onto étale sheaves, hence every $\mathcal M/\ell^n\mathcal M$ lies in the image of $\alpha^*$, i.e. $\mathcal M/\ell^n\mathcal M = \alpha^*\alpha_*(\mathcal M/\ell^n\mathcal M)$. Since $\alpha^*$ preserves right-bounded $\ell$-adically complete objects we deduce $\mathcal M = \alpha^*\alpha_* \mathcal M$, as desired.
\end{proof}

\begin{definition}
Let $X$ be a small v-stack and $\mathcal M \in \D_\nuc(X,\Lambda)$. We say that $\mathcal M$ is \emph{overconvergent} if for every map $f\colon Y \to X$ from a strictly totally disconnected space $Y$ the pullback $f^* \mathcal M$ lies in the essential image of the embedding
\begin{align*}
	\pi^* \D_\nuc(\Lambda(Y)) \injto \D_\nuc(Y,\Lambda)
\end{align*}
from \cref{rslt:prpoerties-of-pi-0-push-pull}. We denote by
\begin{align*}
	\D_\nuc(X,\Lambda)^\oc \subset \D_\nuc(X,\Lambda)
\end{align*}
the full subcategory spanned by the overconvergent sheaves.
\end{definition}

\begin{lemma}
\begin{lemenum}
	\item Let $X$ be a small v-stack. Then $\D_\nuc(X,\Lambda)^\oc$ is stable under the symmetric monoidal structure, all colimits, all pullbacks, all $\ell$-adic completions and under the forgetful and base-change functors for any map $\Lambda \to \Lambda'$ of nuclear $\Z_\ell$-algebras.

	\item The assignment $X \mapsto \D_\nuc(X,\Lambda)^\oc$ defines a hypercomplete sheaf of $\infty$-category on the v-site of small v-stacks.

	\item Let $X$ be a small v-stack and $\mathcal M \in \D_\nuc(X,\Lambda)$. Then $\mathcal M$ is overconvergent if and only if for every pro-étale map $Y' \to Y$ of strictly totally disconnected spaces over $X$ such that $\pi_0(Y') \isoto \pi_0(Y)$ is an isomorphism, the induced map $\Gamma(Y, \mathcal M) \isoto \Gamma(Y', \mathcal M)$ is an isomorphism.

	\item \label{rslt:dual-of-compact-is-overconvergent} Let $X$ be an $\ell$-bounded spatial diamond and $\mathcal P \in \D_\solid(X,\Lambda)_{\omega_1}$ compact. Then $\IHom_\solid(\mathcal P, \Lambda)$ is nuclear and overconvergent.
\end{lemenum}
\end{lemma}
\begin{proof}
Part (i) follows easily from \cref{rslt:prpoerties-of-pi-0-push-pull}. Part (ii) can be checked on the v-site of strictly totally disconnected spaces, where it follows from the v-descent of nuclear sheaves (see \cref{rslt:v-descent-for-nuclear-Lambda-modules}) on profinite sets.

We now prove (iii), so let $X$ and $\mathcal M$ be given. By (i) we can assume that $\Lambda = \Z_\ell$. We first show that the given condition on $\Gamma(Y,\mathcal M) \isoto \Gamma(Y',\mathcal M)$ implies overconvergence, so assume that $\mathcal M$ satisfies this condition. We immediately reduce to the case that $X$ is strictly totally disconnected, in which case we want to show that the natural map $\pi^* \pi_* \mathcal M \isoto \mathcal M$ is an isomorphism. Note that for every pro-étale $U \to X$ we have
\begin{align*}
	\Gamma(U, \pi^* \pi_* \mathcal M) = \Gamma(\pi_0(U), \pi_* \mathcal M) = \Gamma(X \cprod_{\pi_0(X)} \pi_0(U), \mathcal M)
\end{align*}
(the first identity holds for all $\omega_1$-solid sheaves $\mathcal N$ in place of $\mathcal M$, which can be checked on compact generators and thus on qcqs étale sheaves, where it is obvious). But the map $U \to X \cprod_{\pi_0(X)} \pi_0(U)$ is clearly an isomorphism on $\pi_0$, so we deduce $\Gamma(U, \pi^* \pi_* \mathcal M) = \Gamma(U, \mathcal M)$, which implies $\pi^* \pi_* \mathcal M = \mathcal M$. By reversing the argument, we see that if $\mathcal M$ is overconvergent then $\Gamma(U, \mathcal M)$ depends only on $\pi_0(U)$. This proves (iii).

It remains to prove (iv), so let $X$ and $\mathcal P$ be given. By \cref{rslt:compact-generators-of-w1-solid-Lambda-sheaves,rslt:compact-generators-of-w1-solid-sheaves} we can assume that $\mathcal P = \Z_{\ell,\solid}[U] \tensor_{\Z_\ell} \Lambda$ for some basic $U = \varprojlim_n U_n$ in $X_\proet$. If we denote $f\colon U \to X$ the structure map then $\mathcal P = f_\natural \Lambda$, hence $\IHom_\solid(\mathcal P, \Lambda) = f_* \Lambda$. By \cref{rslt:pushforward-for-w1-solid-sheaves} this is indeed nuclear. It remains to show that this sheaf is overconvergent, for which we can w.l.o.g. assume that $X$ is strictly totally disconnected. Then for every affinoid pro-étale $V \to X$ we have
\begin{align*}
	\Gamma(V, f_* \Lambda) = \Gamma(V \cprod_X U, \Lambda) = \Gamma(\pi_0(V \cprod_X U), \Lambda) = \Gamma(\pi_0(V) \cprod_{\pi_0(X)} \pi_0(U), \Lambda)
\end{align*}
(where in the last identity we use e.g. \cite[Lemma 5.8]{mann-werner-simpson}). By (iii) this implies that $f_* \Lambda$ is indeed overconvergent, as desired.
\end{proof}

With a good understanding of overconvergent sheaves at hand, we can now come to perfect and dualizable sheaves. We get the following:

\begin{definition}
Let $X$ be a small v-stack. A sheaf $\mathcal P \in \D_\nuc(X,\Lambda)$ is called \emph{perfect} if there is a v-cover $(f_i\colon Y_i \surjto X)_i$ by strictly totally disconnected spaces $Y_i$ and for each $Y_i$ a dualizable $\Lambda(Y_i)$-module $P_i \in \D_\nuc(\Lambda(Y_i))$ such that $f_i^*\mathcal P \isom \pi_i^* P_i$, where $\pi_i\colon Y_i \to \pi_0(Y_i)$ is the map from \cref{def:map-to-pi-0}. We denote by
\begin{align*}
	\D_\nuc(X,\Lambda)_\perf \subset \D_\nuc(X,\Lambda)^\oc
\end{align*}
the full subcategory spanned by the perfect sheaves.
\end{definition}

\begin{remark}
We do not know if a dualizable $\Lambda(Y_i)$-module is automatically perfect in the sense that it is generated under retracts and finite colimits from $\Lambda(Y_i)$. By \cite[Proposition 9.3]{condensed-complex-geometry} being dualizable is equivalent to being compact in $\D_\nuc(\Lambda(Y_i))$ and by the other results in \cite[\S9]{condensed-complex-geometry} one should be able to show that this is equivalent to being perfect in all cases of practical interest. For example, it holds if $\Lambda$ is ind-compact as a module over some $\ell$-adically complete nuclear $\Z_\ell$-algebra $\Lambda_0$: First reduce to the case $\Lambda = \Lambda_0$, in which case $\Hom(\Lambda,-)$ is easily seen to be conservative on compact objects -- because these are $\ell$-adically complete -- which implies that all compact objects are perfect.
\end{remark}

\begin{proposition} \label{rslt:properties-of-perfect-sheaves}
\begin{propenum}
	\item On every small v-stack $X$ we have $\D_\nuc(X,\Lambda)_\perf = \D_\nuc(X,\Lambda)_\dlb$, i.e. the perfect sheaves are precisely the dualizable objects.

	\item The assignment $X \mapsto \D_\nuc(X,\Lambda)_\perf$ defines a hypercomplete sheaf of $\infty$-categories on the v-site of small v-stacks.
\end{propenum}
\end{proposition}
\begin{proof}
Let us first argue why the assignment $X \mapsto \D_\nuc(X,\Lambda)_\dlb$ is a hypercomplete v-sheaf. To see this, let $f_\bullet\colon Y_\bullet \to X$ be any v-hypercover of small v-stacks. First observe that dualizable objects are stable under pullback by \cref{rslt:dualizable-preserved-by-symm-monoidal-functor}. Now suppose we have some $\mathcal P \in \D_\nuc(X,\Lambda)$ such that all $f_n^* \mathcal P$ are dualizable; we need to show that $\mathcal P$ is dualizable. Given any coCartesian section $\mathcal M_\bullet \in \D_\nuc(Y_\bullet,\Lambda)$, it follows immediately from \cref{rslt:dualizable-preserved-by-symm-monoidal-functor} that $\IHom(\mathcal P_\bullet, \mathcal M_\bullet)$ is again a coCartesian section in $\D_\nuc(Y_\bullet,\Lambda)$. If $\mathcal M \in \D_\nuc(X,\Lambda)$ is the sheaf corresponding to $\mathcal M_\bullet$, it follows that $\IHom(\mathcal P, \mathcal M)$ is the sheaf corresponding to $\IHom(\mathcal P_\bullet, \mathcal M_\bullet)$. In particular, $f_0^* \IHom(\mathcal P, \mathcal M) = \IHom(f_0^* \mathcal P, f_0^* \mathcal M)$. Now \cref{rslt:characterization-of-dualizable} lets us easily deduce that $\mathcal P$ is dualizable as desired.

Now suppose that $X$ is a strictly totally disconnected space and fix any $\mathcal P \in \D_\nuc(X,\Lambda)$. Assume that $\mathcal P$ is dualizable. Then $\Hom(\mathcal P, -) = \Gamma(X, \IHom(\mathcal P, -)) = \Gamma(X, \mathcal P^\vee \tensor -)$ by \cref{rslt:characterization-of-dualizable}, which preserves all small colimits. Hence $\mathcal P$ is compact in $\D_\nuc(X,\Lambda)$ and since the nuclearization functor preserves small colimits, $\mathcal P$ is also compact in $\D_\solid(X,\Lambda)_{\omega_1}$. Applying the same argument to $\mathcal P^\vee$ and using that $\mathcal P = \mathcal P^{\vee\vee}$ we deduce from \cref{rslt:dual-of-compact-is-overconvergent} that $\mathcal P$ and $\mathcal P^\vee$ are overconvergent and hence lie in the essential image of $\pi^*$. Since $\pi^*$ is fully faithful, we deduce immediately from the definition of dualizable objects that $\mathcal P$ is dualizable as an object of $\D_\nuc(\Lambda(X))$. Altogether we see that if $\mathcal P$ is dualizable then it is of the form $\mathcal P = \pi^* P$ for some dualizable nuclear $\Lambda(X)$-module $P$. The converse of this statement is obviously true as well. By combining this observation with v-descent for nuclear modules we easily deduce (i) and (ii).
\end{proof}

\section{Relatively Dualizable Sheaves} \label{sec:reldual}

Fix a prime $\ell \ne p$ and a nuclear $\Z_\ell$-algebra $\Lambda$. In the previous section we studied dualizable sheaves and identified them with the perfect ones. We now want to introduce a more general version of \emph{relatively dualizable} sheaves. This concept is not new: It is known as universally locally acyclic sheaves in \cite{fargues-scholze-geometrization} (for discrete $\Lambda$) and has also previously been studied in the realm of algebraic geometry. Although the original motivation and intuition for these objects comes from geometry, Lu-Zheng \cite{lu-zhen-ula} have recently found an abstract way of describing them, which was also adopted in \cite[Theorem IV.2.23]{fargues-scholze-geometrization}. As we will see in the following, this abstract definition can be carried out with great results in any 6-functor formalism and in particular produces an extremely powerful tool to study smoothness (see \cref{sec:smoothness}). On top of that, in the realm of representation theory the relatively dualizable sheaves will be precisely the admissible representations (see \cref{rslt:properties-of-admissible-reps}) which makes them very interesting for applications to the Langlands program.

Without further ado, let us start with the definition of relatively dualizable sheaves. It relies on the following magical 2-category (cf. \cite[\S IV.2.3.3]{fargues-scholze-geometrization}):

\begin{definition} \label{def:magical-2-category}
Given any small v-stack $S$ we denote by $\mathcal C_S$ the following 2-category: The objects of $\mathcal C_S$ are the $\ell$-fine maps $X \to S$. For any two objects $X, Y \to S$ in $\mathcal C_S$ we define the category $\Fun_{\mathcal C_S}(X, Y)$ as
\begin{align*}
	\Fun_{\mathcal C_S}(X, Y) := \D_\nuc(Y \cprod_S X, \Lambda),
\end{align*}
where on the right-hand side we implicitly take the underlying 1-category of $\D_\nuc$. Given three objects $X, Y, Z \to S$ in $\mathcal C_S$, the composition functor
\begin{align*}
	\Fun_{\mathcal C_S}(Y, Z) \cprod \Fun_{\mathcal C_S}(X, Y) \to \Fun_{\mathcal C_S}(X, Z)
\end{align*}
is defined to be the functor
\begin{align*}
	\D_\nuc(Z \cprod_S Y,\Lambda) \cprod \D_\nuc(Y \cprod_S X,\Lambda) &\to \D_\nuc(Z \cprod_S X,\Lambda),\\
	(\mathcal N, \mathcal M) & \mapsto \mathcal N \star \mathcal M := \pi_{13!}(\pi_{12}^* \mathcal N \tensor \pi_{23}^* \mathcal M),
\end{align*}
where $\pi_{ij}$ denote the various projections of $Z \cprod_S Y \cprod_S X$. It follows from the projection formula that $\mathcal C_S$ is indeed a 2-category. For every $X \to S$ in $\mathcal C_S$ the identity functor on $X$ is given by $\Delta_!\Lambda$, where $\Delta\colon X \to X \cprod_S X$ is the diagonal.
\end{definition}

We also recall the definition of adjoint morphisms in a 2-category. Applied to the 2-category of categories, this recovers the usual notion of adjoint functors.

\begin{definition}
Let $\mathcal C$ be a 2-category. Then a morphism $f\colon X \to Y$ in $\mathcal C$ is \emph{left adjoint} to a morphism $g\colon Y \to X$ if there are a \emph{unit} $\varepsilon\colon \id_X \to gf$ and a \emph{counit} $\eta\colon fg \to \id_Y$ such that the composites
\begin{align*}
	f \xto{f\varepsilon} fgf \xto{\eta f} f, \qquad g \xto{\varepsilon g} gfg \xto{g\eta} g
\end{align*}
are the identity. In this case $g$ is uniquely determined up to unique isomorphism.
\end{definition}

\begin{example} \label{ex:dualizable-equiv-adjoint-functor}
The following example is a very enlightening special case of what is to come: Let $\mathcal D$ be a symmetric monoidal $\infty$-category. We associate to it the 2-category $\mathcal C$ which has only one object $*$ such that $\Fun_{\mathcal C}(*,*) = \D$ (viewed as the underlying 1-category) and such that the composition is given by the tensor product. Then an object $P \in \mathcal D$ is dualizable if and only if it is a left adjoint when viewed as a morphism $* \to *$ in $\mathcal C$. In this case the dual of $P$ is its right adjoint. In fact, the evaluation and coevaluation map for $P$ translate directly to the counit and unit map for the corresponding adjunction.
\end{example}

\begin{remark}
It may seem a bit awkward that in the definition of $\mathcal C_S$ we only work with the 1-categorical version of $\D_\nuc$, thereby throwing away the $\infty$-enrichment (and in particular any sensible notion of limits and colimits). One may also attempt to construct $\mathcal C_S$ as an $(\infty,2)$-category, but this seems hard to do. However, similar to how dualizable objects in a symmetric monoidal $\infty$-category only depend on the underlying 1-category (and still satisfy nice properties in the $\infty$-enrichment, see \cref{rslt:dualizable-objects-stable-under-finite-lim}), for all our applications it is completely sufficient to work with the 2-category (instead of a potential $(\infty,2)$-category) $\mathcal C_S$. This is further illustrated by \cref{ex:dualizable-equiv-adjoint-functor}.
\end{remark}

With the above preparations at hand, we can finally come to the definition of relatively dualizable nuclear sheaves:

\begin{definition}
Let $f\colon X \to S$ be an $\ell$-fine map of small v-stacks. A sheaf $\mathcal P \in \D_\nuc(X,\Lambda)$ is called \emph{$f$-dualizable} if it is a left adjoint when viewed as a morphism $X \to S$ in $\mathcal C_S$. We denote by
\begin{align*}
	\D_\nuc(X,\Lambda)_{\dlb_f} \subset \D_\nuc(X,\Lambda)
\end{align*}
the full subcategory spanned by the $f$-dualizable sheaves.
\end{definition}

For example, if $X = S$ and $f = \id_S$ then by \cref{ex:dualizable-equiv-adjoint-functor} an $\id$-dualizable sheaf on $S$ is the same as a dualizable sheaf on $S$. Somewhat surprisingly, it turns out that relatively dualizable objects have very similar formal properties as dualizable objects, which we collect in the following. The results tend to get rather convoluted due to many pullback and shriek functors flying around -- we encourage the reader in each statement to first try to understand the case that all geometric maps are the identity, in which case one always obtains a basic property of dualizable sheaves like the ones proved in \cref{sec:perf}.

\begin{definition}
Let $f\colon X \to S$ be an $\ell$-fine map of small v-stacks. Then for every $\mathcal M \in \D_\nuc(X,\Lambda)$ we denote
\begin{align*}
	D_f(\mathcal M) := \IHom(\mathcal M, f^! \Lambda)
\end{align*}
and call it the \emph{$f$-dual} of $\mathcal M$.
\end{definition}

\begin{proposition} \label{rslt:characterization-of-f-dualizable-sheaves}
Let $f\colon X \to S$ be an $\ell$-fine map of small v-stacks and $\mathcal P \in \D_\nuc(X,\Lambda)$. Then the following are equivalent:
\begin{propenum}
	\item $\mathcal P$ is $f$-dualizable.

	\item \label{rslt:simple-iso-criterion-for-f-dualizable-sheaves} The natural map
	\begin{align*}
		\pi_1^* D_f(\mathcal P) \tensor \pi_2^* \mathcal P \isoto \IHom(\pi_1^* \mathcal P, \pi_2^! \mathcal P)
	\end{align*}
	is an isomorphism, where $\pi_1, \pi_2\colon X \cprod_S X \to X$ are the two projections.

	\item \label{rslt:functor-isomorphisms-for-f-dualizable-sheaves} For every map $g\colon S' \to S$ of small v-stacks with associated pullback square
	\begin{center}\begin{tikzcd}
		X' \arrow[r,"g'"] \arrow[d,"f'"] & X \arrow[d,"f"]\\
		S' \arrow[r,"g"] & S
	\end{tikzcd}\end{center}
	and base-change $\mathcal P' := g'^* \mathcal P$ the following natural maps of functors are isomorphisms:
	\begin{align*}
		D_{f'}(\mathcal P') \tensor f'^* &\isoto \IHom(\mathcal P', f'^!),\\
		g'^* \IHom(\mathcal P, f^!) &\isoto \IHom(\mathcal P', f'^! g^*).
	\end{align*}
	In particular the natural map $g'^* D_f(\mathcal P) \isoto D_{f'}(g'^* \mathcal P)$ is an isomorphism.
\end{propenum}
If this is the case then also $D_f(\mathcal P)$ is $f$-dualizable and the natural map $\mathcal P \isoto D_f(D_f(\mathcal P))$ is an isomorphism.
\end{proposition}
\begin{proof}
We first prove that (ii) implies (i), which is a more elaborate version of the observation that if $P \tensor P^\vee \isoto \IHom(P, P)$ is an isomorphism for some object $P$ in a closed symmetric monoidal $\infty$-category, then $P$ is dualizable (see \cref{rslt:characterization-of-dualizable}). Namely, as in the proof of \cref{rslt:characterization-of-dualizable} we explicitly construct the counit (the analog of the evaluation map) and unit (the analog of the coevaluation map) for an adjunction between $\mathcal P$ and $D_f(\mathcal P)$. With this analog in mind, it is not surprising that the counit is easy to construct and does not require (ii); it is the map
\begin{align*}
	\mathcal P \star D_f(\mathcal P) = f_!(D_f(\mathcal P) \tensor \mathcal P) \to \Lambda
\end{align*}
which is adjoint to the canonical pairing $D_f(\mathcal P) \tensor \mathcal P \to f^!\Lambda$. On the other hand, the unit is obtained by inverting the isomorphism in (ii), namely it is the map
\begin{align*}
	\Delta_!\Lambda \to D_f(\mathcal P) \star \mathcal P = \pi_1^* D_f(\mathcal P) \tensor \pi_2^* \mathcal P = \IHom(\pi_1^* \mathcal P, \pi_2^! \mathcal P)
\end{align*}
which is given via adjunction by the canonical map (see \cref{rslt:IHom-adjunction-for-shriek-functors})
\begin{align*}
	\Lambda \to \Delta^! \IHom(\pi_1^*\mathcal P, \pi_2^! \mathcal P) = \IHom(\Delta^* \pi_1^* \mathcal P, \Delta^! \pi_2^! \mathcal P) = \IHom(\mathcal P, \mathcal P)
\end{align*}
induced by the identity on $\mathcal P$; here $\Delta\colon X \to X \cprod_S X$ denotes the diagonal. To prove that these maps indeed define an adjunction one can argue similar to \cref{rslt:characterization-of-dualizable}; we leave the details to the reader.

We now prove that (i) implies (iii). First of all, note that condition (i) is stable under every base-change: Given a map $g\colon S' \to S$, consider the functor of 2-categories $\mathcal C_S \to \mathcal C_{S'}$ given by mapping $X \to S$ to $X' := X \cprod_S S' \to S'$ and acting via pullbacks on the morphisms. Since functors of 2-categories obviously preserve adjoint functors, it follows immediately that condition (i) is indeed stable under base-change. Therefore for proving the first isomorphism in (iii) we can from now on assume $S' = S$. Then the claim is an analog of the observation that if $P$ is a dualizable object in a closed symmetric monoidal $\infty$-category then the map $P^\vee \tensor - \isoto \IHom(P, -)$ is an isomorphism of functors (see \cref{rslt:characterization-of-dualizable}). In fact, we can apply a similar proof strategy: Consider the functor from $\mathcal C_S$ to the 2-category of stable $\infty$-categories (viewed as 1-categories by forgetting the $\infty$-enhancement) which maps $X$ to $\D_\nuc(X,\Lambda)$ and $\mathcal M \in \Fun_{\mathcal C_S}(X, Y)$ to the functor $\pi_{2!}(\mathcal M \tensor \pi_1^*)$. By assumption $\mathcal P$ is left adjoint to some object $\mathcal Q \in \D_\nuc(X,\Lambda)$ which by the just constructed functor of 2-categories results in the fact that the functor $f_! (\mathcal P \tensor -)$ is left adjoint to the functor $\mathcal Q \tensor f^*$. But the right adjoint of the former functor is also given by $\IHom(\mathcal P, f^!)$, which produces a uniquely determined 1-categorical isomorphism of functors
\begin{align*}
	\mathcal Q \tensor f^* \isom \IHom(\mathcal P, f^!).
\end{align*}
Plugging in $\Lambda$ yields $\mathcal Q \isom D_f(\mathcal P)$. Note furthermore that the adjunction is induced by a map $\mathcal P \star \mathcal Q = f_!(\mathcal P \tensor \mathcal Q) \to \Lambda$, i.e. a pairing $\mathcal P \tensor \mathcal Q \to f^! \Lambda$. It follows from this observation that under the identification $\mathcal Q \isom D_f(\mathcal P)$ the above 1-categorical isomorphism of functors is the first one in (iii) (and in particular it is an $\infty$-categorical isomorphism). To prove the second isomorphism, note that by applying the first isomorphism on both sides we end up with proving that the natural map $g'^* D_f(\mathcal P) \isoto D_{f'}(g'^* \mathcal P)$ is an isomorphism. But this follows easily from the fact that the pullback functor $\mathcal C_S \to \mathcal C_{S'}$ preserves right adjoints and that these right adjoints are uniquely determined up to unique isomorphism.

We now prove that (iii) implies (ii), so assume that (iii) satisfied. Then after base-change along $X \to S$ we deduce that the natural map of functors
\begin{align*}
	D_{\pi_2}(\pi_1^*\mathcal P) \tensor \pi_2^* \isoto \IHom(\pi_1^* \mathcal P, \pi_2^!)
\end{align*}
is an isomorphism. By evaluating both sides on $\mathcal P$ and using that $D_{\pi_2}(\pi_1^*\mathcal P) = \pi_1^*D_f(\mathcal P)$ (by the second isomorphism in (iii)) we get the desired identity in (ii).

The final claim follows from the above identification of $D_f(\mathcal P)$ as the right adjoint of $\mathcal P$ in $\mathcal C_S$ and the fact that by the equivalence $\mathcal C_S \isom \mathcal C_S^\opp$ it is also a left adjoint of $\mathcal P$.
\end{proof}

Using \cref{rslt:functor-isomorphisms-for-f-dualizable-sheaves} we can show that the notion of relative dualizability satisfies v-descent, in the following sense.

\begin{corollary} \label{rslt:v-descent-for-rel-dualizable-sheaves}
Let
\begin{center}\begin{tikzcd}
	X' \arrow[r,"g'"] \arrow[d,"f'"] & X \arrow[d,"f"]\\
	S' \arrow[r,"g"] & S
\end{tikzcd}\end{center}
be a cartesian square of small v-stacks such that $f$ is $\ell$-fine and let $\mathcal P \in \D_\nuc(X,\Lambda)$ be given.
\begin{corenum}
	\item If $\mathcal P$ is $f$-dualizable then $g'^* \mathcal P$ is $f'$-dualizable.
	\item If $g'^* \mathcal P$ is $f'$-dualizable and $g$ is a v-cover then $\mathcal P$ is $f$-dualizable.
\end{corenum}
\end{corollary}
\begin{proof}
Part (i) is easy, e.g. use the pullback functor $\mathcal C_S \to \mathcal C_{S'}$ constructed in the proof of \cref{rslt:characterization-of-f-dualizable-sheaves}. The proof of (ii) is very similar to the proof of the descent of dualizable objects in \cref{rslt:properties-of-perfect-sheaves}. Assume that $g$ is a v-cover and that $\mathcal P' := g'^* \mathcal P$ is $f'$-dualizable. Fix any v-hypercover $g_\bullet\colon S'_\bullet \to S$ extending $g$ and let $g'_\bullet\colon X'_\bullet \to X$ be the base-change. Let $\mathcal P'_\bullet \in \D_\nuc(X'_\bullet,\Lambda)$ denote the coCartesian section given by $\mathcal P$ (i.e. $\mathcal P'_n = g'^*_n \mathcal P$) and let $f'_n\colon X'_n \to S'_n$ denote the base-change of $f$. Then for all $n$, $\mathcal P'_n$ is $f'_n$-dualizable by (i), hence from the second isomorphism of functors in \cref{rslt:functor-isomorphisms-for-f-dualizable-sheaves} the functor
\begin{align*}
	\D_\nuc(S'_\bullet,\Lambda) \to \D_\nuc(X'_\bullet,\Lambda), \qquad \mathcal M_\bullet \mapsto \IHom(\mathcal P'_\bullet, f'^!_\bullet \mathcal M_\bullet)
\end{align*}
preserves coCartesian sections and hence restricts to a functor $\D_\nuc(S,\Lambda) \to \D_\nuc(X,\Lambda)$ which necessarily coincides with the functor $\mathcal M \mapsto \IHom(\mathcal P, f^! \mathcal M)$ (look at the left adjoints). It follows that the natural morphism of functors
\begin{align*}
	g'^* \IHom(\mathcal P, f^!) \isoto \IHom(\mathcal P', f'^! g^*)
\end{align*}
is an isomorphism. We can deduce that also the natural morphism $D_f(\mathcal P) \tensor f^* \isoto \IHom(\mathcal P, f^!)$ is an isomorphism, because this can be checked after applying $g'^*$ where by the just proved isomorphism of functors it transforms to the corresponding statement for $\mathcal P'$ and $f'$, which follows from \cref{rslt:characterization-of-f-dualizable-sheaves}. But note that the whole argument still works after any base-change, so we deduce that $\mathcal P$ satisfies \cref{rslt:characterization-of-f-dualizable-sheaves} and is therefore $f$-dualizable.
\end{proof}

We also get the following analog of \cref{rslt:dualizable-objects-stable-under-finite-lim}, showing that the notion of $f$-dualizable sheaves is compatible with the $\infty$-categorical enhancement on $\D_\nuc$ (even though we ignored this enhancement in the definition of $f$-dualizability).

\begin{corollary}
Let $f\colon X \to S$ be an $\ell$-fine map of small v-stacks. Then the subcategory $\D_\nuc(X,\Lambda)_{\dlb_f} \subset \D_\nuc(X,\Lambda)$ of $f$-dualizable sheaves is stable under retracts and finite (co)limits.
\end{corollary}
\begin{proof}
This follows immediately from \cref{rslt:functor-isomorphisms-for-f-dualizable-sheaves} because all the involved functors commute with retracts and finite (co)limits in $\mathcal P$.
\end{proof}

We have the following base-change result, generalizing smooth base-change to a relatively dualizable version:

\begin{proposition} \label{rslt:rel-dualizable-base-change}
Let
\begin{center}\begin{tikzcd}
	X' \arrow[r,"g'"] \arrow[d,"f'"] & X \arrow[d,"f"]\\
	S' \arrow[r,"g"] & S
\end{tikzcd}\end{center}
be a cartesian square of small v-stacks and assume that $g$ is $\ell$-fine. Then for every $g$-dualizable $\mathcal P \in \D_\nuc(S',\Lambda)$ the natural morphism of functors
\begin{align*}
	\mathcal P \tensor g^* f_* \isoto f'_* (f'^* \mathcal P \tensor g'^*)
\end{align*}
is an isomorphism.
\end{proposition}
\begin{proof}
We apply the first isomorphism in \cref{rslt:functor-isomorphisms-for-f-dualizable-sheaves} for $D_g(\mathcal P)$ and use the reflexivity $\mathcal P = D_g(D_g(\mathcal P))$ in order to obtain the natural isomorphism
\begin{align*}
	\mathcal P \tensor g^* f_* = \IHom(D_g(\mathcal P), g^! f_*).
\end{align*}
By passing to right adjoints in proper base-change we obtain $g^! f_* = f'_* g'^!$, so that we can transform the right-hand side to
\begin{align*}
	\IHom(D_g(\mathcal P), f'_* g'^!) = f'_* \IHom(f'^* D_g(\mathcal P), g'^!).
\end{align*}
Apply the isomorphisms in \cref{rslt:functor-isomorphisms-for-f-dualizable-sheaves} again to arrive at the desired identity.
\end{proof}

The next statement tells us that relatively dualizable sheaves are stable under ``relatively dualizable pullback''. An important special case is that relatively dualizable sheaves are stable under smooth pullback (see \cref{rslt:rel-dualizable-stable-under-smooth-pullback}).

\begin{proposition} \label{rslt:rel-dualizable-stable-under-rel-dualizable-pullback}
Let $f\colon X \to S$ and $g\colon Y \to X$ be $\ell$-fine maps of small v-stacks and let $\mathcal P \in \D_\nuc(X,\Lambda)$ be $f$-dualizable and $\mathcal Q \in \D_\nuc(Y,\Lambda)$ be $g$-dualizable. Then $g^* \mathcal P \tensor \mathcal Q$ is $(f \comp g)$-dualizable and the natural map
\begin{align*}
	g^* D_f(\mathcal P) \tensor D_g(\mathcal Q) \isoto D_{f\comp g}(g^* \mathcal P \tensor \mathcal Q)
\end{align*}
is an isomorphism.
\end{proposition}
\begin{proof}
There is a functor $\mathcal C_X \to \mathcal C_S$ which sends $[Z \to X] \in \mathcal C_X$ to $[Z \to S] \in \mathcal C_S$ and a morphism $\mathcal M \in \D_\nuc(Z' \cprod_X Z,\Lambda) = \Fun_{\mathcal C_X}(Z,Z')$ to $i_!\mathcal M \in \D_\nuc(Z' \cprod_S Z,\Lambda) = \Fun_{\mathcal C_S}(Z,Z')$, where $i\colon Z' \cprod_X Z \to Z' \cprod_S Z$ is the obvious map: This follows easily from repeated application of the projection formula; we leave the details to the reader. With this functor at hand, the claim is now easy: By assumption $\mathcal Q$ is a left adjoint in $\mathcal C_X$, hence so is its image in $\mathcal C_S$ (as a morphism from $Y$ to $X$ over $S$). If we denote $i\colon Y \to X \cprod_S Y$ the canonical map then this image of $\mathcal Q$ is $i_! \mathcal Q$. Denoting $\pi_X$ and $\pi_Y$ the two projections on $X \cprod_S Y$ we can compute the composition of $i_!\mathcal Q$ with $\mathcal P$ in $\mathcal C_S$ as
\begin{align*}
	&\mathcal P \star i_! \mathcal Q = \pi_{Y!} (i_!\mathcal Q \tensor \pi_X^* \mathcal P) = \pi_{Y!} (i_! i^* \pi_Y^* \mathcal Q \tensor \pi_X^* \mathcal P) = \pi_{Y!}(\pi_Y^* \mathcal Q \tensor i_! \Lambda \tensor \pi_X^* \mathcal P) =\\&\qquad= \mathcal Q \tensor \pi_{Y!}(a_!\Lambda \tensor \pi_X^* \mathcal P) = \mathcal Q \tensor \pi_{Y!} a_! (a^* \pi_X^* \mathcal P) = \mathcal Q \tensor g^* \mathcal P.
\end{align*}
This proves that $g^* \mathcal P \tensor \mathcal Q$ is a left adjoint in $\mathcal C_S$ and hence $(f\comp g)$-dualizable. The claim about the duals follows from the symmetry of the situation and the fact that adjoints are unique up to unique isomorphism.
\end{proof}

Next up we want to prove that relatively dualizable sheaves are stable under proper pushforward. In fact this holds for pushforward along maps which are only cohomologically proper, which roughly means that $f_! = f_*$. Defining this notion thoroughly requires some effort and will be carried out in \cref{sec:properness} (see in particular \cref{rslt:rel-dualizable-stable-under-proper-pushforward}). For now, we will work with an even more general version of maps which are ``cohomologically proper up to a twist''.

\begin{definition} \label{def:f-proper-sheaves}
Let $f\colon X \to S$ be an $\ell$-fine map of small v-stacks. We say that a sheaf $\mathcal P \in \D_\nuc(X,\Lambda)$ is \emph{$f$-proper} if it is a right adjoint when viewed as a morphism from $X$ to $S$ in $\mathcal C_S$. The \emph{$f$-proper dual} $P_f(\mathcal P) \in \D_\nuc(X,\Lambda)$ is the corresponding left adjoint.
\end{definition}

\begin{proposition} \label{rslt:rel-dualizable-stable-under-f-proper-pushforward}
Let $f\colon X \to S$ and $g\colon Y \to X$ be $\ell$-fine maps of small v-stacks. Suppose that $\mathcal P \in \D_\nuc(Y,\Lambda)$ is $(f \comp g)$-dualizable and that $\mathcal Q \in \D_\nuc(Y,\Lambda)$ is $g$-proper. Then $g_*\IHom(\mathcal Q, \mathcal P)$ is $f$-dualizable.
\end{proposition}
\begin{proof}
Let $i\colon Y \to Y \cprod_S X$ be the natural map. Then the functor $\mathcal C_X \to \mathcal C_S$ from \cref{rslt:rel-dualizable-stable-under-rel-dualizable-pullback} sends $\mathcal Q$ to $i_!\mathcal Q$ and $P_g(\mathcal Q)$ to $i_!P_g(\mathcal Q)$. By assumption $P_g(\mathcal Q)$ is a left adjoint in $\mathcal C_X$ when viewed as a morphism from $X$ to $Y$. Consequently the same is still true in $\mathcal C_S$. Since left adjoints are stable under composition, we deduce that the following morphism from $X$ to $S$ is a left adjoint:
\begin{align*}
	\mathcal P \star i_! P_g(\mathcal Q) = \pi_{X!} (\pi_Y^* \mathcal P \tensor i_! P_g(\mathcal Q)) = \pi_{X!} i_! (P_g(\mathcal Q) \tensor i^* \pi_Y^* \mathcal P) = g_!(P_g(\mathcal Q) \tensor \mathcal P).
\end{align*}
Now consider the functor from $\mathcal C_X$ to the category of (underlying 1-categories of) stable $\infty$-categories mapping $\mathcal M$ to $\pi_{2!}(\mathcal M \tensor \pi_1^*)$ (as considered in the proof of \cref{rslt:characterization-of-f-dualizable-sheaves}). This functor preserves adjoint morphisms so that we deduce that the morphism $g_!(\mathcal Q \tensor -)$ is 1-categorically right adjoint to the functor $P_g(\mathcal Q) \tensor g^*$. But we know that the latter functor also has the right adjoint $g_* \IHom(P_g(\mathcal Q), -)$, which produces a 1-categorical isomorphism of functors $g_* \IHom(P_g(\mathcal Q), -) \isom g_!(\mathcal Q \tensor -)$. Reversing the roles of $\mathcal Q$ and $P_g(\mathcal Q)$ via the natural equivalence $\mathcal C_X \isom \mathcal C_X^\opp$ we also deduce that there is a 1-categorical isomorphism of functors $g_*\IHom(\mathcal Q, -) \isom g_!(P_g(\mathcal Q) \tensor -)$. In particular it follows that $\mathcal P \star i_! P_g(\mathcal Q) \isom g_*\IHom(\mathcal Q, \mathcal P)$, as desired.
\end{proof}

Let us also discuss how the notion of relatively dualizable sheaves behaves under a change of the nuclear $\Z_\ell$-algebra $\Lambda$.

\begin{proposition} \label{rslt:rel-dualizable-under-change-of-Lambda}
Let $f\colon X \to S$ be an $\ell$-fine map of small v-stacks and $\mathcal P \in \D_\nuc(X,\Lambda)$.
\begin{propenum}
	\item If $\mathcal P$ is $f$-dualizable then $\mathcal P \tensor_\Lambda \Lambda' \in \D_\nuc(X,\Lambda')$ is also $f$-dualizable.

	\item Suppose that $\Lambda = \Z_\ell$ and that $\mathcal P$ is locally bounded and $\ell$-adically complete. Then $\mathcal P$ is $f$-dualizable if and only if $\mathcal P/\ell\mathcal P \in \D_\nuc(X,\Fld_\ell)$ is $f$-dualizable.
\end{propenum}
\end{proposition}
\begin{proof}
Part (i) follows immediately from the fact that base-change along $\Lambda \to \Lambda'$ invokes a functor on the respective versions of $\mathcal C_S$. This also implies the ``only if'' part of (ii), so it remains to prove the ``if'' part. We therefore assume that $\Lambda = \Z_\ell$ and that $\mathcal P$ is locally bounded and $\ell$-adically complete with $\mathcal P/\ell \mathcal P$ being $f$-dualizable as an $\Fld_\ell$-module. Note that pullback and upper shriek functors preserve $\ell$-adically complete sheaves and that if $\mathcal M$ is $\ell$-adically complete then so is $\IHom(\mathcal N, \mathcal M)$ for \emph{any} $\mathcal N$. Moreover, tensoring with a fixed $\ell$-adically complete and locally bounded sheaf preserves $\ell$-adically complete sheaves (this follows from \cref{rslt:solid-tensor-product-preserves-complete-sheaves} by observing that the statement also holds if one of the sheaves is bounded and the other unbounded because the tensor product has finite Tor dimension by \cite[Proposition VII.2.3]{fargues-scholze-geometrization}). Altogether this implies that both $\pi_1^* D_f(\mathcal P) \tensor \pi_2^* \mathcal P$ and $\IHom(\pi_1^* \mathcal P, \pi_2^! \mathcal P)$ are $\ell$-adically complete, where $\pi_i$ denote the projections from $X \cprod_S X$. By \cref{rslt:characterization-of-f-dualizable-sheaves} we need to check that the natural morphism $\pi_1^* D_f(\mathcal P) \tensor \pi_2^* \mathcal P \isoto \IHom(\pi_1^* \mathcal P, \pi_2^! \mathcal P)$ is an isomorphism, which by $\ell$-adic completeness can be checked modulo $\ell$. But one checks immediately that this reduces the claim to the similar statement for $\mathcal P/\ell \mathcal P$ as an $\Fld_\ell$-module, where the claim holds by assumption and \cref{rslt:characterization-of-f-dualizable-sheaves}.
\end{proof}

\section{Cohomological Smoothness} \label{sec:smoothness}

Fix a prime $\ell \ne p$ and a nuclear $\Z_\ell$-algebra $\Lambda$. We now introduce $\ell$-cohomologically smooth maps of small v-stacks, which are those that satisfy a strong form of Poincaré duality. With the magic of relatively dualizable sheaves, the definition of cohomologically smooth maps is rather simple:

\begin{definition} \label{def:cohom-smooth-maps}
An $\ell$-fine map $f\colon Y \to X$ of small v-stacks is called \emph{$\ell$-cohomologically smooth} if the constant sheaf $\Fld_\ell \in \D_\et(Y,\Fld_\ell)$ is $f$-dualizable and its $f$-dual $D_f(\Fld_\ell)$ is invertible.
\end{definition}

Our definition of cohomological smoothness is a bit unorthodox, but we believe it to be the ``right'' one from a formal standpoint. In the specific 6-functor formalism at hand, we recover the previously defined notion of $\ell$-cohomologically smooth maps:

\begin{lemma} \label{rslt:smoothness-criteria}
Let $f\colon Y \to X$ be an $\ell$-fine map of small v-stacks. Then the following are equivalent:
\begin{lemenum}
	\item $f$ is $\ell$-cohomologically smooth.

	\item The sheaf $f^!\Fld_\ell \in \D_\et(Y,\Fld_\ell)$ is invertible and its formation commutes with base-change along $f$.
\end{lemenum}
If $f$ is fdcs then these conditions are also equivalent to:
\begin{enumerate}[(i)]
	\setcounter{enumi}{2}
	\item For every map $X' \to X$ from a strictly totally disconnected space $X'$ with base-change $f'\colon Y' \to X'$, the sheaf $f'^!\Fld_\ell \in \D_\et(Y',\Fld_\ell)$ is invertible and the natural transformation of functors $f'^!\Fld_\ell \tensor f'^* \isoto f'^!$ is an isomorphism.
\end{enumerate}
\end{lemma}
\begin{proof}
The equivalence of (i) and (ii) follows immediately from \cref{rslt:characterization-of-f-dualizable-sheaves} by observing that $\pi_1^* D_f(\Fld_\ell) \tensor \pi_2^* \Fld_\ell = \pi_1^* f^! \Fld_\ell$ and $\IHom(\pi_1^* \Fld_\ell, \pi_2^! \Fld_\ell) = \pi_2^! \Fld_\ell$, so that the equivalence of these two sheaves amounts precisely to the condition that the formation of $f^!\Fld_\ell$ commutes with base-change along $f$.

The implication that (i) implies (iii) follows immediately from \cref{rslt:functor-isomorphisms-for-f-dualizable-sheaves} (here we do not need $f$ to be fdcs). For the converse, assume that $f$ is fdcs and satisfies (iii). We can assume that $f$ is separated. Then $f$ is $\ell$-cohomologically smooth in the sense of \cite[Definition 23.8]{etale-cohomology-of-diamonds} and so by the results in \cite[\S23]{etale-cohomology-of-diamonds} we see that $f$ satisfies \cref{rslt:functor-isomorphisms-for-f-dualizable-sheaves} for $\mathcal P = \Fld_\ell$.
\end{proof}

\begin{remark}
In \cite[Definition 23.8]{etale-cohomology-of-diamonds} a notion of $\ell$-cohomological smoothness for separated fdcs maps is introduced. It follows from \cref{rslt:smoothness-criteria} that this notion coincides with the one in \cref{def:cohom-smooth-maps}. It follows that our definition of $\ell$-cohomological smoothness is also equivalent to the one in \cite{mod-ell-stacky-6-functors} and to the definition of $\ell$-cohomological smooth maps of Artin v-stacks in \cite{fargues-scholze-geometrization} (all maps in the references for which $\ell$-cohomological smoothness is defined are $\ell$-fine in our sense; the converse is probably false). In particular all results in \cite{etale-cohomology-of-diamonds,fargues-scholze-geometrization} which show that certain maps of small v-stacks are $\ell$-cohomologically smooth also apply in our setting.
\end{remark}

\begin{remark}
We were unable to show that the last criterion in \cref{rslt:smoothness-criteria} implies $\ell$-cohomological smoothness without the fdcs assumption. In particular, this criterion does not seem to capture the correct notion of smoothness in the abstract setup -- there is always a non-formal argument specific to the setting required to make this definition work.
\end{remark}

With the ambiguity of definitions out of the way, we can now deduce all expected properties of nuclear $\Lambda$-modules along smooth maps. Note that everything is completely formal.

\begin{proposition}
\begin{propenum}
	\item \label{rslt:Poincare-duality-for-smooth-maps} Let $f\colon Y \to X$ be an $\ell$-fine and $\ell$-cohomologically smooth map of small v-stacks. Then $f^!\Lambda$ is invertible and the natural morphism
	\begin{align*}
		f^!\Lambda \tensor f^* \isoto f^!
	\end{align*}
	is an isomorphism of functors $\D_\nuc(X,\Lambda) \to \D_\nuc(Y,\Lambda)$.

	\item Let
	\begin{center}\begin{tikzcd}
		Y' \arrow[r,"g'"] \arrow[d,"f'"] & Y \arrow[d,"f"]\\
		X' \arrow[r,"g"] & X
	\end{tikzcd}\end{center}
	be a cartesian square of small v-stacks.
	\begin{enumerate}[(a)]
		\item Assume that $f$ is $\ell$-fine and that either $f$ or $g$ is $\ell$-fine and $\ell$-cohomologically smooth. Then the natural morphism
		\begin{align*}
			g'^* f^! \isoto f'^! g^*
		\end{align*}
		is an isomorphism of functors $\D_\nuc(X,\Lambda) \to \D_\nuc(Y',\Lambda)$.

		\item Assume that $g$ is $\ell$-fine and $\ell$-cohomologically smooth. Then the natural morphism
		\begin{align*}
			g^* f_* \isoto f'_* g'^*
		\end{align*}
		is an isomorphism of functors $\D_\nuc(Y,\Lambda) \to \D_\nuc(X',\Lambda)$.
	\end{enumerate}

	\item Let $f\colon Y \to X$ be an $\ell$-fine and $\ell$-cohomologically smooth map of small v-stacks. Then for all $\mathcal M, \mathcal N \in \D_\nuc(X,\Lambda)$ the natural map
	\begin{align*}
		f^* \IHom(\mathcal M, \mathcal N) \isoto \IHom(f^* \mathcal M, f^* \mathcal N)
	\end{align*}
	is an isomorphism.
\end{propenum}
\end{proposition}
\begin{proof}
Suppose that $f\colon Y \to X$ is an $\ell$-fine and $\ell$-cohomologically smooth map of small v-stacks. Then it follows from \cref{rslt:rel-dualizable-under-change-of-Lambda} that the constant sheaf $\Lambda \in \D_\nuc(Y,\Lambda)$ is $f$-dualizable. Also, $f^! \Lambda$ is invertible: In the case that $\Lambda = \Z_\ell$ this easily reduces to the mod-$\ell$ case, because $f^!$ commutes with $\ell$-adic completions. For general $\Lambda$ we now deduce that $f^!\Lambda = \Lambda \tensor_{\Z_\ell} f^!\Z_\ell$ from \cref{rslt:functor-isomorphisms-for-f-dualizable-sheaves}, so that $f^!\Lambda$ is indeed invertible.

By the previous paragraph, all claims reduce to similar claims about relatively dualizable sheaves or follow formally from them. Namely, (i) is a special case of \cref{rslt:functor-isomorphisms-for-f-dualizable-sheaves}. The case in (ii).(a) where $f$ is $\ell$-cohomologically smooth also follows from \cref{rslt:functoriality-of-shriek-functors-in-Lambda}. The case where $g$ is $\ell$-cohomologically smooth can formally be reduced to that case, see \cite[Proposition 3.8.6.(iii)]{mann-mod-p-6-functors}. Part (ii).(b) is a special case of \cref{rslt:rel-dualizable-base-change} (it also follows easily from (ii).(a)). Part (iii) follows formally from (i), see \cite[Proposition 3.8.7]{mann-mod-p-6-functors}.
\end{proof}

The following result shows that the notion of relatively dualizable sheaves is $\ell$-cohomologically smooth local on the source.

\begin{proposition} \label{rslt:rel-dualizable-is-smooth-local-on-source}
Let $f\colon X \to S$ and $g\colon Y \to X$ be $\ell$-fine maps of small v-stacks and let $\mathcal P \in \D_\nuc(X,\Lambda)$. Assume that $g$ is $\ell$-cohomologically smooth.
\begin{propenum}
	\item \label{rslt:rel-dualizable-stable-under-smooth-pullback} If $\mathcal P$ is $f$-dualizable then $g^* \mathcal P$ is $(f \comp g)$-dualizable.
	\item \label{rslt:rel-dualizable-can-be-checked-on-smooth-cover} If $g^* \mathcal P$ is $(f\comp g)$-dualizable and $g$ is surjective then $\mathcal P$ is $f$-dualizable.
\end{propenum}
\end{proposition}
\begin{proof}
Part (i) is a special case of \cref{rslt:rel-dualizable-stable-under-rel-dualizable-pullback}. We now prove (ii), so assume that $g$ is surjective and that $g^* \mathcal P$ is $(f\comp g)$-dualizable. Let us denote by $h\colon Y \cprod_S Y \to X \cprod_S X$ the obvious map and for $i = 1, 2$ let $\pi_i\colon X \cprod_S X \to X$ and $\pi_i'\colon Y \cprod_S Y \to Y$ denote the $i$th projection. One checks easily that $\ell$-cohomologically smooth maps are stable under base-change and composition (the latter is actually a special case of (i)), which implies that $h$ is $\ell$-cohomologically smooth. Using also that invertible objects can be pulled out of upper shriek and internal Hom functors (which is easily checked with Yoneda) we deduce
\begin{align*}
	&h^*(\pi_1^* D_f(\mathcal P) \tensor \pi_2^* \mathcal P) = \pi_1'^* (g^* D_f(\mathcal P)) \tensor \pi_2'^* (g^* \mathcal P) = (\pi_1'^* g^! \Lambda)^{-1} \tensor \pi_1'^* (g^! D_f(\mathcal P)) \tensor \pi_2'^* (g^* \mathcal P) \\&\qquad= \pi_1'^* D_{f\comp g}(g^* \mathcal P) \tensor \pi_2'^* g^* \mathcal P \tensor (\pi_1'^* g^! \Lambda)^{-1},
\intertext{and using the fact that $g^* \mathcal P$ is $(f \comp g)$-dualizable and \cref{rslt:characterization-of-f-dualizable-sheaves},}
	&\qquad = \IHom(\pi_1'^* g^* \mathcal P, \pi_2'^! g^* \mathcal P) \tensor (\pi_1'^* g^! \Lambda)^{-1} = \IHom(\pi_1'^* g^* \mathcal P, \pi_2'^! g^! \mathcal P) \tensor (\pi_1'^* g^! \Lambda \tensor \pi_2'^* g^! \Lambda)^{-1},
\intertext{and by an easy computation via factoring $h$ as $Y \cprod_S Y \to X \cprod_S Y \to X \cprod_S X$ we have $(\pi_1'^* g^! \Lambda \tensor \pi_2'^* g^! \Lambda)^{-1} = h^! \Lambda$ and therefore}
	&\qquad = \IHom(h^* \pi_1^* \mathcal P, h^! \pi_2^! \mathcal P) \tensor (h^!\Lambda)^{-1} = h^!(\IHom(\pi_1^* \mathcal P, \pi_2^! \mathcal P) \tensor (h^! \Lambda)^{-1}\\
	&\qquad= h^* \IHom(\pi_1^* \mathcal P, \pi_2^! \mathcal P).
\end{align*}
All in all we see that the natural map $\pi_1^* D_f(\mathcal P) \tensor \pi_2^* \mathcal P \to \IHom(\pi_1^* \mathcal P, \pi_2^! \mathcal P)$ becomes an isomorphism after applying $h^*$. Since $h$ is surjective, this implies that $\mathcal P$ is indeed $f$-dualizable by \cref{rslt:characterization-of-f-dualizable-sheaves}.
\end{proof}

We also get the following stability properties of $\ell$-cohomologically smooth maps. Again, all of them are formal:

\begin{lemma}
\begin{lemenum}
	\item The condition of being $\ell$-cohomologically smooth is étale local on both source and target.

	\item Among $\ell$-fine maps the condition of being $\ell$-cohomologically smooth is v-local on the target and $\ell$-cohomologically smooth local on the source.

	\item $\ell$-fine and $\ell$-cohomologically smooth maps are stable under composition and base-change.

	\item Every étale map is $\ell$-cohomologically smooth.
\end{lemenum}
\end{lemma}
\begin{proof}
Part (iv) is obvious and part (i) is a special case of (ii). By \cref{rslt:v-descent-for-rel-dualizable-sheaves} $\ell$-cohomological smoothness is v-local on the target and by \cref{rslt:rel-dualizable-can-be-checked-on-smooth-cover} it is $\ell$-cohomologically smooth local on the source; this proves (ii). The claim about base-change in (iii) is a special case of (ii) and the claim about compositions follows immediately from \cref{rslt:rel-dualizable-stable-under-smooth-pullback}.
\end{proof}

It is similarly formal that $\ell$-cohomologically smooth maps satisfy universal $\ell$-codescent, which provides us with a lot of stacky $\ell$-fine maps.

\begin{lemma} \label{rslt:ell-fine-smooth-maps-admit-codescent}
Let $f\colon Y \to X$ be an $\ell$-fine and $\ell$-cohomologically smooth cover of small v-stacks. Then the natural functor
\begin{align*}
	\D^!_\nuc(X,\Lambda) \isoto \varprojlim_{n\in\Delta} \D^!_\nuc(Y_n,\Lambda)
\end{align*}
is an equivalence. Here $\D_\nuc^!$ denotes the functor whose transition maps are given by upper shriek functors.
\end{lemma}
\begin{proof}
The same argument as in \cite[Lemma 2.8]{mod-p-stacky-6-functors} applies.
\end{proof}

\begin{corollary} \label{rslt:smooth-maps-admit-codescent}
Every fdcs and $\ell$-cohomologically smooth map of small v-stacks admits universal $\ell$-codescent.
\end{corollary}
\begin{proof}
This is a special case of \cref{rslt:ell-fine-smooth-maps-admit-codescent}.
\end{proof}

With the formalism of $\ell$-cohomologically smooth maps at hand, we can answer the question how the solid 5-functor formalism from \cite[\S VII]{fargues-scholze-geometrization} relates to the nuclear 6-functor formalism. The following result generalizes \cite[Proposition VII.3.5]{fargues-scholze-geometrization}:

\begin{proposition} \label{rslt:comparison-of-natural-and-shriek-functor-fdcs-smooth}
Let $f\colon Y \to X$ be an fdcs and $\ell$-cohomologically smooth map of small v-stacks. Then $f_\natural\colon \D_\solid(Y,\Lambda) \to \D_\solid(X,\Lambda)$ preserves nuclear sheaves and the natural morphism
\begin{align*}
	f_\natural \isoto f_!(- \tensor f^! \Lambda)
\end{align*}
is an isomorphism of functors $\D_\nuc(Y,\Lambda) \to \D_\nuc(X,\Lambda)$.
\end{proposition}
\begin{proof}
It is enough to show that $f_\natural$ preserves nuclear sheaves, then the claimed isomorphism of functors follows easily from the fact that $f_\natural$ is left adjoint to $f^*$ (as functors on nuclear sheaves) and $f_!$ is left adjoint to $f^! = f^* \tensor f^! \Lambda$. Since $f_\natural$ commutes with the forgetful functor along the map $\Z_\ell \to \Lambda$ and satisfies arbitrary base-change (see \cite[Proposition VII.3.1]{fargues-scholze-geometrization}) we can formally reduce to the case that $X$ is a strictly totally disconnected perfectoid space and $\Lambda = \Z_\ell$. Then $Y$ is a locally spatial diamond, and by passing to an open cover of $Y$ we can assume that $Y$ is a spatial diamond (note that $f_\natural$ is computed as the colimit along an open cover). Note that $Y$ is $\ell$-bounded because it has finite $\dimtrg$ over the strictly totally disconnected space $X$.

We now need to show that for every nuclear $\mathcal M \in \D_\nuc(Y,\Z_\ell)$ the natural map $f_\natural \mathcal M \isoto f_!(\mathcal M \tensor f^! \Z_\ell)$ is an isomorphism in $\D_\solid(X,\Z_\ell)$. Since both sides commute with colimits in $\mathcal M$, we can assume that $\mathcal M$ is a right-bounded Banach sheaf. Note that $f_\natural$ preserves right-bounded $\ell$-adically complete objects: It preserves compact solid sheaves (i.e. finite (co)limits and retracts of objects of the form $\Z_{\ell,\solid}[U]$ for w-contractible $U \in Y_\proet$) and since the compact solid sheaves are $\ell$-adically complete one can argue similarly to the proof of \cref{rslt:IHom-tr-preserves-completeness}. Clearly also $f_!$ preserves $\ell$-adically complete sheaves (by computing it as the composition of an étale lower shriek and a pushforward), hence both $f_\natural \mathcal M$ and $f_!(\mathcal M \tensor f^! \Z_\ell)$ are $\ell$-adically complete. Therefore the desired isomorphism can be checked modulo $\ell$, so from now on we can replace $\Lambda = \Z_\ell$ by $\Lambda = \Fld_\ell$ and in particular assume that $\mathcal M$ is étale.

Using that both $f_\natural$ and $f_!$ preserve colimits, we can now reduce to the case that $\mathcal M = \Fld_\ell[U]$ for some $U \in Y_\et$. Then $\mathcal M = j_! \Fld_\ell = j_\natural \Fld_\ell$, where $j\colon U \to Y$ is the structure map. Hence by replacing $Y$ by $U$ we can further reduce to the case that $\mathcal M = \Fld_\ell$. We now end up with the claim that the natural map
\begin{align*}
	f_\natural \Fld_\ell \isoto f_! f^! \Fld_\ell
\end{align*}
is an isomorphism in $\D_\solid(X,\Fld_\ell)$. Note that $f^! \Fld_\ell$ is compact in $\D_\et(Y,\Fld_\ell)$ (because it is invertible), hence $f_! f^! \Fld_\ell$ is compact in $\D_\et(X,\Fld_\ell)$. In particular it is pseudocoherent in $\D_\solid(X,\Fld_\ell)$, i.e. $\Hom(f_! f^! \Fld_\ell, -)$ preserves uniformly left-bounded filtered colimits in $\D_\solid(X,\Fld_\ell)$ (this follows because $(-)_\et\colon \D_\solid(X,\Fld_\ell) \to \D_\et(X,\Fld_\ell)$ preserves uniformly left-bounded filtered colimits). Similarly $f_\natural \Fld_\ell$ is a pseudocoherent object of $\D_\solid(X,\Fld_\ell)$ because $\Fld_\ell$ is a pseudocoherent object of $\D_\solid(Y,\Fld_\ell)$. To prove the above isomorphism, we now need to show that for all $\mathcal N \in \D_\solid(X,\Fld_\ell)$ the natural map
\begin{align*}
	\Hom(f_\natural \Fld_\ell, \mathcal N) \isofrom \Hom(f_! f^! \Fld_\ell, \mathcal N)
\end{align*}
is an isomorphism of spectra. Via Postnikov limits we can assume that $\mathcal N$ is left-bounded. Writing $\mathcal N$ as a filtered colimit of its right truncations and using pseudocoherence of $f_\natural \Fld_\ell$ and $f_! f^! \Fld_\ell$ to pull out this colimit, we can further reduce to the case that $\mathcal N$ is bounded, which further reduces to the case that $\mathcal N$ is static. Now write $\mathcal N$ as a filtered colimit of static finitely presented sheaves in $\D_\solid(X,\Fld_\ell)$ to reduce to the case that $\mathcal N$ is static and finitely presented, i.e. of the form $\mathcal N = \varprojlim_i \mathcal N_i$ for some qcqs étale sheaves $\mathcal N_i \in \D_\et(X,\Fld_\ell)$ (see \cite[Theorem VII.1.3]{fargues-scholze-geometrization}). By pulling out this limit (which is automatically a derived limit by \cite[Proposition VII.1.6]{fargues-scholze-geometrization}) we can further reduce to the case that $\mathcal N$ is qcqs étale. But then
\begin{align*}
	& \Hom(f_! f^! \Fld_\ell, \mathcal N) = \Hom(f^! \Fld_\ell, f^! \mathcal N) = \Hom(f^! \Fld_\ell, f^* \mathcal N \tensor f^! \Fld_\ell) = \Hom(\Fld_\ell, f^* \mathcal N) =\\&\qquad= \Hom(f_\natural \Fld_\ell, \mathcal N),
\end{align*}
as desired.
\end{proof}

\begin{remark} \label{rslt:comparison-of-natural-and-shriek-functor-ell-fine-smooth}
One can generalize \cref{rslt:comparison-of-natural-and-shriek-functor-fdcs-smooth} to many $\ell$-fine and $\ell$-cohomologically smooth maps $f\colon Y \to X$ which are not necessarily fdcs. Here are two examples:
\begin{enumerate}[(a)]
	\item If there is an fdcs and $\ell$-cohomologically smooth cover $g\colon Z \surjto Y$ such that the composition $Z \to X$ is fdcs and $\ell$-cohomologically smooth, then the conclusion of \cref{rslt:comparison-of-natural-and-shriek-functor-fdcs-smooth} holds for $f$. Namely, in this case $f_\natural$ is computed as the colimit of the functors $(f \comp g_n)_\natural g_n^*$, where $g_\bullet\colon Z_\bullet \to Y$ is the Čech nerve of $g$.

	\item If $G$ is a virtually $\ell$-Poincaré group (see \cref{def:virtually-Poincare-group}) then the conclusion of \cref{rslt:comparison-of-natural-and-shriek-functor-fdcs-smooth} holds for $f\colon */G \to *$. Namely, by the proof of \cref{rslt:nuclear-sheaves-on-cohom-finite-classifying-stack} $*/G$ behaves very similarly to an $\ell$-bounded spatial diamond, so the proof of \cref{rslt:comparison-of-natural-and-shriek-functor-fdcs-smooth} applies.
\end{enumerate}
From these two examples we see that \cref{rslt:comparison-of-natural-and-shriek-functor-fdcs-smooth} generalizes to all $\ell$-fine and $\ell$-cohomologically smooth maps that appear in practice. However, we do not know if it holds for \emph{all} $\ell$-fine and $\ell$-cohomologically smooth maps because we cannot control maps admitting universal $\ell$-codescent well enough.
\end{remark}

\section{Cohomological Properness} \label{sec:properness}

Fix a prime $\ell \ne p$ and a nuclear $\Z_\ell$-algebra $\Lambda$. We now define a notion of \emph{$\ell$-cohomologically proper} maps of small v-stacks which roughly requires that lower shriek and pushforward along this map agree. Similar to the case of cohomologically smooth maps we will make use of the magical 2-category $\mathcal C_S$ constructed in \cref{def:magical-2-category} in order to reduce cohomological properties to a condition on the unit object.

For simplicity we will only study cohomological properness in the context where the diagonal is proper. One can extend the notion of cohomological properness to more general maps by mimicking the definition of smoothness: A map $f\colon X \to S$ can be called $\ell$-cohomologically proper if $\Z_\ell$ is $f$-proper and $P_f(\Z_\ell)$ is invertible (cf. \cref{def:f-proper-sheaves} for this terminology). We will not pursue this idea further as we see no useful application for it.

\begin{definition}
A map $f\colon Y \to X$ of small v-stacks is called \emph{1-separated} if the diagonal $\Delta_f\colon Y \to Y \cprod_X Y$ is proper.
\end{definition}

\begin{lemma} \label{rslt:stabilities-of-1-separated-maps}
\begin{lemenum}
	\item 1-separated morphisms of small v-stacks are stable under composition and base-change.

	\item The notion of 1-separatedness is v-local on the target.

	\item Let $f\colon Y \to X$ and $g\colon Z \to Y$ be maps of small v-stacks. If $f$ and $f \comp g$ are 1-separated then so is $g$.
\end{lemenum}
\end{lemma}
\begin{proof}
These results are all formal and follow in the same way as for algebraic stacks. More concretely, for (i) see \cite[Lemma 050K, 050F]{stacks-project}, for (ii) see \cite[Lemma 06TZ]{stacks-project} and \cite[Proposition 10.11.(ii)]{etale-cohomology-of-diamonds} and for (iii) see \cite[Lemma 050M]{stacks-project}.
\end{proof}

The notion of 1-separatedness is useful because it automatically gives us a comparison map between lower shriek and pushforward:

\begin{lemma} \label{rslt:1-separated-implies-morphism-of-shriek-to-star}
Let $f\colon Y \to X$ be an $\ell$-fine 1-separated map of small v-stacks with diagonal $\Delta$. Then the equivalence $\Delta_* = \Delta_!$ induces a morphism
\begin{align*}
	f_! \to f_*
\end{align*}
of functors $\D_\nuc(Y,\Lambda) \to \D_\nuc(X,\Lambda)$.
\end{lemma}
\begin{proof}
The desired morphism of functors is adjoint to a morphism $f^* f_! \to \id$ which can be constructed as follows: For $i = 1, 2$ let $\pi_i\colon Y \cprod_X Y \to Y$ denote the $i$th projection. Using proper base-change and the natural map $\id \to \Delta_* \Delta^* = \Delta_! \Delta^*$ we get the map
\begin{align*}
	f^* f_! = \pi_{2!} \pi_1^* \to \pi_{2!} \Delta_! \Delta^* \pi_1^* = \id,
\end{align*}
as desired.
\end{proof}

One can now define a 1-separated $\ell$-fine map $f\colon Y \to X$ to be $\ell$-cohomologically proper if the map of functors $f_! \to f_*$ is an isomorphism. However, this makes it hard to check cohomological properness in practice and it is also unclear why this notion is stable under base-change, so we prefer to work a little harder and use the magic of the 2-category from \cref{def:magical-2-category} to come up with a simpler definition. The crucial observation is the following (recall the notion of $f$-proper sheaves from \cref{def:f-proper-sheaves}):

\begin{lemma} \label{rslt:characterization-of-f-proper-sheaves-for-1-separated-map}
Let $f\colon Y \to X$ be a 1-separated $\ell$-fine map of small v-stacks and let $\mathcal P, \mathcal Q \in \D_\nuc(Y,\Lambda)$ be given. Suppose that the natural map $f_!(\mathcal P \tensor \mathcal Q) \isoto f_*(\mathcal P \tensor \mathcal Q)$ is an isomorphism. Then $\mathcal P$ is $f$-proper with $f$-proper dual $\mathcal Q$ if and only if $\mathcal P$ is dualizable with dual $\mathcal Q$.
\end{lemma}
\begin{proof}
The property of $\mathcal P$ being $f$-proper with dual $\mathcal Q$ is captured by the existence of a unit and a counit map satisfying certain commuting diagrams. A similar description holds for the property of $\mathcal P$ being dualizable with dual $\mathcal Q$, so all we need to do is to show that these data are equivalent. Let us investigate $f$-properness: The unit of an adjunction between $\mathcal P$ and $\mathcal Q$ is a map
\begin{align*}
	\varepsilon\colon \Lambda \to \mathcal P \star \mathcal Q = f_!(\mathcal P \tensor \mathcal Q) = f_*(\mathcal P \tensor \mathcal Q),
\end{align*}
which by adjunction of $f_*$ and $f^*$ is the same as a map $i\colon \Lambda \to \mathcal P \tensor \mathcal Q$, i.e. the same as a coevaluation map. Similarly, the counit of an adjunction between $\mathcal P$ and $\mathcal Q$ is a map
\begin{align*}
	\eta\colon \mathcal Q \star \mathcal P = \pi_1^* \mathcal Q \tensor \pi_2^* \mathcal P \to \Delta_!\Lambda = \Delta_* \Lambda,
\end{align*}
which by adjunction between $\Delta_*$ and $\Delta^*$ is the same as a map $\ev\colon \mathcal Q \tensor \mathcal P = \Delta^* (\pi_1^* \mathcal Q \tensor \pi_2^* \mathcal P) \to \Lambda$, i.e. the same as an evaluation map. It remains to see that the required compatibilities are the same for $f$-properness and dualizability. For $f$-properness, the first requirement is that the following solid map is the identity:
\begin{center}\begin{tikzcd}
	\mathcal P \arrow[r,equal] & f^*\Lambda \tensor \mathcal P \arrow[d,"\delta"] \arrow[r,"f^*\varepsilon \tensor \id"] & f^*f_!(\mathcal P \tensor \mathcal Q) \tensor \mathcal P \arrow[r,equal] \arrow[dl,dashed,"\beta"] \arrow[dd,dashed,"\alpha"] & \pi_{2!}(\pi_1^* \mathcal P \tensor \pi_1^* \mathcal Q \tensor \pi_2^* \mathcal P) \arrow[d,"\pi_{2!}(\id \tensor \eta)"]\\
	& f^* f_* (\mathcal P \tensor \mathcal Q) \tensor \mathcal P \arrow[dr,dashed,"\gamma"] && \pi_{2!}(\pi_1^* \mathcal P \tensor \Delta_!\Lambda) \arrow[d,equal]\\
	&& \mathcal P \tensor \mathcal Q \tensor \mathcal P \arrow[r,dashed,"\id \tensor \ev"] & \mathcal P
\end{tikzcd}\end{center}
Note that there is a dashed map $\alpha$ induced by the map $f^* f_! \to \id$ and one checks easily that the right-hand square commutes. Moreover, there is a dashed isomorphism $\beta$ induced by the map $f_! \to f_*$ and a dashed map $\gamma$ induced by the counit $f^* f_* \to \id$ such that the triangle consisting of $\alpha$, $\beta$ and $\gamma$ also commutes. The dashed map $\delta$ is chosen such that the diagram commutes. One checks easily that $\gamma \comp \delta = i \tensor \id$. This shows that the first $f$-properness requirement for $\mathcal P$ and $\mathcal Q$ is indeed the same as the first dualizability requirement for these sheaves. One argues similarly with the second requirement.
\end{proof}

We now get a simple yet effective definition of $\ell$-cohomologically proper maps with all the expected properties.

\begin{definition}
We say that an $\ell$-fine 1-separated map $f\colon Y \to X$ is \emph{$\ell$-cohomologically proper} if the induced map $f_! \Z_\ell \isoto f_* \Z_\ell$ is an isomorphism.
\end{definition}

\begin{lemma} \label{rslt:ell-cohomologically-proper-equiv-Z-ell-is-f-proper}
Let $f\colon Y \to X$ be an $\ell$-fine 1-separated map. Then $f$ is $\ell$-cohomologically proper if and only if $\Z_\ell \in \D_\nuc(Y,\Z_\ell)$ is $f$-proper with invertible $f$-proper dual. If this is the case then also $\Lambda \in \D_\nuc(Y,\Lambda)$ is $f$-proper and its $f$-proper dual is $\Lambda$.
\end{lemma}
\begin{proof}
If $f$ is $\ell$-cohomologically proper then by \cref{rslt:characterization-of-f-proper-sheaves-for-1-separated-map} $\Z_\ell$ is indeed $f$-proper and is its $f$-proper dual. It is clear that the same then holds for $\Lambda$ by considering the base-change functor along $\Z_\ell \to \Lambda$ between the associated versions of $\mathcal C_X$.

Conversely assume that $\Z_\ell$ is $f$-proper and $P_f(\Z_\ell)$ is invertible. Consider the functor from $\mathcal C_X$ to the category of (underlying 1-categories of) stable $\infty$-categories sending $\mathcal M$ to $\pi_{2!}(\mathcal M \tensor \pi_1^*)$ (as in the proof of \cref{rslt:characterization-of-f-dualizable-sheaves}). This functor preserves adjoint functors, which implies that the functor $f_! = f_!(\Z_\ell \tensor -)$ is 1-categorically right adjoint to the functor $P_f(\Z_\ell) \tensor f^*$. We deduce that there is a 1-categorical isomorphism of functors $f_!(\Z_\ell \tensor -) \isom f_* \IHom(P_f(\Z_\ell), -)$. By the magic of $\mathcal C_X$ the same still holds after pullback along $f$, i.e. if $\pi_i\colon Y \cprod_X Y \to Y$ denote the two projections then there is a 1-categorical isomorphism of functors $\pi_{2*} \IHom(\pi_1^* P_f(\Z_\ell), -) \isom \pi_{2!}(\pi_1^* \Z_\ell \tensor -)$. Plugging in $\Delta_!\Z_\ell$ yields
\begin{align*}
	&\Z_\ell = \pi_{2!}(\pi_1^* \Z_\ell \tensor \Delta_!\Z_\ell) \isom \pi_{2*} \IHom(\pi_1^* P_f(\Z_\ell), \Delta_! \Z_\ell) = \pi_{2*} \IHom(\pi_1^* P_f(\Z_\ell), \Delta_* \Z_\ell) \\&\qquad= \pi_{2*} \Delta_* \IHom(\Delta^* \pi_1^* P_f(\Z_\ell), \Z_\ell) = \IHom(P_f(\Z_\ell), \Z_\ell).
\end{align*}
Since $P_f(\Z_\ell)$ is invertible, this implies $P_f(\Z_\ell) \isom \Z_\ell$. In particular we obtain a 1-categorical isomorphism of functors $f_! \isom f_* \IHom(P_f(\Z_\ell), -) \isom f_*$. Both isomorphisms are induced from the adjunction of $\Z_\ell$ with $P_f(\Z_\ell)$ and the second isomorphism additionally used the fact $\Delta_! = \Delta_*$ to get the isomorphism $P_f(\Z_\ell) \isom \Z_\ell$. One checks that these isomorphisms ``cancel'', so that the obtained isomorphism $f_! \isom f_*$ is indeed given by the natural map $f_! \to f_*$. In particular it is an isomorphism of $\infty$-functors and plugging in $\Z_\ell$ we obtain that $f$ is $\ell$-cohomologically proper.
\end{proof}

In the following, when we require a map to be $\ell$-cohomologically proper then we always assume that it is additionally $\ell$-fine and 1-separated.

\begin{proposition}
Let $f\colon Y \to X$ be an $\ell$-cohomologically proper map of small v-stacks. Then the natural map $f_! \isoto f_*$ is an isomorphism of functors $\D_\nuc(Y,\Lambda) \to \D_\nuc(X,\Lambda)$.
\end{proposition}
\begin{proof}
This was part of the proof of \cref{rslt:ell-cohomologically-proper-equiv-Z-ell-is-f-proper}.
\end{proof}

\begin{lemma}
\begin{lemenum}
	\item Every $\ell$-fine proper map of small v-stacks is $\ell$-cohomologically proper.

	\item $\ell$-cohomologically proper maps are stable under composition and base-change.

	\item Among $\ell$-fine maps, the condition of being $\ell$-cohomologically proper is v-local on the target.

	\item Let $f\colon Y \to X$ and $g\colon Z \to Y$ be maps of small v-stacks. If $f$ and $f \comp g$ are $\ell$-cohomologically proper then so is $g$.
\end{lemenum}
\end{lemma}
\begin{proof}
Note that all of the stabilities are satisfied by 1-separated and by $\ell$-fine maps by \cref{rslt:stabilities-of-1-separated-maps,rslt:stabilities-of-ell-fine-maps}, which we will use without mention. Part (i) is clear. In part (ii), stability under base-change follows from \cref{rslt:ell-cohomologically-proper-equiv-Z-ell-is-f-proper} because $f$-proper sheaves are certainly stable under base-change (same as for $f$-dualizable sheaves). Stability under composition follows immediately from the definition. In part (iii), if a map is $\ell$-cohomologically proper on a v-cover, then the pushforward along the base-changes of this map preserves coCartesian edges in the associated Čech cover and therefore the pushforward along the map commutes with base-change; then it follows immediately that the map is $\ell$-cohomologically proper. To prove (iv) assume that $f$ and $g$ are given as in the claim and that both $f$ and $f \comp g$ are $\ell$-cohomologically proper. Then we can factor $g$ as $Z \to Z \cprod_X Y \to Y$. The second map is $\ell$-cohomologically proper because it is a base-change of $f \comp g$, so it remains to see that the first map is $\ell$-cohomologically proper. We can thus assume that $f \comp g = \id$, i.e. $g\colon X \to Y$ is a section of the $\ell$-cohomologically proper map $f\colon Y \to X$. We have a cartesian square of small v-stacks (cf. the proof of \cite[Lemma 050H]{stacks-project})
\begin{center}\begin{tikzcd}
	X = X \cprod_Y Y \arrow[r] \arrow[d] & Y = X \cprod_X Y \arrow[d]\\
	Y \arrow[r,"\Delta_f"] & Y \cprod_X Y
\end{tikzcd}\end{center}
The top map is evidently $g$ and the bottom map is proper by 1-separatedness of $f$. Thus $g$ is proper and in particular $\ell$-cohomologically proper.
\end{proof}

\begin{remark}
We see that the magical 2-category $\mathcal C_X$ allows us to do two things: Firstly we can check cohomological properness on the unit object and secondly we immediately see that cohomological properness is stable under base-change, which is not at all obvious.
\end{remark}

With a good notion of $\ell$-cohomologically proper maps at hand, we can now prove that relatively dualizable sheaves are stable under proper pushforward:

\begin{proposition} \label{rslt:rel-dualizable-stable-under-proper-pushforward}
Let $f\colon Y \to X$ and $g\colon Z \to Y$ be $\ell$-fine maps and $\mathcal P \in \D_\nuc(Z,\Lambda)$ such that $g$ is $\ell$-cohomologically proper and $\mathcal P$ is $(f\comp g)$-dualizable. Then $g_*\mathcal P$ is $f$-dualizable.
\end{proposition}
\begin{proof}
Combine \cref{rslt:rel-dualizable-stable-under-f-proper-pushforward} with \cref{rslt:ell-cohomologically-proper-equiv-Z-ell-is-f-proper}.
\end{proof}

\section{Classifying Stacks and Representations} \label{sec:representations}

Fix a prime $\ell \ne p$ and a nuclear $\Z_\ell$-algebra $\Lambda$. Throughout this section we work with small v-stacks over $\Spec \overline\Fld_p$. Our goal is to apply the above theory of nuclear sheaves to the classifying stack of a locally profinite group $G$. This will give us an $\infty$-category of \emph{nuclear $G$-representations} together with a robust notion of admissible representations and a full 6-functor formalism.

Before we can start studying classifying stacks, we need to get one technicality out of the way: The final object $*$ is not representable in perfectoid spaces. We can use the same arguments as in the case of discrete coefficients to compute sheaves on $*$:

\begin{proposition} \label{rslt:solid-sheaves-invariant-under-base-change-of-fields}
Let $X \to S \from S'$ be a diagram of small v-stacks and assume that $S$ and $S'$ satisfy one of the following conditions:
\begin{propenum}
	\item Both $S$ and $S'$ are spectra of some algebrically closed discrete field of characteristic $p$.

	\item $S$ is the spectrum of some algebraically closed discrete field of characteristic $p$ and $S' = \Spa(C',C'^+)$ for some algebraically closed non-archimedean field $C'$ and an open and bounded valuation subring $C'^+$.

	\item $S = \Spa(C, C^+)$ and $S' = \Spa(C', C'^+)$ for algebraically closed non-archimedean fields and open and bounded valuation subrings $C^+$ and $C'^+$.
\end{propenum}
Let $X' := X \cprod_S S'$. Then the pullback functor
\begin{align*}
	\D_\solid(X,\Lambda) \injto \D_\solid(X',\Lambda)
\end{align*}
is fully faithful and hence also induces fully faithful pullback functors on nuclear and on $\omega_1$-solid sheaves.
\end{proposition}
\begin{proof}
This is a generalization of \cite[Theorem 19.5]{etale-cohomology-of-diamonds} and for a large part we can argue very similarly. We start with some general observations which are valid in all cases. By taking a hypercover of $X$ in terms of disjoint unions of strictly totally disconnected spaces and using the fact that fully faithfulness is preserves under limits of $\infty$-categories, we can reduce to the case that $X = \Spa(A, A^+)$ is strictly totally disconnected. Now let $f\colon X' \to X$ denote the natural map; we need to show that for all $\mathcal M \in \D_\solid(X,\Lambda)$ the natural morphism $\mathcal M \isoto f_{\vsite*} f^* \mathcal M$ is an isomorphism. We can assume $\Lambda = \Z_\ell$ and by the usual Postnikov argument we can also assume that $\mathcal M$ is static.

We now come to the specific proofs of each claim. First we note that (i) follows from (ii) and (iii). Moreover, part (iii) is rather easy with the above preparations: Now $f$ is qcqs and hence $f_{\vsite*}$ preserves uniformly left-bounded filtered colimits. By writing $\mathcal M$ as a filtered colimit of finitely presented static solid sheaves, we can assume that $\mathcal M$ is finitely presented and hence a cofiltered limit of qcqs étale sheaves. Since both $f^*$ and $f_{\vsite*}$ preserve limits of solid sheaves, we end up with the case of a qcqs étale sheaf in $\D_\solid(X,\Z_\ell)$. This sheaf is killed by some power of $\ell$, so the claim follows from \cite[Theorem 19.5.(iii)]{etale-cohomology-of-diamonds} (more concretely from the final part of \cite[Theorem 16.1]{etale-cohomology-of-diamonds}).

It remains to prove (ii). By (iii) it is enough to prove it in the case that $C$ is the completed algebraic closure of $k((t))$, where $S = \Spec k$. Fix a pseudouniformizer $\pi \in A$, so that $X'$ can be written as the increasing union of the affinoid perfectoid subspaces
\begin{align*}
	X'_n = \{ \abs t \le \abs \pi \le \abs{t^{1/n}} \} \subset X'.
\end{align*}
Now $f_{\vsite*} f^*$ is the limit over the functors $f_{n\vsite*} f_n^*$, where $f_n\colon X'_n \to X$ is the natural map. It is therefore enough to show that for all $n$ the morphism $\mathcal M \isoto f_{n\vsite*} f_n^* \mathcal M$ is an isomorphism. But $f_n$ is qcqs, so by the same reasoning as in the proof of (iii) we can reduce to the case that $\mathcal M$ is étale and thus reduce to \cite[Theorem 19.5.(ii)]{etale-cohomology-of-diamonds}.
\end{proof}

\begin{corollary} \label{rslt:solid-sheaves-on-*}
Let $S$ be a locally profinite set and let $\underline S$ denote the associated small v-stack (so that if $S$ is a point then $\underline S = *$). Then there is a natural equivalence of $\infty$-categories
\begin{align*}
	\D_\solid(\underline S,\Lambda) = \D_\solid(S,\Lambda),
\end{align*}
where on the right-hand side we mean solid sheaves on the pro-étale site of $S$. The same is true for nuclear and $\omega_1$-solid $\Lambda$-sheaves.
\end{corollary}
\begin{proof}
Fix any algebraically closed non-archimedean field $C$. Then $\D_\solid(S,\Lambda) = \D_\solid(\underline S \cprod \Spa C, \Lambda)$, so by \cref{rslt:solid-sheaves-invariant-under-base-change-of-fields} the pullback functor induces an embedding $\D_\solid(\underline S,\Lambda) \injto \D_\solid(S,\Lambda)$. But this embedding has a section obtained by pullback along the map of sites $\underline S_\vsite \to S_\proet$.
\end{proof}

With the technicalities involving the final object $*$ out of the way, we can now come to the representation theory. Fix a locally profinite group $G$. In \cite[\S3.4]{mann-mod-p-6-functors} we constructed the $\infty$-category $\D_\solid(\Z_\ell)^{BG}$ of continuous $G$-representations on solid $\Z_\ell$-modules. It can be identified with the derived $\infty$-category of solid $\Z_{\ell,\solid}[G]$-modules. By passing to $\Lambda$-modules we similarly obtain the $\infty$-category $\D_\solid(\Lambda,\Z_\ell)^{BG}$ of continuous $G$-representations on $(\Lambda,\Z_\ell)_\solid$-modules (we write $(\Lambda,\Z_\ell)_\solid$ here to denote the solid structure induced from $\Z_\ell$; this is the analog of $\D_\solid(X,\Lambda)$ for small v-stacks $X$ by somewhat sloppy notation in the latter case). Also recall the definition of $\ell$-cohomological dimension of locally profinite groups in \cite[Definition 3.4.20]{mann-mod-p-6-functors} (see also \cite[Proposition 3.4.22]{mann-mod-p-6-functors}). We get the following interpretation of sheaves on classifying stacks:

\begin{lemma} \label{rslt:solid-sheaves-on-classifying-stacks}
Let $G$ be a locally profinite group. Then there is a natural equivalence of $\infty$-categories
\begin{align*}
	\D_\solid(*/G,\Lambda) = \D_\solid(\Lambda,\Z_\ell)^{BG}.
\end{align*}
Moreover, we have the following:
\begin{lemenum}
	\item Under the above equivalence, the pushforward functor along the natural projection $*/G \to *$ corresponds to the functor $\Gamma(G,-)\colon \D_\solid(\Lambda,\Z_\ell)^{BG} \to \D_\solid(\Lambda,\Z_\ell)$ computing (continuous) group cohomology.

	\item \label{rslt:finite-group-cohom-on-solid-sheaves} Suppose that $G$ is profinite and $\cd_\ell G < \infty$. Then the functor $\Gamma(G,-)$ has cohomological dimension $\le \cd_\ell G + 1$ and thus preserves all small colimits.
\end{lemenum}
\end{lemma}
\begin{proof}
For all claims we can assume $\Lambda = \Z_\ell$. To prove the claimed equivalence of $\infty$-categories we note that both $\infty$-categories admit natural left-complete $t$-structures and all the pullback functors in the Čech nerve of the cover $* \to */G$ are $t$-exact, hence by \cite[Proposition A.1.2.(ii)]{mann-mod-p-6-functors} the claim can be verified on the hearts. The heart of the left-hand category admits a simple description in terms of descent data (cf. \cite[Proposition A.1.2.(i)]{mann-mod-p-6-functors}), which one easily verifies to be the same as a continuous $G$-representation (here we implicitly use \cref{rslt:solid-sheaves-on-*}).

It remains to prove claims (i) and (ii). Claim (i) follows immediately by comparing the associated left adjoints. For claim (ii) we note that if $G$ is profinite then $*/G$ is qcqs, hence $\Gamma(G,-)$ commutes with filtered colimits of static objects. Thus the cohomological dimension of $\Gamma(G,-)$ can be checked on $\ell$-adically complete objects (recall that $\D_\solid(\Z_\ell)^{BG}$ is compactly generated by the $\ell$-adically complete objects $\Z_{\ell,\solid}[G]$), where we can pull the limit $\varprojlim_n \mathcal M/\ell^n \mathcal M$ out of the cohomology, so that the claim follows from the fact that countable limits have cohomological dimension $1$ and that mod $\ell^n$ the claim follows inductively from the definition and the standard cofiber sequences $M/\ell^{n-1}M \to M/\ell^n M \to M/\ell M$.
\end{proof}

With \cref{rslt:solid-sheaves-on-classifying-stacks} at hand we now study the full subcategory spanned by the nuclear sheaves. The corresponding representations are the continuous representations on nuclear modules:

\begin{definition}
Let $G$ be a locally profinite group. We denote by $\D_\nuc(\Lambda)^{BG} \subset \D_\solid(\Lambda,\Z_\ell)^{BG}$ the full subcategory spanned by those $G$-representations whose underlying $\Lambda$-module is nuclear. The objects of $\D_\nuc(\Lambda)^{BG}$ are called the \emph{nuclear $G$-representations} over $\Lambda$.
\end{definition}

\begin{lemma} \label{rslt:compute-nuclear-sheaves-on-classifying-stack}
Let $G$ be a locally profinite group. Then there is a natural equivalence of $\infty$-categories
\begin{align*}
	\D_\nuc(*/G,\Lambda) = \D_\nuc(\Lambda)^{BG}.
\end{align*}
Moreover, the $t$-structure on $\D_\solid(\Lambda,\Z_\ell)^{BG}$ restricts to a complete $t$-structure on $\D_\nuc(\Lambda)^{BG}$. The heart $\mathcal A := (\D_\nuc(\Lambda)^{BG})^\heartsuit$ is a Grothendieck abelian category which is stable under kernels, cokernels and extensions in the heart of $\D_\solid(\Lambda,\Z_\ell)^{BG}$. If $\Lambda$ is static then there is a natural equivalence $\D^+_\nuc(\Lambda)^{BG} = \D^+(\mathcal A)$.
\end{lemma}
\begin{proof}
The first claim follows immediately from \cref{rslt:solid-sheaves-on-classifying-stacks} because on both sides of the claimed equivalence the nuclear sheaves are characterized by those solid sheaves which become nuclear after pullback to $*$ (respectively, after forgetting the $G$-action).

That the $t$-structure on $\D_\solid(\Lambda,\Z_\ell)^{BG}$ restricts to a $t$-structure on nuclear representations can be checked on underlying $\Lambda$-modules, i.e. we need to see that the $t$-structure on $\D_\solid(\Lambda,\Z_\ell)$ restricts to a $t$-structure on $\D_\nuc(\Lambda)$. We can further assume $\Lambda = \Z_\ell$. Then the claim boils down to the observation that truncations of Banach $\Z_\ell$-modules are again Banach $\Z_\ell$-modules, which follows from the fact that $\ell$-adically complete objects are stable under truncations and that discreteness mod $\ell$ is stable under truncations because discretization is $t$-exact (the latter property fails on spatial diamonds, already on profinite sets). The completeness of the $t$-structure follows immediately from the completeness of the $t$-structure of the ambient $\infty$-category $\D_\solid(\Lambda,\Z_\ell)^{BG}$.

Now let $\mathcal A$ denote the heart of $\D_\nuc(\Lambda)^{BG}$. It is clear that $\mathcal A$ is stable under kernels, cokernels and extensions because all of these can be constructed using finite (co)limits and truncations in $\D_\solid(\Lambda,\Z_\ell)^{BG}$. It is also clearly stable under filtered colimits, hence filtered colimits in $\mathcal A$ are exact. It is formal that $\mathcal A$ is a Grothendieck abelian category: By descent, $\D_\nuc(\Lambda)^{BG}$ is presentable and clearly its $t$-structure is accessible in the sense of \cite[Proposition 1.4.4.13]{lurie-higher-algebra}; thus the claim follows from \cite[Remark 1.3.5.23]{lurie-higher-algebra}.

Now assume that $\Lambda$ is static. Then for all profinite sets $S$ the solid $G$-representation $(\Lambda,\Z_\ell)_\solid[G \cprod S] = \Lambda \tensor_{\Z_{\ell,\solid}} \Z_{\ell,\solid}[G \cprod S]$ is static because $\Z_{\ell,\solid}[G \cprod S]$ is flat over $\Z_{\ell,\solid}$ (by \cref{rslt:solid-tensor-product-preserves-complete-sheaves} this flatness reduces to the flatness of $\Fld_{\ell,\solid}[G \cprod S]$ over $\Fld_{\ell,\solid}$ which can be deduced easily from the fact that compact $\Fld_{\ell,\solid}$-modules are stable under kernels and cokernels, cf. the proof of \cite[Lemma 2.9.35]{mann-mod-p-6-functors}). But $\D_\solid(\Lambda,\Z_\ell)^{BG}$ is the $\infty$-category of modules over the associative analytic ring $(\Lambda,\Z_\ell)_\solid[G]$ and since all the compact projective generators of this ring are static it follows that $\D_\solid(\Lambda,\Z_\ell)^{BG}$ is the derived $\infty$-category of its heart. Now consider the natural functor
\begin{align*}
	F\colon \D^+(\mathcal A) \to \D^+_\solid(\Lambda,\Z_\ell)^{BG}.
\end{align*}
It admits a right adjoint $RG$ which is the right derived functor of the nuclearization functor $G$ on the hearts. We claim that the unit $\id \isoto RG \comp F$ is an isomorphism. This can be checked on static representations, i.e. for $M \in \mathcal A$ we need to see that $M \isoto RG(F(M))$ is an isomorphism. Let $M \to I^\bullet$ be an injective resolution of $M = F(M)$ in the heart of $\D_\solid(\Lambda,\Z_\ell)^{BG}$. Then $RG(F(M))$ is represented by the complex $G(I^\bullet)$. On the other hand, this complex clearly also computes $M_\nuc$ in $\D_\solid(\Lambda,\Z_\ell)^{BG}$. Since $M$ is nuclear this implies that this complex is indeed isomorphic to $M$ in either derived $\infty$-category. This proves that $F$ is fully faithful. The essential image is stable under finite and filtered colimits and contains $\mathcal A$ and is therefore precisely $\D^+_\nuc(\Lambda)^{BG}$.
\end{proof}

\begin{proposition} \label{rslt:nuclear-sheaves-on-cohom-finite-classifying-stack}
Let $G$ be a locally profinite group with locally finite $\ell$-cohomological dimension.
\begin{propenum}
	\item If $\Lambda$ is $\ell$-adically complete and $G$ is profinite then $\D_\nuc(\Lambda)^{BG}$ is generated under filtered colimits by $\ell$-adically complete nuclear $G$-representations.

	\item \label{rslt:group-cohom-preserves-nuclear-and-colim} If $G$ is profinite and $\cd_\ell G < \infty$ then the v-pushforward along $*/G \to *$ preserves nuclear sheaves and thus restricts to a colimit-preserving functor $\D_\nuc(*/G,\Lambda) \to \D_\nuc(\Lambda)$.

	\item If $\Lambda$ is static then $\D_\nuc(\Lambda)^{BG}$ is the derived $\infty$-category of its heart.
\end{propenum}
\end{proposition}
\begin{proof}
We first prove (i), so assume that $G$ is profinite. Let us furthermore assume that $\cd_\ell G < \infty$. Then by \cref{rslt:finite-group-cohom-on-solid-sheaves} the étale site of the stack $*/G$ behaves in a very similar way as it does for an $\ell$-bounded spatial diamond. In particular by the same arguments as in \cref{sec:nuclear} we see that $\D_\nuc(\Lambda)^{BG}$ is $\omega_1$-compactly generated by the basic nuclear sheaves, which themselves are sequential colimits of $\ell$-adically complete sheaves. We also get a good description of the $\omega_1$-solid sheaves on $*/G$ in a similar fashion as in \cref{sec:w1-solid} and the nuclearization functor $(-)_\nuc\colon \D_\solid(\Lambda,\Z_\ell)_{\omega_1}^{BG} \to \D_\nuc(\Lambda)^{BG}$ preserves all small colimits and is bounded. In particular we deduce that every nuclear $G$-representation is a filtered colimit of $\ell$-adically complete nuclear $G$-representations.

To finish the proof of (i) we still need to treat the case that $\cd_\ell G = \infty$. But by assumption on $G$ there is an open compact subgroup $H \subset G$ such that $\cd_\ell H < \infty$. The conservative pullback along $*/H \surjto */G$ has a left adjoint (the lower shriek functor, see \cref{rslt:properties-of-etale-lower-shriek}) on nuclear sheaves. Therefore, since $\D_\nuc(\Lambda)^{BH}$ is $\omega_1$-compactly generated, the same follows for $\D_\nuc(\Lambda)^{BG}$ and the $\omega_1$-compact generators are the shriek pushforwards along $*/H \surjto */G$ of the $\omega_1$-compact generators of $\D_\nuc(\Lambda)^{BH}$. This implies (i).

We now prove (ii) so assume that $G$ is as in the claim. We can assume $\Lambda = \Z_\ell$. By \cref{rslt:finite-group-cohom-on-solid-sheaves} the v-pushforward along $*/G \to *$ preserves small colimits of solid sheaves, hence by (i) we only need to show that this v-pushforward preserves Banach sheaves. This in turn reduces to showing that it preserves étale sheaves, i.e. that the continuous group cohomology of a discrete representation is discrete. This can for example be checked by a direct computation of group cohomology, cf. \cite[Proposition 3.4.6]{mann-mod-p-6-functors}.

We now prove (ii), so assume that $\Lambda$ is static and let $\mathcal A$ be the heart of $\D_\nuc(\Lambda)^{BG}$. By \cref{rslt:compute-nuclear-sheaves-on-classifying-stack} we have $\D^+(\mathcal A) = \D^+_\nuc(\Lambda)^{BG}$ and $\D_\nuc(\Lambda)^{BG}$ is left-complete. It is therefore enough to show that $\D(\mathcal A)$ is left-complete. For this it is enough to show that countable products have finite cohomological dimension in $\D(\mathcal A)$ (e.g. by adapting the proof of \cite[Proposition 1.2.1.19]{lurie-higher-algebra}). Equivalently we need to show that countable products in $\D_\nuc(\Lambda)^{BG}$ have finite cohomological dimension. This can be checked after pullback along any étale cover of $*/G$ (such a pullback is $t$-exact and preserves limits of nuclear sheaves by the existence of the left adjoint lower shriek functor), so we can replace $G$ by any open subgroup. In particular we can assume that $G$ is profinite and $\cd_\ell G < \infty$. Limits in $\D_\nuc(\Lambda)^{BG}$ can be computed as the composition of the limit in $\D_\solid(\Lambda,\Z_\ell)_{\omega_1}^{BG}$ and the nuclearization functor. By the proof of (i) nuclearization has finite cohomological dimension, hence so do countable products in $\D_\nuc(\Lambda)^{BG}$.
\end{proof}

We have acquired a clear understanding of the connection of nuclear sheaves on classifying stacks and nuclear representations. We now study the geometry of these classifying stacks from an $\ell$-cohomological viewpoint. Our first goal is to show that essentially all maps of classifying stacks that appear in practice are $\ell$-fine.

\begin{proposition} \label{rslt:vanishing-of-Ext-for-complete-reps}
Let $G$ be a profinite group and let $M, N \in \D_\nuc(\Z_\ell)^{BG}$ be static nuclear $G$-representations such that $N$ is $\ell$-adically complete. Then
\begin{align*}
	\Ext^k(M, N) = 0 \qquad \text{for all $k > \cd_\ell G + 3$}.
\end{align*}
\end{proposition}
\begin{proof}
We can assume that $\cd_\ell G < \infty$ because otherwise there is nothing to prove. Let $\IHom_\solid(M, N)$ denote the internal hom of $M$ and $N$ in $\D_\solid(\Z_\ell)^{BG}$. Then for the spectra-enriched Hom from $M$ to $N$ we have $\Hom(M, N) = \Gamma(G, \IHom_\solid(M, N))$. By \cref{rslt:finite-group-cohom-on-solid-sheaves} it is therefore enough to show that $\IExt^k_\solid(M, N) = 0$ for $k > 2$. Note that the solid internal hom is computed on the underlying $\Z_{\ell,\solid}$-modules (it agrees with the pro-étale internal hom and is thus preserved under pullback along the pro-étale map $* \to */G$). We can therefore ignore the group action from now on and simply assume that $M, N \in \D_\nuc(\Z_\ell)$ with $N$ being $\ell$-adically complete. By pulling out the limit $N = \varprojlim_n N/\ell^n N$ and using the fact that countable limits have cohomological dimension 1, we reduce to showing that $\IExt^k_\solid(M, N/\ell^n N) = 0$ for $i > 1$. In other words, from now on we can assume that $N$ is a discrete $\Z/\ell^n\Z$-module. Then $\IHom_\solid(M, N) = \IHom_{\Z/\ell^n\Z,\solid}(M/\ell^n M, N)$. Since $M/\ell^n M$ is discrete, we can equivalently write it as $M/\ell^n M = M_0/\ell^n M_0$, where $M_0$ is the underlying discrete abelian group of $M$. We can pick a short exact sequence $0 \to \bigdsum_J \Z \to \bigdsum_I \Z \to M_0 \to 0$ for some sets $I$ and $J$. Thus $M_0/\ell^n M_0$ admits a resolution of length 2 in terms of direct sums of copies of $\Z/\ell^n \Z$. Since products are exact in $\D_\solid(\Z/\ell^n \Z)$ it follows immediately that $\IHom_{\Z/\ell^n\Z,\solid}(M_0/\ell^n M_0, N)$ is concentrated in cohomological degrees $0$ and $1$, as desired.
\end{proof}

\begin{lemma} \label{rslt:map-from-*-to-*-mod-G-has-codescent}
Let $G$ be a profinite group with $\cd_\ell G < \infty$. Then the map $* \to */G$ is fdcs and admits universal $\ell$-codescent.
\end{lemma}
\begin{proof}
It is clear that the map $* \to */G$ is fdcs because it is proper and pro-étale. To prove universal $\ell$-codescent we follow our argument in \cite[Lemma 3.11]{mod-p-stacky-6-functors}, so the reader is encouraged to have a look at loc. cit. for more details. Pick any small v-stack $X$ with a map $X \to */G$ and let $Y := X \cprod_{*/G} *$. We denote $q\colon Y \to X$ the base-change of $* \to */G$ and $q_\bullet\colon Y_\bullet \to X$ the associated Čech nerve. Then we need to show that the natural functor
\begin{align*}
	\D_\nuc^!(X,\Z_\ell) \isoto \varprojlim_{n\in\Delta} \D_\nuc^!(Y_n,\Z_\ell)
\end{align*}
is an equivalence. By employing Lurie's Beck-Chevalley condition (see \cite[Corollary 4.7.5.3]{lurie-higher-algebra}) this reduces to the following claim:
\begin{enumerate}[(a)]
	\item The functor $q^!\colon \D_\nuc(X,\Z_\ell) \to \D_\nuc(Y,\Z_\ell)$ is conservative and preserves geometric realizations of $q^!$-split simplicial objects in $\D_\nuc(X,\Z_\ell)$.
\end{enumerate}
To prove this we apply ideas of Mathew \cite{akhil-galois-group-of-stable-homotopy}. We first note that it is enough to show the following claim:
\begin{enumerate}[(b)]
	\item Let $\langle q_* q^! \rangle \subset \Fun(\D_\nuc(X,\Z_\ell), \D_\nuc(X,\Z_\ell))$ be the full subcategory generated by $q_* q^!$ under finite (co)limits, compositions and retracts; then $\langle q_* q^! \rangle$ contains the identity functor.
\end{enumerate}
One checks easily that (b) implies (a): It follows easily from (b) that the functor $q_* q^!$ is conservative, hence so is $q^!$. Moreover, if $\mathcal M_\bullet$ is any $q^!$-split simplicial object in $\D_\nuc(X,\Lambda)$ then it is also $q_* q^!$-split and by (b) it follows that it is split; then of course its geometric realization commutes with $q^!$.

It remains to prove (b). We compute $q_* q^! = \IHom(q_* \Z_\ell, -)$, so (b) reduces to the claim that $\Z_\ell \in \D_\nuc(X,\Z_\ell)$ can be generated using finite (co)limits, retracts and tensor products from $q_* \Z_\ell$. In other words, using Mathew's notion of descendable algebras (see \cite[Definition 3.18, Proposition 3.20]{akhil-galois-group-of-stable-homotopy}) the claim (b) boils down to:
\begin{enumerate}[(c)]
	\item The algebra object $q_* \Z_\ell \in \D_\nuc(X,\Z_\ell)$ admits descent.
\end{enumerate}
Let us denote $q_0\colon * \to */G$ the canonical projection, so that $q$ is a base-change of $q_0$ along the map $f\colon X \to */G$. The pullback functor $f^*\colon \D_\nuc(*/G,\Z_\ell) \to \D_\nuc(X,\Z_\ell)$ is symmetric monoidal and sends $q_{0*} \Z_\ell$ to $q_* \Z_\ell$ (by proper base-change), so it is enough to show (c) for $q_0$. Thus from now on we assume $X = */G$. By \cref{rslt:nuclear-sheaves-on-cohom-finite-classifying-stack} the proof of (c) reduces to the following claim:
\begin{enumerate}[(d)]
	\item The algebra object $\cts(G,\Z_\ell) \in \D_\nuc(\Z_\ell)^{BG}$ admits descent.
\end{enumerate}
Let $d := \cd_\ell G + 3$ and consider the cofiber sequence
\begin{align*}
	\Z_\ell \to \Tot_d(\cts(G,\Z_\ell)^{\tensor\bullet+1}) \to X
\end{align*}
in $\D_\nuc(\Z_\ell)^{BG}$, where $\Tot_d$ denotes the $d$-truncated totalization of a cosimplicial object. By descent we have $\Tot(\cts(G,\Z_\ell)^{\tensor\bullet+1})$ and thus $X$ is concentrated in cohomological degrees $\ge d$. But then it follows from \cref{rslt:vanishing-of-Ext-for-complete-reps} that the connecting map $X \to \Z_\ell[1]$ must be zero, which implies that $\Z_\ell$ is a retract of $\Tot_d(\cts(G,\Z_\ell)^{\tensor\bullet+1})$, as desired.
\end{proof}

\begin{proposition}
Let $G$ be a locally profinite group which has locally finite $\ell$-cohomological dimension. Then:
\begin{propenum}
	\item \label{rslt:classifying-stack-is-l-fine} The natural projection $*/G \to *$ is $\ell$-fine.
	\item \label{rslt:classifying-stack-is-proper-for-profinite-G} If $G$ is profinite and has finite $\ell$-cohomological dimension then the map $*/G \to *$ is $\ell$-cohomologically proper.
\end{propenum}
\end{proposition}
\begin{proof}
For (i) we can replace $G$ by any compact open subgroup because the property of being $\ell$-fine is étale local on the source (see \cref{rslt:stabilities-of-ell-fine-maps}). We can thus assume that $G$ is profinite with $\cd_\ell G < \infty$. In this case the map $*/G \to *$ is covered by the map $* \to */G$ which is fdcs and admits universal $\ell$-codescent by \cref{rslt:map-from-*-to-*-mod-G-has-codescent}, so that $* \to */G$ is $\ell$-fine by definition.

It remains to prove (ii) so assume that $G$ is as in the claim. First note that the diagonal of $f\colon */G \to *$ has fiber $G$ and is thus proper, so that $f$ is 1-separated. By \cref{rslt:group-cohom-preserves-nuclear-and-colim} $f_*$ preserves all small colimits, so we can argue as in \cite[Corollary 3.12]{mod-p-stacky-6-functors} to deduce that the natural map $f_! \to f_*$ is an isomorphism, implying $\ell$-cohomological properness of $f$.
\end{proof}

We have established the fact that for nice enough locally profinite groups $G$ the classifying stack $*/G$ is $\ell$-fine, so in particular we have the full 6-functor formalism for nuclear sheaves on classfying stacks at our disposal. In order to compute shriek functors, it is very useful to know that in practice the classifying stack $*/G$ is even $\ell$-cohomologically smooth. We have already worked this out for $p$-adic sheaves in \cite[\S3.2]{mod-p-stacky-6-functors} and it works the same way $\ell$-adically. For the convenience of the reader we present the main definitions and results.

\begin{definition}
\begin{defenum}
	\item \label{def:Poincare-group} Let $G$ be a profinite group with $\cd_\ell G < \infty$. We say that $G$ is \emph{$\ell$-Poincaré of dimension $d$} if it satisfies the following conditions:
	\begin{enumerate}[(i)]
		\item $H^k(G, M)$ is finite for all $k \ge 0$ and all $G$-representations $M$ on finite $\Fld_\ell$-vector spaces.

		\item The solid $\Fld_\ell$-vector space $\Gamma(G,\Fld_{\ell,\solid}[G])$ is invertible and concentrated in cohomological degree $d$.
	\end{enumerate}

	\item \label{def:virtually-Poincare-group} Let $G$ be a locally profinite group with locally finite $\ell$-cohomological dimension. We say that $G$ is \emph{virtually $\ell$-Poincaré of dimension $d$} if there is some compact open subgroup $H \subset G$ such that $\cd_\ell H < \infty$ and $H$ is $\ell$-Poincaré of dimension $d$.
\end{defenum}
\end{definition}

\begin{examples}
Let $G$ be a locally profinite group.
\begin{exampleenum}
	\item If $G$ is an $\ell$-adic Lie group of dimension $d$ then $G$ is virtually $\ell$-Poincaré of dimension $d$. This follows from results of Lazard \cite{lazard-monster}, see \cite[Theorem 3.18]{mod-p-stacky-6-functors}.

	\item If $G$ is locally pro-$p$ then it is virtually $\ell$-Poincaré of dimension $0$. Indeed, if $G$ is pro-$p$ then $\cd_\ell G = 0$ by \cite[Corollary III.3.3.7]{cohomology-of-number-fields}, i.e. $\Gamma(G,-)$ is acyclic on $\ell$-torsion representations of $G$; it follows easily that $G$ is $\ell$-Poincaré of dimension $0$.
\end{exampleenum}
\end{examples}

In the following when we say that a locally profinite group $G$ is virtually $\ell$-Poincaré then we implicitly always mean that $G$ has locally finite $\ell$-cohomological dimension.

\begin{lemma} \label{rslt:first-condition-of-Poincare-equiv-F-l-relatively-dualizable}
Let $G$ be a profinite group with $\cd_\ell G < \infty$. Then $G$ satisfies condition (i) of \cref{def:Poincare-group} if and only if $\Fld_\ell \in \D_\et(*/G,\Fld_\ell)$ is dualizable over $*$.
\end{lemma}
\begin{proof}
We have a cartesian square
\begin{center}\begin{tikzcd}
	*/(G \cprod G) \arrow[r,"\pi_1"] \arrow[d,"\pi_2"] & */G \arrow[d,"f"]\\
	*/G \arrow[r,"f"] & *
\end{tikzcd}\end{center}
By \cref{rslt:characterization-of-f-dualizable-sheaves} the constant sheaf $\Fld_\ell \in \D_\et(*/G,\Fld_\ell)$ is $f$-dualizable if and only if the natural map $\pi_1^* f^! \Fld_\ell \to \pi_2^! \Fld_\ell$ is an isomorphism. For every compact open subgroup $H \subset G$ let $g_H\colon */(H \cprod H) \to *$ and $h_H\colon */(H \cprod H) \to */(G \cprod G)$ be the natural maps. Then the collection of functors $g_{H*} h_H^*\colon \D_\et(*/(G \cprod G),\Fld_\ell) \to \D_\et(*,\Fld_\ell)$ is conservative (this is a reformulation of the fact that a smooth $(G \cprod G)$-representation is determined by its $(H \cprod H)$-invariants for all $H$). Consequently $\Fld_\ell$ is $f$-dualizable if and only if for all $H$ the map
\begin{align}
	g_{H*} h_H^* \pi_1^* f^! \Fld_\ell \to g_{H*} h_H^* \pi_2^! \Fld_\ell \label{eq:isom-for-ula-of-classifying-stack}
\end{align}
is an isomorphism. We compute both sides in the case that $H = G$; for general $H$ one gets the same terms but with $G$ replaced by $H$. By \cref{rslt:classifying-stack-is-proper-for-profinite-G} all the appearing maps are $\ell$-cohomologically proper, so that in particular their pushforwards satisfy the projection formula and base-change. We get
\begin{align*}
	&g_{G*} h_G^* \pi_1^* f^! \Fld_\ell = f_* \pi_{1*} \pi_1^* f^! \Fld_\ell = f_* (f^! \Fld_\ell \tensor \pi_{1*} \Fld_\ell) = f_* (f^! \Fld_\ell \tensor f^* f_* \Fld_\ell) =\\&\qquad= f_* f^! \Fld_\ell \tensor f_* \Fld_\ell = (f_* \Fld_\ell)^\vee \tensor f_* \Fld_\ell,
\end{align*}
and similarly we have
\begin{align*}
	&g_{H*} h_H^* \pi_2^! \Fld_\ell = f_* \pi_{2*} \pi_2^! \Fld_\ell = f_* \IHom(\pi_{2*} \Fld_\ell, \Fld_\ell) = f_* \IHom(f^* f_* \Fld_\ell, \Fld_\ell) =\\&\qquad= \IHom(f_*\Fld_\ell, f_* \Fld_\ell).
\end{align*}
We deduce that \cref{eq:isom-for-ula-of-classifying-stack} is an isomorphism (for $H = G$) if and only if $f_* \Fld_\ell = \Gamma(G, \Fld_\ell)$ is dualizable. If we now let $H$ vary we find that $\Fld_\ell$ is $f$-dualizable if and only if $\Gamma(H, \Fld_\ell)$ is dualizable for all compact open subgroups $H \subset G$. Since every finite $G$-representation is a retract of a finite (co)limit of the generators $\Fld_\ell[G/H]$, we deduce that this condition is indeed equivalent to condition (i) of \cref{def:Poincare-group}.
\end{proof}

\begin{theorem} \label{rslt:classifying-stacks-are-smooth}
Let $G$ be a locally profinite group with locally finite $\ell$-cohomological dimension. Then the map $*/G \to *$ is $\ell$-cohomologically smooth (of pure dimension $d/2$) if and only if $G$ is virtually $\ell$-Poincaré (of dimension $d$).
\end{theorem}
\begin{proof}
By \cref{rslt:classifying-stack-is-l-fine} the map $f\colon */G \to *$ is $\ell$-fine so it makes sense to speak of $\ell$-cohomological smoothness. Since $\ell$-cohomological smoothness is étale local on the source we can assume that $G$ is profinite with $\cd_\ell G < \infty$. By \cref{rslt:first-condition-of-Poincare-equiv-F-l-relatively-dualizable} we can assume that $G$ satisfies condition (i) of \cref{def:Poincare-group} and we only need to show that then $G$ satisfies condition (ii) of that definition if and only if $f^! \Fld_\ell$ is invertible. This is a formal computation for which we refer the reader to our argument in \cite[Proposition 3.14.(ii)]{mod-p-stacky-6-functors}. 
\end{proof}

With a good understanding of the classifying stacks at hand, it now makes sense to introduce admissible representations in the following form:

\begin{definition} \label{def:admissible-representation}
Let $G$ be a locally profinite group with locally finite $\ell$-cohomological dimension. A nuclear $G$-representation $M \in \D_\nuc(\Lambda)^{BG}$ is called \emph{admissible} if it is relatively dualizable over $*$ (when viewed as a nuclear $\Lambda$-module on $*/G$).
\end{definition}

The above definition of admissible representations falls naturally out of the geometric setting, but if $\Lambda$ is not discrete then it seems to be a notion which has not been studied much in the literature yet. We do not attempt a full study of admissible nuclear representations, but we can at least state some easy formal consequences of the definition:

\begin{proposition} \label{rslt:properties-of-admissible-reps}
Let $G$ be a locally profinite virtually $\ell$-Poincaré group and let $M \in \D_\nuc(\Lambda)^{BG}$ be a nuclear $G$-representation.
\begin{propenum}
	\item Let $i_1, i_2\colon \D_\nuc(\Lambda)^{BG} \to \D_\nuc(\Lambda)^{B(G \cprod G)}$ denote the two inflation operators. Then $M$ is admissible if and only if the natural map
	\begin{align*}
		i_1(M^\vee) \tensor i_2(M) \to \IHom(i_1(M), i_2(M))
	\end{align*}
	is an isomorphism of nuclear $(G \cprod G)$-representations.

	\item If $M$ is admissible then it is reflexive, i.e. the map $M \isoto M^{\vee\vee}$ is an isomorphism.

	\item Let $H \subset G$ be an open subgroup. Then $M$ is admissible as a $G$-representation if and only if it is admissible as an $H$-representation.

	\item If $M$ is admissible then $\Gamma(H, M)$ is dualizable for every compact open subgroup $H \subset G$ with $\cd_\ell H < \infty$. If $\Lambda$ is discrete then the converse of this statement holds.

	\item Let $\Lambda \to \Lambda'$ be a map of nuclear $\Z_\ell$-algebras. If $M$ is admissible then $M \tensor_\Lambda \Lambda' \in \D_\nuc(\Lambda')^{BG}$ is admissible.

	\item Suppose that $\Lambda = \Z_\ell$ and that $M$ is bounded and $\ell$-adically complete. Then $M$ is admissible if and only if $M/\ell M \in \D_\nuc(\Lambda/\ell \Lambda)^{BG}$ is admissible.
\end{propenum}
\end{proposition}
\begin{proof}
For (i), note that $i_1$ and $i_2$ are just the pullbacks along the two projections $*/(G \cprod G) \to */G$. Using also the fact that both are $\ell$-cohomologically smooth by \cref{rslt:classifying-stacks-are-smooth} we easily reduce the claim to \cref{rslt:characterization-of-f-dualizable-sheaves}. Part (ii) is a special case of the last part of \cref{rslt:characterization-of-f-dualizable-sheaves}. For (iii), note that $*/H \to */G$ is an étale and hence $\ell$-cohomologically smooth cover, hence the claim is a special case of \cref{rslt:rel-dualizable-is-smooth-local-on-source}. For (iv), by \cref{rslt:classifying-stack-is-proper-for-profinite-G} the stack $*/H$ is $\ell$-cohomologically proper over $*$, hence the first part of the claim is a special case of \cref{rslt:rel-dualizable-stable-under-proper-pushforward}. If $\Lambda$ is discrete then the family of functors $\Gamma(H,-)\colon \D_\nuc(\Lambda)^{BG} \to \D_\nuc(\Lambda)$ is conservative, so one can argue as in the proof of \cref{rslt:first-condition-of-Poincare-equiv-F-l-relatively-dualizable} to get the second part of (iv). Claims (v) and (vi) are special cases of \cref{rslt:rel-dualizable-under-change-of-Lambda}.
\end{proof}

\bibliography{bibliography}
\addcontentsline{toc}{section}{References}

\end{document}